\documentclass[a4paper,12pt]{article}
\usepackage[top=30mm, left=17mm, right=17mm, bottom=40mm, headsep=0mm, footskip=10mm]{geometry}
\usepackage{mathrsfs,amsthm,graphicx,color,verbatim,bm,bbm,amsmath,amsfonts,amssymb,newclude,nicefrac,amsfonts,graphicx,enumerate,hyperref}
\usepackage[latin1]{inputenc}

\newtheorem{lemma}{Lemma}[section]
\newtheorem{theorem}[lemma]{Theorem}

\providecommand{\N}{{\ensuremath{\mathbbm{N}}}}
\providecommand{\R}{{\ensuremath{\mathbbm{R}}}}
\providecommand{\E}{{\ensuremath{\mathbb{E}}}}
\renewcommand{\P}{{\ensuremath{\mathbb{P}}}}
\providecommand{\1}{{\ensuremath{\mathbbm{1}}}}

\newcommand{\diffns}[1]{d#1}
\providecommand{\Cb}[1]{{\mathcal{C}_{\mathrm{b}}^{#1}}}
\newcommand{\lpn}[3]{\mathcal{L}^{#1}(#2;#3)}
\newcommand{\lpnb}[3]{L^{#1}(#2;#3)}
\newcommand{\smallsum}{\textstyle\sum}
\newcommand{\smallprod}{\textstyle\prod}

\newcommand{\deltaset}[1]{\mathbb{D}_{#1}}
\newcommand{\stochval}[1]{\big[[#1]\big]}
\newcommand{\randval}[1]{|\![#1]\!|}
\newcommand{\randvalauto}[1]{\left|\!\!\left[#1\right]\!\!\right|}
\newcommand{\nzspace}[1]{{#1}\setminus\{0\}}

\title{On the differentiability of solutions of stochastic evolution equations with respect to their initial values}

\author{Adam Andersson, Arnulf Jentzen, Ryan Kurniawan, and Timo Welti}

\begin{document}

\maketitle

\begin{abstract}

In this article we study the differentiability of solutions of parabolic semilinear stochastic evolution equations (SEEs) with respect to their initial values.
We prove that if the nonlinear drift coefficients and the nonlinear diffusion coefficients of the considered SEEs are $n$-times continuously Fr\'{e}chet differentiable, 
then the solutions of the considered SEEs are also $n$-times continuously Fr\'{e}chet differentiable with respect to their initial values.
Moreover, a key contribution of this work is to establish suitable enhanced regularity properties of the derivative processes of the considered SEE in the sense that the dominating linear operator appearing in the SEE smoothes the higher order derivative processes.
\end{abstract}

\section{Introduction}

In this article we study the differentiability of solutions of parabolic semilinear stochastic evolution equations (SEEs) with respect to their initial values. (Semilinear) SEEs have been extensively studied in the last decades by means of several different approaches; see, e.g., the monographs by Rozovski{\u\i}~\cite{r90}, Pr{\'e}v{\^o}t \& R{\"o}ckner~\cite{PrevotRoeckner2007}, and Liu \& R{\"o}ckner~\cite{WeiRoeckner2015} for results on SEEs in the context of the so-called ``variational approach" for SEEs, see, e.g., Da Prato \& Zabczyk~\cite{dz92} for results on semilinear SEEs in the context of the so-called ``semigroup approach" for SEEs, and see, e.g., Walsh~\cite{Walsh1986} for results on semilinear SEEs in the context of the so-called ``martingale measure approach".
In this paper we employ the semigroup approach to establish differentiability of solutions of parabolic semilinear SEEs with respect to their initial values.
More precisely, we prove that the smoothness of the coefficients of the considered SEEs transfers to the smoothness of the solutions of the SEEs with respect to their initial values. 
We demonstrate that if the nonlinear drift coefficients and the nonlinear diffusion coefficients of the considered SEEs are $n$-times continuously Fr\'{e}chet differentiable, 
then the solutions of the considered SEEs are also $n$-times continuously Fr\'{e}chet differentiable with respect to their initial values.
In addition, a key contribution of this work is to establish suitable enhanced regularity properties of the derivative processes of the considered SEE in the sense that the dominating linear operator appearing in the SEE smoothes the higher order derivative processes (see~\eqref{eq:intro.regularity}--\eqref{eq:condition.delta} below).
In the following theorem we summarize some of the key findings of this article.
\begin{theorem}
\label{thm:intro}
Let 
$(H,\left\|\cdot\right\|_H,\langle\cdot,\cdot\rangle_H)$ 
and 
$
  (
    U,
    \left\| \cdot \right\|_U,
    \langle\cdot,\cdot\rangle_U
  )
$
be non-trivial separable $\R$-Hilbert spaces,
let 
$n\in\N=\{1,2,\ldots\}$,
$ T \in (0,\infty) $, 
$ \eta \in \R $, 
let 
$ F \colon H \to H $ and 
$ B \colon H \to HS(U,H) $ 
be $n$-times continuously Fr\'{e}chet differentiable functions with globally bounded derivatives, 
let
$
( \Omega , \mathcal{F}, \P )
$
be a probability space with a normal filtration 
$
( \mathcal{F}_t )_{ t \in [0,T] }
$,
let
$
  ( W_t )_{ t \in [0,T] }
$
be an $\operatorname{Id}_U$-cylindrical 
$
  ( \Omega , \mathcal{F}, \P, ( \mathcal{F}_t )_{ t \in [0,T] } )
$-Wiener process,
let
$
  A \colon D(A)
  \subseteq
  H \rightarrow H
$
be a generator of a strongly continuous analytic semigroup
with 
$
  \operatorname{spectrum}( A )
  \subseteq
  \{
    z \in \mathbb{C}
    \colon
    \operatorname{Re}( z ) < \eta
  \}
$,
let
$
  (
    H_r
    ,
    \left\| \cdot \right\|_{ H_r }
    ,
    \left< \cdot , \cdot \right>_{ H_r }
  )
$,
$ r \in \R $,
be a family of interpolation spaces associated to
$
  \eta - A
$
(cf., e.g., \cite[Section~3.7]{sy02}),
and for every 
$\mathcal{F}$/$\mathcal{B}(H)$-measurable
function $ X \colon \Omega \to H $ let 
$\randval{X}$ be the set given by 
$
  \randval{X}
  =
  \big\{ Y \colon \Omega \to H 
  \colon 
  ( Y \text{ is }\mathcal{F}\text{/}\mathcal{B}(H)\text{-measurable and } \P(X=Y) = 1  )
  \big\}
$.
Then  
\begin{enumerate}[(i)]
\item
\label{item:thm.base.existence}
there exist up-to-modifications unique
$ ( \mathcal{F}_t )_{ t \in [0,T] } $/$ \mathcal{B}(H) $-predictable stochastic processes
$ X^{0,x} \colon 
$
$
[0,T] \times \Omega \to H $, 
$ x \in H $, 
which fulfill for all
$ p \in [2,\infty) $, 
$ x \in H $, 
$ t \in [0,T] $
that
$
\int^t_0
\|e^{(t-s)A} F(X^{0,x}_s)\|_H
$
$
+
\|e^{(t-s)A} B(X^{0,x}_s)\|^2_{HS(U,H)}
\, ds
< \infty
$, 
$
  \sup_{ s \in [0,T] }
  \E\big[\|
    X_s^{0,x}
  \|^p_H
  \big]
  < \infty
$, 
and
\begin{equation}
\begin{split}
  \randval{{X_t^{0,x}}}
&=
    \randvalauto{{e^{tA} x+\int_0^t
    e^{ ( t - s ) A }
      F(X_s^{0,x})
    \, ds}}
+
  \int_0^t
    e^{ ( t - s ) A }
      B(X_s^{0,x})
  \, \diffns W_s
  ,
\end{split}
\end{equation}
\item
\label{item:thm.existence.derivatives}
it holds for all 
$ p \in [2,\infty) $, 
$ t \in [0,T] $ 
that 
$
  H \ni x \mapsto \randval{X^{0,x}_t} \in \lpnb{p}{\P}{H}
$ 
is $n$-times continuously Fr\'{e}chet differentiable with globally bounded derivatives,
\item
\label{item:thm.derivative.predictable}
there exist up-to-modifications unique
$ ( \mathcal{F}_t )_{ t \in [0,T] } $/$ \mathcal{B}(H) $-predictable stochastic processes
$
  X^{k,\mathbf{u}} \colon 
$
$
  [0,T] \times \Omega \to H
$, 
$
  \mathbf{u} \in H^{k+1}
$, 
$
  k \in \{1,2,\ldots,n\}
$,
which fulfill for all 
$ p \in [2,\infty) $, 
$
  k \in \{1,2,\ldots,n\}
$,
$ x, u_1, u_2, \ldots, u_k \in H $, 
$ t \in [0,T] $ 
that 
$
  \sup_{ s \in [0,T] }
  \E\big[ \|X^{k,(x,u_1,u_2,\ldots,u_k)}_s\|^p_H \big]
  < \infty
$ 
and 
\begin{equation}
\label{eq:intro.derivatives}
\begin{split}
&
  \big(
  \tfrac{d^k}{dx^k}
  \randval{X^{0,x}_t}
  \big) (u_1,u_2,\ldots,u_k)
  =
  \randval{X^{k,(x,u_1,u_2,\ldots,u_k)}_t}
  ,
\end{split}
\end{equation}
\item
\label{item:intro.regularity}
it holds for all 
$ p \in (0,\infty) $, 
$ k \in \{1,2,\ldots,n\} $, 
$ \delta_1, \delta_2, \ldots, \delta_k \in [0,\nicefrac{1}{2}) $
with 
$
  \sum^k_{i=1}
  \delta_i
  < \nicefrac{1}{2}
$
that 
\begin{equation}
\label{eq:intro.regularity}
  \sup_{ \mathbf{u}=(u_0,u_1,\ldots,u_k) \in H \times (\nzspace{H})^k }
  \sup_{ t \in (0,T] }
  \left[
  \frac{
    t^{ (\sum^k_{i=1} \delta_i) - \nicefrac{1}{2} \, \1_{[2,\infty)}(k)
      }
    \,
    \|
      X_t^{ k, \mathbf{u} }
    \|_{
      \mathcal{L}^p( \P; H )
    }
  }{  
    \prod_{ i = 1 }^k
    \left\| u_i \right\|_{ H_{ - \delta_i } }
  }
  \right]
  <
  \infty
  ,
\end{equation}
and
\item
\label{item:Lip}
it holds for all 
$ p \in (0,\infty) $, 
$ k \in \{ 1,2, \dots, n \} $,
$
  \delta_1, \delta_2, \dots, \delta_{ k } \in [0,\nicefrac{1}{2})
$ 
with 
$
  \sum^k_{i=1}
  \delta_i
  < \nicefrac{1}{2}
$, 
$
  |F|_{\operatorname{Lip}^k(H,H)}
  < \infty
$, 
and 
$
  |B|_{\operatorname{Lip}^k(H,HS(U,H))}
  < \infty
$
that 
\begin{equation}
\label{eq:intro.Lip}
  \sup_{\substack{
    x,y \in H,
    \\
    x\neq y
  }}
  \sup_{\mathbf{u}=(u_1,u_2,\dots,u_k)\in 
  (\nzspace{H})^k}
  \sup_{ t \in (0,T] }
  \left[
  \frac{
  t^{ (\sum^k_{i=1} \delta_i) - \nicefrac{1}{2}
  }
    \,
    \|
      X_t^{ k, ( x, \mathbf{u} ) }
      -
      X_t^{ k, ( y, \mathbf{u} ) }
    \|_{
      \mathcal{L}^p( \P; H )
    }
  }{
    \| x-y \|_{ H }
    \prod^k_{ i=1 }
    \|u_i\|_{ H_{ -\delta_i } }
  }
  \right]
  < \infty.
\end{equation}
\end{enumerate}
\end{theorem}
In Theorem~\ref{thm:intro}
we denote for non-trivial
$ \R $-Banach spaces
$
( V, \left\| \cdot \right\|_V ) 
$
and 
$
( W, \left\| \cdot \right\|_W ) 
$, 
a natural number $k\in\N$, 
and a $k$-times continuously differentiable function $f\colon V \to W$
by
$ |f|_{\operatorname{Lip}^k( V, W )} $
the $k$-Lipschitz semi-norm associated to $f$
(see~\eqref{eq:Lip.def} in Subsection~\ref{sec:notation} below for details).
Theorem~\ref{thm:intro} is an immediate consequence of items~\eqref{item:lem_derivative:existence}, \eqref{item:lem_derivative:a_priori}, \eqref{item:lem_derivative:a_priori_Lip}, \eqref{item:time.derivative.smoothness}, and~\eqref{item:time.derivative.representation} of Theorem~\ref{lem:derivative_processes} below.
In Theorem~\ref{lem:derivative_processes} below we also specify explicitly
for every natural number $k\in\N$ 
the SEEs which the $k$-th derivative processes in~\eqref{eq:intro.derivatives} above are solutions of (see item~\eqref{item:lem_derivative:existence} of Theorem~\ref{lem:derivative_processes} below for details). 
Moreover, Theorem~\ref{lem:derivative_processes} below provides explicit bounds for the left hand sides of~\eqref{eq:intro.regularity} and~\eqref{eq:intro.Lip} (see items~\eqref{item:lem_derivative:a_priori} and~\eqref{item:lem_derivative:a_priori_Lip} of Theorem~\ref{lem:derivative_processes} below) 
in a more general framework than in Theorem~\ref{thm:intro} above and establishes several further regularity properties for the derivative processes in item~\eqref{item:thm.derivative.predictable} of Theorem~\ref{thm:intro}.
Next we would like to emphasize that Theorem~\ref{thm:intro} and Theorem~\ref{lem:derivative_processes}, respectively, prove finiteness of~\eqref{eq:intro.regularity} and~\eqref{eq:intro.Lip} even though the denominators in~\eqref{eq:intro.regularity} and~\eqref{eq:intro.Lip} contain rather weak norms from negative Sobolev-type spaces for the multilinear arguments of the derivative processes.
In particular, item~\eqref{item:intro.regularity} of Theorem~\ref{thm:intro}
and item~\eqref{item:lem_derivative:a_priori} of Theorem~\ref{lem:derivative_processes} below, respectively, 
reveal for every
$ p \in [1,\infty) $, 
$ k \in \{1,2,\ldots,n\} $, 
$ \delta_1, \delta_2, \ldots, \delta_k \in [0,\nicefrac{1}{2}) $, 
$ x \in H $
that the derivative processes
$
  \big(
  H^k \ni (u_1,u_2,\ldots,u_k)
  \mapsto
  \randval{X^{k,(x,u_1,u_2,\ldots,u_k)}_t} \in \lpnb{p}{\P}{H} 
  \big)
  \in L( H^{\otimes k}, \lpnb{p}{\P}{H} )
$, 
$ t \in (0,T] $, 
even take values in the continuously embedded subspace 
\begin{equation}
\label{eq:rough.space}
  L( \otimes^k_{i=1} H_{-\delta_i}, \lpnb{p}{\P}{H} )
\end{equation} 
of 
$
  L( H^{\otimes k}, \lpnb{p}{\P}{H} )
$
provided that the hypothesis 
\begin{equation}
\label{eq:condition.delta}
  \smallsum^k_{i=1} \delta_i
  < \nicefrac{1}{2}
\end{equation}
is satisfied.
Items~\eqref{item:intro.regularity}--\eqref{item:Lip} of Theorem~\ref{thm:intro} 
and items~\eqref{item:lem_derivative:a_priori} and~\eqref{item:lem_derivative:a_priori_Lip} of Theorem~\ref{lem:derivative_processes} below, respectively,  are of major importance for establishing essentially sharp probabilistically \emph{weak convergence rates} for numerical approximation processes
as the analytically weak norms for the multilinear arguments of the derivative processes (see the denominators in~\eqref{eq:intro.regularity} and~\eqref{eq:intro.Lip} above) translate in analytically weak norms for the approximation errors in the probabilistically weak error analysis which, in turn, result in essentially sharp probabilistically weak convergence rates for the numerical approximation processes
(cf., e.g., 
Theorem~2.2 in Debussche~\cite{Debussche2011},
Theorem~2.1 in Wang \& Gan~\cite{WangGan2013_Weak_convergence}, 
Theorem~1.1 in Andersson \& Larsson~\cite{AnderssonLarsson2015},
Theorem~1.1 in Br\'{e}hier~\cite{Brehier2014},
Theorem~5.1 in Br\'{e}hier \& Kopec~\cite{BrehierKopec2016},
Corollary~1 in Wang~\cite{Wang2016481},
Corollary~5.2 in Conus et al.~\cite{ConusJentzenKurniawan2014arXiv}, 
Theorem~6.1 in Kopec~\cite{Kopec2014_PhD_Thesis},
and
Corollary~8.2 in~\cite{JentzenKurniawan2015arXiv}).
In the following we briefly relate 
items~\eqref{item:thm.base.existence}--\eqref{item:Lip} of Theorem~\ref{thm:intro} and
Theorem~\ref{lem:derivative_processes} below with results from the literature. 
Item~\eqref{item:thm.base.existence} of Theorem~\ref{thm:intro} is well-known and can, e.g., be found in Theorem~7.4 in Da Prato \& Zabczyk~\cite{dz92} (cf., e.g., 
Theorem 4.3 in Brze{\'z}niak~\cite{b97b},
Theorem~7.3.5 in Da Prato \& Zabczyk~\cite{dz02b}, 
Theorem 6.2 in Van Neerven et al.~\cite{vvw08}, 
and Theorem~6.2.3 in Liu \& R\"{o}ckner~\cite{WeiRoeckner2015}).
Items~\eqref{item:thm.existence.derivatives}--\eqref{item:thm.derivative.predictable} of Theorem~\ref{thm:intro} and items~\eqref{item:lem_derivative:existence}, \eqref{item:lem_derivative:smoothness}, and~\eqref{item:lem_derivative:representation} of Theorem~\ref{lem:derivative_processes} below are generalizations and enhancements of Theorem~7.3.6 in Da Prato \& Zabczyk~\cite{dz02b}.
In particular, we allow $F$ and $B$ to grow linearly (cf.\ \eqref{eq:Cb.def} in Subsection~\ref{sec:notation} below), 
we prove continuous Fr\'{e}chet differentiability (cf.\ item~\eqref{item:thm.existence.derivatives} of Theorem~\ref{thm:intro}), 
and we develop the combinatorics (cf., e.g., Theorem~2 in Clark \& Houssineau~\cite{ClarkHoussineau2013}) to explicitly specify the SEEs to which the derivative processes of any order are solutions of (cf.\ item~\eqref{item:lem_derivative:existence} of Theorem~\ref{lem:derivative_processes} below).
Nonetheless, the main contribution of this paper is to establish that the derivative processes even take values in the space~\eqref{eq:rough.space} provided that the assumption~\eqref{eq:condition.delta} is fulfilled.

\subsection{Notation}
\label{sec:notation}

In this section we introduce some of the notation which we employ throughout this article
(cf., e.g., Section~1.1 in~\cite{AnderssonJentzenKurniawan2016arXiv}).
For two
measurable spaces
$
  ( A, \mathcal{A} )
$
and
$
  ( B, \mathcal{B} )
$
we denote by
$
  \mathcal{M}( \mathcal{A}, \mathcal{B} )
$
the set of
$ \mathcal{A} $/$ \mathcal{B} $-measurable
functions.
For a set $ A $ 
we denote by 
$ \mathcal{P}(A) $ the power set of $ A $
and we denote by 
$ \#_A \in \N_0 \cup \{\infty\} $ 
the number of elements of $ A $.
For an $ \R $-vector space $ V $
we denote by 
$ V^{[k]} \subseteq V $, $ k \in \N_0 $, 
the sets which satisfy for all $ k \in \N $ 
that $ V^{[0]} = V $ 
and 
$ V^{[k]} = V \backslash \{ 0 \} $.
For a real number $ T \in (0,\infty) $, 
a set $ \Omega $, 
and a family 
$ (\mathcal{F}_t)_{ t \in [0,T] } \subseteq \mathcal{P}(\mathcal{P}(\Omega)) $ 
of sigma-algebras on $\Omega$ 
we denote by
$
  \mathrm{Pred}( ( \mathcal{F}_t )_{ t \in [0,T] } )
$
the sigma-algebra given by
$
  \mathrm{Pred}( ( \mathcal{F}_t )_{ t \in [0,T] } )
  =
  \sigma_{ [0,T] \times \Omega }\big(
    \big\{
      ( s, t ] \times A
      \colon
      s \in [0,T), t \in (s,T], 
      A \in \mathcal{F}_s
    \big\}
    \cup
    \big\{
      \{ 0 \} \times A
      \colon
      A \in \mathcal{F}_0
    \big\}
  \big)
$
(the predictable sigma-algebra associated
to
$
  ( \mathcal{F}_t )_{ t \in [0,T] }
$).
For $ \R $-Banach spaces
$ ( V , \left\| \cdot \right\|_V ) $
and
$ ( W , \left\| \cdot \right\|_W ) $
with 
$ \#_V > 1 $
and a natural number $ n \in \N $
we denote by
$
  \left| \cdot \right|_{ \Cb{n}( V, W ) }
  \colon
  \mathcal{C}^n( V, W ) \to [0,\infty]
$
and
$
  \left\| \cdot \right\|_{ \Cb{n}( V, W ) }
  \colon
  \mathcal{C}^n( V, W ) \to [0,\infty]
$
the functions which satisfy 
for all $ f \in \mathcal{C}^n( V, W ) $
that
\begin{equation}
\label{eq:Cb.def}
\begin{split}
  \left| f \right|_{
    \Cb{n}( V, W )
  }
& =
  \sup_{
    x \in V
  }
  \left\|
    f^{ (n) }( x )
  \right\|_{
    L^{ (n) }( V, W )
  }
  ,
\qquad
  \left\| f \right\|_{
    \Cb{n}( V, W )
  }
  =
  \|f(0)\|_W
  +
  \sum_{ k = 1 }^n
  \left| f \right|_{ \Cb{k}(V,W) }
\end{split}
\end{equation}
and we denote by
$
  \Cb{n}( V, W )
$
the set given by
$
  \Cb{n}( V, W ) =
  \{ f \in \mathcal{C}^n( V, W ) \colon \left\| f \right\|_{ \Cb{n}( V, W ) } < \infty \}
$.
For $ \R $-Banach spaces
$ ( V , \left\| \cdot \right\|_V ) $
and
$ ( W , \left\| \cdot \right\|_W ) $
with
$ \#_V > 1 $
and a nonnegative integer $ n \in \N_0 $
we denote by
$
  \left| \cdot \right|_{
    \operatorname{Lip}^n( V, W )
  }
  \colon 
  \mathcal{C}^n( V, W )
  \to [0,\infty]
$
and 
$
  \left\| \cdot \right\|_{
    \operatorname{Lip}^n( V, W )
  }
  \colon 
  \mathcal{C}^n( V, W )
  \to [0,\infty]
$
the functions which satisfy
for all $ f \in \mathcal{C}^n( V, W ) $
that
\begin{equation}
\label{eq:Lip.def}
\begin{split}
  \left| f \right|_{ 
    \operatorname{Lip}^n( V, W )
  }
  &=
\begin{cases}
  \sup_{ 
    \substack{
      x, y \in V ,\,
      x \neq y
    }
  }
  \left(
  \frac{
    \left\| f( x ) - f( y ) \right\|_W
  }{
    \left\| x - y \right\|_V
  }
  \right)
&
  \colon
  n = 0
\\
  \sup_{ 
    \substack{
      x, y \in V ,\,
      x \neq y
    }
  }
  \left(
  \frac{
    \| f^{ (n) }( x ) - f^{ (n) }( y ) \|_{ L^{ (n) }( V, W ) }
  }{
    \left\| x - y \right\|_V
  }
  \right)
&
  \colon
  n \in \N
\end{cases}
  ,
\\
  \left\| f \right\|_{
    \operatorname{Lip}^n( V, W )
  }
  &
  =
  \|f(0)\|_W
  +
  \sum_{ k = 0 }^n
  \left| f \right|_{ \operatorname{Lip}^k(V,W) }
\end{split}
\end{equation}
and we denote by
$
  \operatorname{Lip}^n( V, W )
$
the set given by
$
  \operatorname{Lip}^n( V, W ) =
  \{ f \in \mathcal{C}^n( V, W ) \colon \left\| f \right\|_{ \operatorname{Lip}^n( V, W ) } < \infty \}
$.
For an $ \R $-Hilbert space
$ ( H , \left\| \cdot \right\|_H , \left< \cdot , \cdot \right>_H ) $,
real numbers $ r \in [0,1] $, $ \eta \in \R $, 
$ T \in (0,\infty) $,
and a generator of a strongly continuous analytic semigroup
$
  A \colon D(A)
  \subseteq
  H \rightarrow H
$
with 
$
  \operatorname{spectrum}( A )
  \subseteq
  \{
    z \in \mathbb{C}
    \colon
    \text{Re}( z ) < \eta
  \}
$
we denote by
$
  \chi^{ r, T }_{
    A, \eta
  }
  \in [0,\infty)
$
the real number 
given by
$
  \chi^{ r, T }_{ A, \eta }
  =
  \sup_{ t \in (0,T] }
    t^r
    \,
    \|
      ( \eta - A )^r
      e^{ t A }
    \|_{ L( H ) }
$
(cf., e.g., \cite[Lemma~11.36]{rr93}).
We denote by 
$
  \mathbbm{B} \colon (0,\infty)^2 \to (0,\infty)
$
the function which satisfies
for all $ x, y \in (0,\infty) $ that
$
  \mathbbm{B}( x, y ) = \int_0^1 t^{ (x - 1) } 
  \left( 1 - t \right)^{ (y - 1) }
  \diffns t
$ 
(Beta function).
We denote by 
$
  \mathrm{E}_{ \alpha, \beta } \colon [0,\infty) \to [0,\infty)
$,
$ \alpha, \beta \in (-\infty,1) $,
the functions which satisfy for all
$
  \alpha, \beta \in (-\infty,1)
$,
$
  x \in [0,\infty)
$
that
$
  \mathrm{E}_{ \alpha, \beta }[ x ]
  =
  1
  +
  \sum_{ n = 1 }^{ \infty }
  x^n
  \prod_{ k = 0 }^{ n - 1 }
  \mathbb{B}\big(
    1-\beta
    ,
    k(1-\beta) + 1-\alpha
  \big)
$
(generalized exponential function; cf.\ Exercise~3 in Chapter~7 in Henry~\cite{h81},
(1.0.3) in Chapter~1 in Gorenflo et al.~\cite{Gorenfloetal2014}, 
and~(16) in~\cite{AnderssonJentzenKurniawan2016arXiv}). 
For real numbers
$ T \in (0,\infty) $, 
$ \eta \in \R $, 
$ p \in [1,\infty) $, 
$ a \in [0,1) $, 
$ b \in [0, \nicefrac{ 1 }{ 2 } ) $, 
$ \lambda \in (-\infty,1) $,
an $ \R $-Hilbert space
$ 
  ( H , \left\| \cdot \right\|_H , \left< \cdot , \cdot \right>_H ) 
$, 
and a generator
$ A \colon D(A) \subseteq H \to H $
of a strongly continuous analytic semigroup with
$
  \operatorname{spectrum}( A )
  \subseteq
  \{
    z \in \mathbb{C}
    \colon
    \text{Re}( z ) < \eta
  \}
$
we denote by 
$
  \Theta_{ A, \eta, p, T }^{ a, b, \lambda }
  \colon
  [0,\infty)^2
  \to
  [0,\infty]
$
the function which satisfies
for all
$ L, \hat{L} \in [0,\infty) $
that
{\small
\begin{equation}
\begin{split}
&
  \Theta_{ A , \eta, p, T }^{ a, b, \lambda }( L, \hat{L} )
  =
\\ &
\begin{cases}
  \sqrt{2}
  \,
  \bigg|
    E_{
      2 \lambda, \max\{ a, 2 b \}
    }\bigg[
      \Big|
        \tfrac{
        \chi_{ A , \eta }^{ a , T }\,
        L\,
          \sqrt{2}\,
          T^{ (1 - a) }
        }
        {
          \sqrt{1-a}
        }
        +
        \chi_{A,\eta}^{b,T}\,
        \hat{L}\,
        \sqrt{ p \, ( p - 1 ) \, T^{ (1 - 2b) } }
      \Big|^2
    \bigg]
  \bigg|^{1/2}
&
  \colon
  ( \lambda, \hat{L} )
  \in ( -\infty, \frac{ 1 }{ 2 } ) \times (0,\infty)
\\[1ex]
  E_{\lambda,a}\!\left[
    \chi_{ A , \eta }^{ a , T }\,
    L\,
    T^{ (1 - a) }
  \right]
&
  \colon
  \hat{L} = 0
\\[1ex]
  \infty
&
  \colon
  \text{otherwise}
\end{cases}
  .
\end{split}
\end{equation}
}We denote by  
$ \Pi_k, \Pi^*_k \in 
  \mathcal{P}\big(\mathcal{P}\big(
    \mathcal{P}( \N )
  \big)\big)
$, 
$ k \in \N_0 $,
the sets which satisfy for all $ k \in \N $
that
$ \Pi_0 = \Pi_0^{ * } = \emptyset $,
$
  \Pi_k^{ * } =
  \Pi_k \backslash
  \big\{ 
    \{ \{ 1, 2, \dots, k \} \}
  \big\}
$,
and
\begin{equation}
  \Pi_k =
  \big\{
    A \subseteq \mathcal{P}( \N )
    \colon
    \left[
      \emptyset \notin A
    \right]
    \wedge
    \left[
      \cup_{ a \in A }
      a
      =
      \left\{ 1, 2, \dots, k \right\}
    \right]
    \wedge
    \left[
      \forall \, a, b \in A \colon
      \left(
        a \neq b
        \Rightarrow
        a \cap b = \emptyset
      \right)
    \right]
  \big\}
\end{equation}
(cf., e.g., \cite[Theorem~2]{ClarkHoussineau2013}).
Observe, for example, that 
$
  \Pi_0 = \emptyset
$,
$
  \Pi_1 = \big\{ \{ \{ 1 \} \} \big\}
$,
$
  \Pi_2 = \big\{ \{ \{ 1, 2 \} \} , \{ \{ 1 \} , 
$
$
  \{ 2 \} \} \big\} 
$,
and
$
  \Pi_3 = 
  \big\{ 
    \{ \{ 1, 2, 3 \} \} , 
    \{ \{ 1, 2 \}, \{ 3 \} \} ,
    \{ \{ 1, 3 \} , \{ 2 \} \} ,
    \{ \{ 1 \} , \{ 2, 3 \} \} ,
    \{ \{ 1 \} , 
$
$
    \{ 2 \}, \{ 3 \} \} 
  \big\} 
$
and note that for every
$ k \in \N $
it holds that $ \Pi_k $ is the set of all partitions of
$ \{ 1, 2, \dots, k \} $.
For a natural number $ k \in \N $
and a set $ \varpi \in \Pi_k $
we denote by
$
  I^\varpi_1 , I^\varpi_2, \dots , I^\varpi_{ \#_\varpi }
  \in
  \varpi
$
the sets which satisfy that 
$
  \min\!\big( I^\varpi_1 \big) < 
  \min\!\big( I_2^\varpi \big) < \dots < 
  \min\!\big( 
    I_{ \#_\varpi }^{ \varpi } 
  \big)
$.
For 
a natural number
$ k \in \N $,
a set
$ \varpi \in \Pi_k $,
and a natural number
$ i \in \{ 1, 2, \dots, \#_\varpi \} $
we denote by
$
  I_{ i, 1 }^\varpi ,
  I_{ i, 2 }^\varpi ,
  \dots,
  I_{ i, \#_{ I_i^{ \varpi } } }^\varpi
  \in
  I_i^{ \varpi }
$
the natural numbers which satisfy that
$
  I_{ i, 1 }^\varpi < I_{ i, 2 }^\varpi < \dots < I_{ i, \#_{ I_i^{ \varpi } } }^\varpi
$.
For a measure space $ ( \Omega , \mathcal{F}, \mu ) $,
a measurable space $ ( S , \mathcal{S} ) $,
a set $ R $, 
and a function
$ f \colon \Omega \to R $
we denote by
$
   \left[ f \right]_{
     \mu, \mathcal{S}
   }
$
the set given by
\begin{equation}
   \left[ f \right]_{
     \mu, \mathcal{S}
   }
   =
   \left\{
     g \in \mathcal{M}( \mathcal{F}, \mathcal{S} )
     \colon
     (
     \exists \, A \in \mathcal{F} \colon
     \mu(A) = 0 
     \text{ and }
     \{ \omega \in \Omega \colon f(\omega) \neq g(\omega) \}
     \subseteq A
     )
   \right\}
   .
\end{equation}

\section{Stochastic evolution equations with smooth coefficients}
\label{sec:SEE_smooth_coefficients}

\subsection{Setting}
\label{sec:global_setting}

Let
$ T \in (0,\infty) $,
$ \eta \in \R $,
let
$
  \left(
    H,
    \left\| \cdot \right\|_H,
    \left< \cdot, \cdot \right>_H
  \right)
$
and
$
  \left(
    U,
    \left\| \cdot \right\|_U ,
    \left< \cdot, \cdot \right>_U
  \right)
$
be separable $ \R $-Hilbert spaces
with $ \#_H > 1 $,
let
$
  ( \Omega , \mathcal{F}, \P )
$
be a probability space with a normal filtration 
$
  ( \mathcal{F}_t )_{ t \in [0,T] }
$,
let
$
  ( W_t )_{ t \in [0,T] }
$
be an $ \operatorname{Id}_U $-cylindrical 
$ ( \Omega, \mathcal{F}, \P, ( \mathcal{F}_t )_{ t \in [0,T] } ) $-Wiener process,
let
$
  A \colon D(A)
  \subseteq
  H \rightarrow H
$
be a generator of a strongly continuous analytic semigroup
with 
$
  \operatorname{spectrum}( A )
  \subseteq
  \{
    z \in \mathbb{C}
    \colon
    \text{Re}( z ) < \eta
  \}
$,
let
$
  (
    H_r
    ,
    \left\| \cdot \right\|_{ H_r }
    ,
    \left< \cdot , \cdot \right>_{ H_r }
  )
$,
$ r \in \R $,
be a family of interpolation spaces associated to
$
  \eta - A
$,
for every $ k \in \N $,
$ \varpi \in \Pi_k $,
$
  i \in \{ 1, 2, \dots, \#_\varpi \} 
$
let 
$
  [ \cdot ]_i^\varpi
  \colon
  H^{ k + 1 }
  \to 
  H^{ 
    \#_{I_i^\varpi} + 1
  }
$
be the function which satisfies for all 
$
  \mathbf{u} = (u_0, u_1, \dots, u_k)
  \in 
  H^{ k + 1 }
$
that
$
  [ \mathbf{u} ]_i^\varpi
  = ( u_0, u_{ I_{ i, 1 }^\varpi } , u_{ I_{ i, 2 }^\varpi } , \dots , u_{ I_{ i, \#_{I_i^\varpi} }^\varpi } )
$, 
let 
$
  \stochval{\cdot}
  \colon
    \mathcal{M}(
      \mathrm{Pred}( ( \mathcal{F}_t )_{ t \in [0,T] } )
      ,
      \mathcal{B}( H )
    )
  \to
  \mathcal{P}\big(
    \mathcal{M}(
      \mathrm{Pred}( ( \mathcal{F}_t )_{ t \in [0,T] } )
      ,
$
$
      \mathcal{B}( H )
    )
  \big)
$ 
be the function which satisfies for all 
$ 
  X \in
    \mathcal{M}(
      \mathrm{Pred}( ( \mathcal{F}_t )_{ t \in [0,T] } )
      ,
      \mathcal{B}( H )
    ) 
$ 
that 
$
  \stochval{X}
  =
    \big\{
      Y \in
      \mathcal{M}(
        \mathrm{Pred}( ( \mathcal{F}_t )_{ t \in [0,T] } )
        ,
        \mathcal{B}( H )
      )
      \colon
$
$
      \inf\nolimits_{ t \in [0,T] }
      \P\big(
        Y_t = X_t
      \big)
      = 1
    \big\}
$, 
for every $ p \in (0,\infty) $
let
$ \mathcal{L}^p $ 
and 
$ \mathbb{L}^p $
be the sets given by
$
  \mathcal{L}^p
  =
  \big\{
    X \in
    \mathcal{M}(
      \mathrm{Pred}( ( \mathcal{F}_t )_{ t \in [0,T] } )
      ,
      \mathcal{B}( H )
    )
    \colon
    \sup_{ t \in [0,T] }
    \| X_t \|_{
      \mathcal{L}^p( \P ; H )
    }
    < \infty
  \big\}
$
and 
$
  \mathbb{L}^p
=
  \big\{
    \stochval{X}
    \colon
    X \in \mathcal{L}^p
  \big\}
$  
and let 
$
  \left\|\cdot\right\|_{
    \mathbb{L}^p
  }
  \colon
  \mathbb{L}^p 
  \to [0,\infty)
$ 
be the function which satisfies for all 
$ X \in \mathcal{L}^p $ 
that 
$
  \left\|\stochval{X}\right\|_{
    \mathbb{L}^p
  } 
  =
    \sup_{ t \in [0,T] }
    \| X_t \|_{
      \mathcal{L}^p( \P ; H )
    }
$, 
and for every separable $ \R $-Banach space $ ( V, \left\| \cdot \right\|_V ) $
and every 
$ a \in \R $, $ b \in (a,\infty) $, $ A \in \mathcal{B}( \R ) $,
$
  X \in \mathcal{M}( \mathcal{B}( A ) \otimes \mathcal{F} , \mathcal{B}( V ) )
$ 
with  
$ (a,b) \subseteq A $
let 
$
  \int_a^b X_s \, {\bf ds}
  \in
  \{
    [Y]_{ \P, \mathcal{B}(V) }
    \colon
    Y \in \mathcal{M}( \mathcal{F}, \mathcal{B}(V) )
  \}
$
be the set given by
$
  \int_a^b X_s \, {\bf ds}
  =
  \big[
    \int_a^b \mathbbm{1}_{ \{ \int_a^b \| X_u \|_V \, du < \infty \} } X_s \, ds
  \big]_{
    \P , \mathcal{B}( V )
  }
$.

\subsection{Differentiability with respect to the initial values}
\label{sec:differentiability_initial_value}
\begin{theorem}[Differentiability with respect to the initial value]
\label{lem:derivative_processes}
Assume the setting in Section~\ref{sec:global_setting},
let $ n \in \N $,
$ F \in \Cb{n}(H , H) $,
$ B \in \Cb{n}(H , HS(U, H)) $, 
$ \alpha \in [ 0, 1 ) $, 
$
\beta \in [ 0, \nicefrac{ 1 }{ 2 } )
$,
and for every 
$ k \in \N $, 
$
  \bm{\delta}
  =
  (\delta_1, \delta_2, \dots, \delta_k) 
  \in \R^k
$,
$
  J \in \mathcal{P}(\R)
$
let
$
  \iota^{ \bm{\delta} }_J \in \R
$
be the real number given by
$
  \iota^{ \bm{\delta} }_J 
  =
    \sum_{ i \in J \cap \{1,2,\ldots,k\} }
    \delta_i
    -
    \mathbbm{1}_{ [2,\infty) }( \#_{ J \cap \{1,2,\ldots,k\} } ) \,
    \min\{ 1-\alpha, \nicefrac{1}{2} - \beta \}
$.
Then
\begin{enumerate}[(i)]
\item
\label{item:lem_derivative:existence}
there exist up-to-modifications unique
$ ( \mathcal{F}_t )_{ t \in [0,T] } $/$ \mathcal{B}(H) $-predictable stochastic processes
$
  X^{k,\mathbf{u}} \colon 
$
$
  [0,T] \times \Omega \to H
$, 
$
  \mathbf{u} \in H^{k+1}
$, 
$
  k \in \{ 0, 1, \dots, n \}
$,
which fulfill for all 
$
  k \in \{ 0, 1, \dots, n \}
$, 
$ p \in [2,\infty) $, 
$
  \mathbf{u} = ( u_0, u_1, \ldots, u_k ) \in
    H^{k+1}
$, 
$ t \in [0,T] $ 
that 
$
  \sup_{ s \in [0,T] }
  \E[ \|X^{k,\mathbf{u}}_s\|^p_H ]
  < \infty
$ 
and 
\begin{equation}
\label{eq:lem_SEE2}
\begin{split}
& \!\!\!\!\!\!\!\!\!\!\!
  [
  X_t^{k,\mathbf{u}}
  -
  e^{tA}
  \, \mathbbm{1}_{ \{ 0, 1 \} }(k) \, u_k  
  ]_{ \P,\mathcal{B}(H) }
\\ & \!\!\!\!\!\!\!\!\!\!\!
  =
  \int_0^t
    e^{ ( t - s ) A }
    \bigg[
      \mathbbm{1}_{ \{ 0 \} }(k)
      \,
      F(X_s^{0,u_0})
      +
      \sum_{ \varpi\in \Pi_k }
      F^{ ( \#_\varpi ) }( X_s^{ 0, u_0 } )
      \big(
        X_s^{ \#_{I^\varpi_1}, [ \mathbf{u} ]_1^{ \varpi } }
        ,
        X_s^{ \#_{I^\varpi_2}, [ \mathbf{u} ]_2^{ \varpi } }
        ,
        \dots
        ,
        X_s^{ \#_{I^\varpi_{\#_\varpi}}, [\mathbf{u} ]_{ \#_\varpi }^{ \varpi } }
      \big)
    \bigg]
  \,{\bf ds}
\\ & \!\!\!\!\!\!\!\!\!\!\!
  +
  \int_0^t
    e^{ ( t - s ) A }
    \bigg[
      \mathbbm{1}_{ \{ 0 \} }(k)
      \,
     B(X_s^{0,u_0})
      +
      \sum_{ \varpi\in \Pi_k }
      B^{ ( \#_\varpi ) }( X_s^{ 0, u_0 } )
      \big(
        X_s^{ \#_{I^\varpi_1}, [ \mathbf{u} ]_1^{ \varpi } }
        ,
        X_s^{ \#_{I^\varpi_2}, [ \mathbf{u} ]_2^{ \varpi } }
        ,
        \dots
        ,
        X_s^{ \#_{I^\varpi_{\#_\varpi}}, [\mathbf{u} ]_{ \#_\varpi }^{ \varpi } }
      \big)
    \bigg]
  \, \diffns W_s
  ,
\end{split}
\end{equation}
\item
\label{item:lem_derivative:a_priori}
for all
$ k \in \{ 1, 2, \dots, n \} $,
$ p \in [2,\infty) $, 
$
  \boldsymbol{\delta}=(\delta_1, \delta_2, \dots, \delta_{ k } ) \in [0,\nicefrac{1}{2})^k
$ 
with 
$
  \sum^k_{i=1}
  \delta_i
  < \nicefrac{1}{2}
$
it holds that
\begin{align}
\label{eq:derivative.apriori}
&
  \sup_{
    \mathbf{u} =
    ( u_0, u_1, \dots, u_k ) 
    \in
    (\times^k_{ i=0 }
    H^{[i]})
  }
  \sup_{ t \in (0,T] }
  \left[
  \frac{
    t^{ 
      \iota^{ \boldsymbol{\delta} }_\N
    }
    \,
    \|
      X_t^{ k, \mathbf{u} }
    \|_{
      \mathcal{L}^p( \P; H )
    }
  }{  
    \prod_{ i = 1 }^k
    \left\| u_i \right\|_{ H_{ - \delta_i } }
  }
  \right]
\nonumber
\\ &
\nonumber
\leq
  \Theta_{ A, \eta, p, T }^{
    \alpha , 
    \beta ,
      \iota^{ \boldsymbol{\delta} }_\N
  }( |F|_{ \Cb{1}( H, H_{-\alpha} ) } 
  , 
  |B|_{ \Cb{1}( H, HS( U, H_{-\beta} ) ) } )
  \Bigg[
      \chi_{ A, \eta }^{ \delta_1 , T }\,
      \mathbbm{1}_{
        \{ 1 \} 
      }( k )
\\&\quad
      +
      \max\{T^k,1\}
  \bigg[
    \chi_{ A, \eta }^{ \alpha , T } \,
    \mathbbm{B}\big(
      1 - \alpha
      ,
      1 - \smallsum_{ i = 1 }^{ k } \delta_i
    \big)
    \| F \|_{ 
      \Cb{ k }( H, H_{-\alpha} ) 
    }
\\&\quad\nonumber+
        \chi_{ A, \eta }^{ \beta , T } \,
        \sqrt{
        \tfrac{p \, ( p - 1 )}{2}
        \,
        \mathbbm{B}\big(
          1 - 2\beta
          ,
          1 
          - 
          2 \smallsum_{ i = 1 }^{ k } \delta_i
        \big)
      } \,
    \| B \|_{ 
      \Cb{ k }( H, HS( U, H_{-\beta} ) ) 
    }
  \bigg]
\\ & \quad
\nonumber
  \cdot
    \sum_{ \varpi \in \Pi_k^{ * } }
  \prod_{ I \in \varpi }
  \sup_{ \mathbf{u} = ( u_i )_{ i \in I \cup \{0\} } \in (\times_{ i \in I \cup \{0\} } H^{[i]}) }
      \sup\limits_{ t \in (0,T] }
      \bigg[
      \displaystyle
      \frac{
        t^{
          \iota^{ \boldsymbol{\delta} }_I
        } \,
        \|
          X_t^{ \#_I, \mathbf{u} }
        \|_{
          \mathcal{L}^{ p \, \#_\varpi }( \P ; H )
        }
      }{
        \prod_{ i \in I }
        \| u_i \|_{ H_{ - \delta_i } }
      }
      \bigg]
  \Bigg]
  < \infty
  ,
\end{align}
\item
\label{item:thm_derivative}
for all
$ k \in \{ 1, 2, \dots, n \} $,
$ p \in [2,\infty) $, 
$ x \in H $
it holds that
$
  \big(
    H^k
    \ni 
    \mathbf{u}
    \mapsto 
    \stochval{X^{ k, (x, \mathbf{u}) }}
    \in
    \mathbb{L}^p
  \big)
  \in 
  L^{(k)}(
    H
    ,
    \mathbb{L}^p
  )
$, 
\item
\label{item:lem_derivative:a_priori_Lip}
for all
$ k \in \{ 1,2, \dots, n \} $,
$ p \in [2,\infty) $, 
$
  \boldsymbol{\delta}=(\delta_1, \delta_2, \dots, \delta_{ k } ) \in [0,\nicefrac{1}{2})^k
$ 
with 
$
  \sum^k_{i=1}
  \delta_i
  < \nicefrac{1}{2}
$, 
$
  |F|_{\operatorname{Lip}^k(H,H_{-\alpha})}
  < \infty
$,  
and 
$
  |B|_{\operatorname{Lip}^k(H,HS(U,H_{-\beta}))}
  < \infty
$ 
it holds that
\allowdisplaybreaks
\begin{align}
\label{eq:a_priori_Lip}
& \nonumber
  \sup_{\substack{
    x,y \in H,
    \\
    x\neq y
  }}
  \sup_{\mathbf{u}=(u_1,u_2,\dots,u_k)\in 
  (\nzspace{H})^k}
  \sup_{ t \in (0,T] }
  \frac{
  t^{ \iota^{(\boldsymbol{\delta},0)}_\N
  }
    \,
    \|
      X_t^{ k, ( x, \mathbf{u} ) }
      -
      X_t^{ k, ( y, \mathbf{u} ) }
    \|_{
      \mathcal{L}^p( \P; H )
    }
  }{
    \| x-y \|_{ H }
    \prod^k_{ i=1 }
    \|u_i\|_{ H_{ -\delta_i } }
  }
\\&\nonumber\leq
    \max\{T^k,1\}
  \,
  \Theta_{ A, \eta, p, T }^{ 
    \alpha, \beta, 
    \iota^{(\boldsymbol{\delta},0)}_\N
  }\!\big(
    | F |_{ \Cb{1}( H , H_{-\alpha} ) }
    ,
    | B |_{ \Cb{1}( H , HS( U , H_{-\beta} ) ) }
  \big)
\\&\nonumber\quad\cdot
    \Bigg(
    \chi^{ 0, T }_{ A, \eta }\,
    \Theta^{ \alpha,\beta,0 }_{ A, \eta, p, T }
    \big(
      | F |_{ \Cb{1}( H , H_{-\alpha} ) }
      ,
      | B |_{ \Cb{1}( H , HS( U , H_{-\beta} ) ) }
    \big)
\\&\nonumber\quad\cdot
  \sum_{
    \varpi \in \Pi_k
  }
  \prod_{ I\in\varpi }
    \sup_{\mathbf{u}=(u_i)_{i\in I \cup \{0\}}\in(\times_{i\in I \cup \{0\}}H^{[i]})}
    \sup_{ t\in (0,T] }
    \Bigg[
    \frac{
      t^{ \iota^{ \boldsymbol{\delta} }_I }\,
      \|
        X_t^{ \#_I, \mathbf{u} }
      \|_{ \lpn{p(\#_\varpi+1)}{\P}{H} }
    }{
      \prod_{ i\in I }
      \| u_i \|_{ H_{ -\delta_i } }
    }
    \Bigg]
\\&\quad+
    \sum_{ \varpi \in \Pi^*_k }
  \sum_{ I \in \varpi }
              \sup_{\substack{
                x,y \in H,
                \\
                x\neq y
              }}
    \sup_{\mathbf{u}=(u_i)_{i\in I}\in (\nzspace{H})^{\#_I}}
            \sup_{ t \in (0,T] }
          \Bigg[
            \frac{
              t^{
                \iota_{ I \cup \{k+1\} }^{ (\boldsymbol{\delta},0) } 
              }
              \|
      X_t^{ \#_I, ( x, \mathbf{u} ) }
      -
      X_t^{ \#_I, ( y, \mathbf{u} ) }
              \|_{
                \mathcal{L}^{ p \#_\varpi }( \P ; H )
              }
            }{
              \| x-y \|_{ H }
              \prod_{i\in I}
              \| u_i \|_{ H_{ - \delta_i } }
            }
          \Bigg]
\\&\quad\nonumber\cdot
  \prod_{ J\in\varpi\setminus\{I\} }
    \sup_{\mathbf{u}=(u_i)_{i\in J \cup \{0\}}\in(\times_{i\in J \cup \{0\}}H^{[i]})}
    \sup_{ t\in (0,T] }
    \Bigg[
    \frac{
      t^{ \iota^{ \boldsymbol{\delta} }_J }\,
      \|
        X_t^{ \#_J, \mathbf{u} }
      \|_{ \lpn{p\#_\varpi}{\P}{H} }
    }{
      \prod_{ i\in J }
      \| u_i \|_{ H_{ -\delta_i } }
    }
    \Bigg]
    \Bigg)
\\&\nonumber\quad\cdot
    \bigg[
      \chi^{ \alpha, T }_{ A, \eta }\,
      \mathbb{B}\big(
        1-\alpha
        ,
        1-\smallsum^k_{ i=1 }\delta_i
      \big) \,
       \|F\|_{ \operatorname{Lip}^k(H,H_{-\alpha}) }
\\&\nonumber\quad+
      \chi^{ \beta, T }_{ A, \eta }\,
      \sqrt{
        \tfrac{p\,(p-1)}{2}
      \,
        \mathbb{B}\big(
          1-2\beta
          ,
          1
          -
          2\smallsum^k_{i=1}\delta_i
        \big)
      }\,
      \|B\|_{ \operatorname{Lip}^k(H,HS(U,H_{-\beta})) }
    \bigg]
    < \infty,
\end{align}
\item 
\label{item:thm_derivative.continuous}
for all
$ k \in \{ 1, 2, \dots, n \} $,
$ p \in [2,\infty) $
it holds that 
$
  \big(
  H
  \ni 
  x
  \mapsto 
  \big[
    H^k
    \ni 
    \mathbf{u}
    \mapsto 
    \stochval{X^{ k, (x, \mathbf{u}) }}
    \in
    \mathbb{L}^p
  \big]
  \in 
  L^{(k)}(
    H
    ,
    \mathbb{L}^p
  )
  \big)
  \in \mathcal{C}(
    H
    ,
  L^{(k)}(
  H
  ,
  \mathbb{L}^p
  )  
  )
$,
\item
\label{item:lem_derivative:frechet}
for all
$
  k \in \{ 1, 2, \dots, n \}
$,
$ p \in [2,\infty) $, 
$
  x \in H
$  
it holds that
\begin{equation}
\!\!\!
\begin{cases}
\limsup\limits_{ \nzspace{H} \ni u_k \rightarrow 0 }
\sup\limits_{t\in[0,T]}
\tfrac{
	\|
	X^{ 0, x+u_k }_t - X^{ 0, x }_t - X^{ 1, (x,u_k) }_t
	\|_{ \lpn{p}{\P}{H} }
}{
\| u_k \|_H
}
=
0
&\!\!\!\!\colon k=1
\\
\begin{split}
&
\limsup_{ \nzspace{H} \ni u_k \rightarrow 0 }
\sup_{ \mathbf{u}=( u_1, u_2, \dots, u_{k-1} ) \in (\nzspace{H})^{k-1} }
\sup_{t\in[0,T]}
\tfrac{
	\|
	X^{ k-1, (x+u_k,\mathbf{u}) }_t - X^{ k-1, (x,\mathbf{u}) }_t - X^{ k, (x,\mathbf{u},u_k) }_t
	\|_{ \lpn{p}{\P}{H} }
}{
\prod^k_{ i=1 } \| u_i \|_H
}
=
0
\end{split}
&\!\!\!\!\colon k>1
\end{cases},
\end{equation}
\item
\label{item:lem_derivative:smoothness}
for all 
$ p \in [2,\infty) $
it holds that
$
  \big(
  H \ni x \mapsto 
  \stochval{X^{ 0,x }}
  \in
  \mathbb{L}^p
  \big)
  \in
  \Cb{n}( H, \mathbb{L}^p )
$,
\item
\label{item:lem_derivative:representation}
for all 
$
  k \in \{ 1, 2, \dots, n \}
$, 
$ p \in [2,\infty) $, 
$ x, u_1, u_2, \ldots, u_k \in H $ 
it holds that 
\begin{equation}
\begin{split}
  \big(
  \tfrac{d^k}{dx^k}
  \stochval{X^{0,x}}
  \big) 
  (u_1,u_2,\ldots,u_k)
&=
  \big(
  H \ni y \mapsto 
  \stochval{X^{0,y}}
  \in
  \mathbb{L}^p
  \big)^{(k)}(x) (u_1,u_2,\ldots,u_k)
\\&=
  \stochval{X^{k,(x,u_1,u_2,\ldots,u_k)}}
  ,
\end{split}
\end{equation}
\item
\label{item:time.derivative.smoothness}
for all $ p \in [2,\infty) $, $ t \in [0,T] $ it holds that 
$
  \big(
    H \ni x \mapsto
    [X^{0,x}_t]_{\P,\mathcal{B}(H)} \in \lpnb{p}{\P}{H}
  \big) \in \Cb{n}(H,\lpnb{p}{\P}{H})
$, 
and 
\item
\label{item:time.derivative.representation}
for all 
$
  k \in \{ 1, 2, \dots, n \}
$, 
$ p \in [2,\infty) $, 
$ x, u_1, u_2, \ldots, u_k \in H $, 
$ t \in [0,T] $
it holds that 
\begin{equation}
\begin{split}
&
  \big(
  \tfrac{d^k}{dx^k}
  [X^{0,x}_t]_{\P,\mathcal{B}(H)}
  \big) 
  (u_1,u_2,\ldots,u_k)
\\&=
  \big(
  H \ni y \mapsto 
  [X^{0,y}_t]_{\P,\mathcal{B}(H)}
  \in
  \lpnb{p}{\P}{H}
  \big)^{(k)}(x) (u_1,u_2,\ldots,u_k)
=
  [X^{k,(x,u_1,u_2,\ldots,u_k)}_t]_{\P,\mathcal{B}(H)}
  .
\end{split}
\end{equation}
\end{enumerate}
\end{theorem}

\begin{proof}
Throughout this proof 
let 
$
  r_0, r_1 \in [ 0, 1 )
$
be the real numbers given by 
$ r_0=\alpha $ and $ r_1=\beta $, 
let $ \mathbf{0}_k \in \R^k $, $ k \in \N $, 
be the vectors which satisfy for all $ k \in \N $ that 
$
  \mathbf{0}_k = (0,0,\ldots,0)
$, 
let 
$ ( V_{l,r}, \left\| \cdot \right\|_{ V_{l,r} }, \left< \cdot, \cdot \right>_{ V_{l,r} } ) $,
$ l \in \{ 0, 1 \} $, $ r \in [0,\infty) $, 
be the $ \R $-Hilbert spaces which satisfy for all $ r \in [0,\infty) $ that 
\begin{equation}
  ( V_{0,r}, \left\| \cdot \right\|_{ V_{0,r} }, \left< \cdot, \cdot \right>_{ V_{0,r} } )
  =
  ( H_{-r}, \left\| \cdot \right\|_{ H_{-r} }, \left< \cdot, \cdot \right>_{ H_{-r} } )
\end{equation}
and 
\begin{equation}
  ( V_{1,r}, \left\| \cdot \right\|_{ V_{1,r} }, \left< \cdot, \cdot \right>_{ V_{1,r} } )
  =
  ( HS(U,H_{-r}), 
  \left\| \cdot \right\|_{ HS(U,H_{-r}) }, 
  \left< \cdot, \cdot \right>_{ HS(U,H_{-r}) } ),
\end{equation}
let 
$
  G_l \colon H \to V_{l,0}
$,
$ l \in \{ 0, 1 \} $,
be the functions given by 
$
  G_0 = F
$
and 
$
  G_1 = B
$, 
let 
$ \lfloor \cdot \rfloor \colon \R \to \R $
and 
$ \lceil \cdot \rceil \colon \R \to \R $
be the functions which satisfy for all
$ t \in \R $
that
\begin{equation}
\lfloor t \rfloor =
\max\!\left(
( -\infty, t ]
\cap
\{ 0, 1, - 1, 2, - 2, \dots \}
\right)
\end{equation}
and 
\begin{equation}
\lceil t \rceil =
\min\!\left(
[ t, \infty )
\cap
\{ 0, 1 , - 1 , 2 , - 2 , \dots \}
\right),
\end{equation} 
let 
$
  \theta^m_l 
  \colon 
    H^{ m + 1 }
  \to 
    H^{ m }
$, 
$ m \in \N $, 
$ l \in \{0,1\} $, 
be the functions which satisfy
for all 
$ l \in \{0,1\} $, 
$ m \in \N $,
$ \mathbf{u} = ( u_0, u_1, \dots, u_m ) \in H^{ m + 1 } $
that
\begin{equation}
  \theta^m_l(\mathbf{u})
  =
\begin{cases}
  u_0 + l u_1 & \colon m=1
  \\
  (u_0 + l u_m, u_1, u_2, \dots, u_{ m - 1 }) & \colon m > 1
\end{cases}
,
\end{equation}
and let $ \deltaset{k} \in \mathcal{P}(\R^k) $, 
$ k \in \N $, 
be the sets which satisfy for all $ k \in \N $ that 
\begin{equation}
  \deltaset{k}=
  \big\{ 
    (\delta_1,\delta_2,\ldots,\delta_k) \in [0,\nicefrac{1}{2})^k \colon 
    \smallsum^k_{i=1} \delta_i < \nicefrac{1}{2}
  \big\}.
\end{equation}
Next we claim that for every 
$ k \in \{1,2,\ldots,n\} $ 
there exist up-to-modifications unique
$ ( \mathcal{F}_t )_{ t \in [0,T] } $/$\mathcal{B}(H)$-predictable stochastic processes
$
  X^{l,\mathbf{u}} \colon [0,T] \times \Omega \to H
$, 
$
  \mathbf{u} \in H^{l+1}
$, 
$
  l \in \{ 0, 1, \dots, k \}
$,
which fulfill for all 
$
  l \in \{ 0, 1, \dots, k \}
$, 
$ p \in [2,\infty) $, 
$
  \mathbf{u} = ( u_0, u_1, \ldots, u_l ) \in
    H^{l+1}
$, 
$ t \in [0,T] $ 
that 
$
  \sup_{ s \in [0,T] }
  \E[ \|X^{l,\mathbf{u}}_s\|^p_H ]
  < \infty
$ 
and 
\begin{equation}
\label{eq:SEE.3}
\begin{split}
&
  [
  X_t^{l,\mathbf{u}}
  -
  e^{tA}
  \, \mathbbm{1}_{ \{ 0, 1 \} }(l) \, u_l  
  ]_{ \P,\mathcal{B}(H) }
\\ &
  =
  \int_0^t
    e^{ ( t - s ) A }
    \bigg[
      \mathbbm{1}_{ \{ 0 \} }(l)
      \,
      F(X_s^{0,u_0})
      +
      \sum_{ \varpi\in \Pi_l }
      F^{ ( \#_\varpi ) }( X_s^{ 0, u_0 } )
      \big(
       X_s^{ \#_{I^\varpi_1}, [ \mathbf{u} ]_1^{ \varpi } }
        ,
        X_s^{ \#_{I^\varpi_2}, [ \mathbf{u} ]_2^{ \varpi } }
        ,
        \dots
        ,
        X_s^{ \#_{I^\varpi_{\#_\varpi}}, [\mathbf{u} ]_{ \#_\varpi }^{ \varpi } }
      \big)
    \bigg]
  \,{\bf ds}
\\ &
  +
  \int_0^t
    e^{ ( t - s ) A }
    \bigg[
      \mathbbm{1}_{ \{ 0 \} }(l)
      \,
      B(X_s^{0,u_0})
      +
      \sum_{ \varpi\in \Pi_l }
      B^{ ( \#_\varpi ) }( X_s^{ 0, u_0 } )
      \big(
       X_s^{ \#_{I^\varpi_1}, [ \mathbf{u} ]_1^{ \varpi } }
        ,
        X_s^{ \#_{I^\varpi_2}, [ \mathbf{u} ]_2^{ \varpi } }
        ,
        \dots
        ,
        X_s^{ \#_{I^\varpi_{\#_\varpi}}, [\mathbf{u} ]_{ \#_\varpi }^{ \varpi } }
      \big)
    \bigg]
  \, \diffns W_s
  .
\end{split}
\end{equation}

We now prove~\eqref{eq:SEE.3} by induction on $ k \in \{1,2,\ldots,n\} $.
For the base case $k=1$ note that, e.g., 
item~(i) of Corollary~2.10 in~\cite{AnderssonJentzenKurniawan2016arXiv} 
(with 
$H=H$, 
$U=U$, 
$T=T$, 
$\eta=\eta$, 
$\alpha=0$, 
$\beta=0$,
$W=W$, 
$A=A$, 
$F=F$, 
$B=B$, 
$ \delta = 0 $
in the notation of Corollary~2.10 in~\cite{AnderssonJentzenKurniawan2016arXiv})
ensures the existence of up-to-modifications unique 
$ ( \mathcal{F}_t )_{ t \in [0,T] } $/$\mathcal{B}(H)$-predictable stochastic processes 
$ X^{0,x} \colon [0,T] \times \Omega \to H $, $ x \in H $, 
which fulfill for all $ p \in [2,\infty) $, 
$ x \in H $, $ t \in [0,T] $ that 
$ \sup_{ s \in [0,T] } \E[ \| X^{0,x}_s \|^p_H ] < \infty $ 
and 
\begin{equation}
\label{eq:base.case.0}
  [ X^{0,x}_t - e^{ tA } x ]_{ \P, \mathcal{B}(H) }
  =
  \int^t_0
  e^{ (t-s) A } F( X^{0,x}_s ) \, {\bf ds}
  +
  \int^t_0
  e^{ (t-s) A } B( X^{0,x}_s ) \, dW_s
  .  
\end{equation}
Next we note that for all 
$ l \in \{0,1\} $, 
$ p \in [2,\infty) $, 
$ u \in H $, 
$ Y, Z \in \lpn{p}{\P}{H} $, 
$ t \in (0,T] $ 
it holds that 
\begin{equation}
\label{eq:Lip.G.1}
\begin{split}
  \| G'_l(X^u_t) Y - G'_l(X^u_t) Z \|_{\lpn{p}{\P}{V_{l,0}}}
&\leq
  | G_l |_{ \Cb{1}( H, V_{l,0} ) } \,
  \| Y - Z \|_{ \lpn{p}{\P}{H} }
  \qquad\text{and}
\\
  \| G'_l(X^u_t) 0 \|_{\lpn{p}{\P}{V_{l,0}}} &= 0
  .
\end{split}
\end{equation}
This allows us to apply item~(i) of Theorem~2.9 in~\cite{AnderssonJentzenKurniawan2016arXiv}
(with 
$
  H=H
$, 
$
  U=U
$, 
$ T = T $, 
$ \eta = \eta $, 
$ p = p $, 
$ \alpha = 0 $, 
$
  \hat{\alpha} = 0
$,
$ \beta = 0 $, 
$
  \hat{\beta} 
 = 
  0
$,
$
  L_0 = |F|_{ \Cb{1}( H, H ) }
$, 
$
  \hat{L}_0 = 0
$,
$
  L_1 = |B|_{ \Cb{1}( H, HS(U,H) ) }
$, 
$
  \hat{L}_1 = 0
$,
$
  W=W
$, 
$ A=A $,
$
  \mathbf{F} = 
  \big(
    [0,T] \times \Omega \times H
    \ni (t,\omega,x)
    \mapsto
    F'(X^{0,u_0}_t(\omega)) x
    \in H
  \big)
$,
$
  \mathbf{B} = 
  \big(
    [0,T] \times \Omega \times H
    \ni (t,\omega,x)
    \mapsto
    B'(X^{0,u_0}_t(\omega)) x
    \in HS(U,H)
  \big)
$,
$
  \delta = 0
$, 
$
  \lambda = 
  0
$, 
$
  \xi = 
  ( \Omega \ni \omega \mapsto u_1 \in H )
$
for 
$
  u_0, u_1
  \in
  H
$, 
$ p \in [2,\infty) $
in the notation of Theorem~2.9 in~\cite{AnderssonJentzenKurniawan2016arXiv}) 
to obtain that there exist up-to-modifications unique
$ ( \mathcal{F}_t )_{ t \in [0,T] } $/$\mathcal{B}(H)$-predictable stochastic processes
$
  X^{1,\mathbf{u}} \colon [0,T] \times \Omega \to H
$, 
$
  \mathbf{u} \in H^2
$, 
which fulfill for all 
$ p \in [2,\infty) $, 
$
  \mathbf{u} = ( u_0, u_1 ) \in
    H^2
$, 
$ t \in [0,T] $ 
that 
$
  \sup_{ s \in [0,T] }
  \E[ \|X^{1,\mathbf{u}}_s\|^p_H ]
  < \infty
$ 
and 
\begin{equation}
\begin{split}
  [X_t^{1,\mathbf{u}}-e^{tA}u_1]_{ \P, \mathcal{B}(H) }
  =
  \int_0^t
    e^{ ( t - s ) A }
      F'( X_s^{ 0, u_0 } ) \,
      X^{ 1, \mathbf{u} }_s
  \, { \bf ds }
  +
  \int_0^t
    e^{ ( t - s ) A }
      B'( X_s^{ 0, u_0 } ) \,
      X^{ 1, \mathbf{u} }_s
  \, \diffns W_s.
\end{split}
\end{equation}
This and~\eqref{eq:base.case.0} prove~\eqref{eq:SEE.3} in the base case $k=1$. 
For the induction step 
$ \{ 1,2,\ldots,n-1 \} \ni k \to k+1 \in \{2,3,\ldots,n\} $ 
we introduce more notation. 
Assume that there exists a natural number 
$ k \in \{ 1,2,\ldots,n-1 \} $ 
such that~\eqref{eq:SEE.3} holds for $k=k$, 
let 
$
  X^{l,\mathbf{u}} \colon [0,T] \times \Omega \to H
$, 
$
  \mathbf{u} \in H^{l+1}
$, 
$
  l \in \{ 2, 3, \dots, k \}
$, 
be up-to-modifications unique
$ ( \mathcal{F}_t )_{ t \in [0,T] } $/$ \mathcal{B}(H) $-predictable stochastic processes
which fulfill for all 
$
  l \in \{ 2, 3, \dots, k \}
$, 
$ p \in [2,\infty) $, 
$
  \mathbf{u} = ( u_0, u_1, \ldots, u_l ) \in
    H^{l+1}
$, 
$ t \in [0,T] $ 
that 
$
  \sup_{ s \in [0,T] }
  \E[ \|X^{l,\mathbf{u}}_s\|^p_H ]
  < \infty
$ 
and 
\begin{equation}
\begin{split}
  [
  X_t^{l,\mathbf{u}}
  ]_{ \P,\mathcal{B}(H) }
&=
  \int_0^t
    e^{ ( t - s ) A }
      \sum_{ \varpi\in \Pi_l }
      F^{ ( \#_\varpi ) }( X_s^{ 0, u_0 } )
      \big(
       X_s^{ \#_{I^\varpi_1}, [ \mathbf{u} ]_1^{ \varpi } }
        ,
        X_s^{ \#_{I^\varpi_2}, [ \mathbf{u} ]_2^{ \varpi } }
        ,
        \dots
        ,
        X_s^{ \#_{I^\varpi_{\#_\varpi}}, [\mathbf{u} ]_{ \#_\varpi }^{ \varpi } }
      \big)
  \,{\bf ds}
\\ & \quad
  +
  \int_0^t
    e^{ ( t - s ) A }
      \sum_{ \varpi\in \Pi_l }
      B^{ ( \#_\varpi ) }( X_s^{ 0, u_0 } )
      \big(
       X_s^{ \#_{I^\varpi_1}, [ \mathbf{u} ]_1^{ \varpi } }
        ,
        X_s^{ \#_{I^\varpi_2}, [ \mathbf{u} ]_2^{ \varpi } }
        ,
        \dots
        ,
        X_s^{ \#_{I^\varpi_{\#_\varpi}}, [\mathbf{u} ]_{ \#_\varpi }^{ \varpi } }
      \big)
  \, \diffns W_s
  ,
\end{split}
\end{equation}
let 
$
  \mathcal{G}^{ \mathbf{u} }_l \colon [0,T] \times \Omega \times H \to V_{l,0}
$, 
$
  \mathbf{u}
  \in H^{k+2}
$, 
$ l \in \{ 0,1 \} $, 
be the functions which satisfy for all
$ l \in \{ 0,1 \} $, 
$
  \mathbf{u} = ( u_0, u_1, \dots, u_{ k + 1 } )
  \in H^{k+2}
$, 
$ t \in [0,T] $,
$ x \in H $
that
\begin{equation}
  \mathcal{G}^{ \mathbf{u} }_l(t, x )
  =
  G_l'( X_t^{ 0, u_0 } ) \, x
  +
  {\smallsum_{ \varpi \in \Pi_{ k + 1 }^{ * } }}
    G_l^{ ( \#_\varpi ) }( X_t^{ 0, u_0 } )
    \big(
       X_t^{ \#_{I^\varpi_1}, [ \mathbf{u} ]_1^{ \varpi } }
        ,
        X_t^{ \#_{I^\varpi_2}, [ \mathbf{u} ]_2^{ \varpi } }
        ,
        \dots
        ,
        X_t^{ \#_{I^\varpi_{\#_\varpi}}, [\mathbf{u} ]_{ \#_\varpi }^{ \varpi } }
    \big),
\end{equation}
and let 
$ \bar{L}^{ \mathbf{u}, p }_l \in [0,\infty) $, 
$
  \mathbf{u}
  \in H^{k+2}
$, 
$ p \in [2,\infty) $, 
$ l \in \{ 0,1 \} $, 
be the real numbers which satisfy for all 
$ l \in \{ 0,1 \} $, 
$ p \in [2,\infty) $, 
$
  \mathbf{u}
  \in H^{k+2}
$
that 
\begin{equation}
\begin{split}
&
  \bar{L}^{ \mathbf{u}, p }_l
  =
    \smallsum_{ \varpi \in \Pi_{ k + 1 }^{ * } }
    | G_l |_{ 
      \Cb{ \#_\varpi }( H, V_{l,0} ) 
    }
      \smallprod_{ i = 1 }^{ \#_\varpi }
      \big\|\stochval{X^{
                  \#_{I^\varpi_i}, 
                  [ \mathbf{u} ]_i^{ \varpi }
                }}\big\|_{ \mathbb{L}^{ p \#_\varpi } }
    .
\end{split}
\end{equation}
Next we note that H\"{o}lder's inequality implies for all 
$ l \in \{0,1\} $, 
$ p \in [2,\infty) $,
$
  \mathbf{u} = ( u_0, u_1, \dots, u_{ k + 1 } )
  \in H^{k+2}
$,
$
  Y, Z \in \mathcal{L}^p( \P; H )
$,
$ t \in (0,T] $
that
\begin{equation}
\label{eq:Lip.G.n}
  \left\|
    \mathcal{G}^{ \mathbf{u} }_l(t, Y )
    -
    \mathcal{G}^{ \mathbf{u} }_l(t, Z )
  \right\|_{
    \mathcal{L}^p( \P; V_{l,0} )
  }
\leq 
  \left| G_l \right|_{
    \Cb{1}( H, V_{l,0} )
  }
  \left\| Y - Z \right\|_{
    \mathcal{L}^p( \P; H )
  }
\end{equation}
and 
\begin{equation}
\label{eq:singular.G.n}
\begin{split}
& \big\|
    \mathcal{G}^{ \mathbf{u} }_l( t , 0 )
  \big\|_{\mathcal{L}^p(\P;V_{l,0})}
\\&\leq
  \smallsum_{ 
    \varpi \in \Pi_{ k + 1 }^{ * } 
  }
    \big\|
      G_l^{ ( \#_\varpi ) }( X_t^{ 0, u_0 } )
      \big(
        X_t^{ \#_{I^\varpi_1}, [ \mathbf{u} ]_1^{ \varpi } }
        ,
        X_t^{ \#_{I^\varpi_2}, [ \mathbf{u} ]_2^{ \varpi } }
        ,
        \dots
        ,
        X_t^{ \#_{I^\varpi_{\#_\varpi}}, [\mathbf{u} ]_{ \#_\varpi }^{ \varpi } }
      \big)
    \big\|_{
      \mathcal{L}^p( \P; V_{l,0} ) 
    }
\\
& \leq
  \smallsum_{ \varpi \in \Pi_{ k + 1 }^{ * } }
  | G_l |_{ 
    \Cb{ \#_\varpi }( H, V_{l,0} ) 
  }
    \smallprod_{ i = 1 }^{ \#_\varpi }
    \big\|
      X_t^{
        \#_{I^\varpi_i}, 
        [ \mathbf{u} ]_i^{ \varpi }
      }
    \big\|_{
      \mathcal{L}^{ p \, \#_\varpi }( \P ; H )
    }
\leq
  \bar{L}^{ \mathbf{u}, p }_l
  .
\end{split}
\end{equation}
We can hence apply item~(i) of Theorem~2.9 in~\cite{AnderssonJentzenKurniawan2016arXiv} 
(with 
$
  H=H
$, 
$
  U=U
$, 
$ T = T $, 
$ \eta = \eta $, 
$ p = p $, 
$ \alpha = 0 $, 
$
  \hat{\alpha} = 0
$,
$ \beta = 0 $, 
$
  \hat{\beta} 
 = 
  0
$,
$
  L_0 = |F|_{ \Cb{1}( H, H ) }
$, 
$
  \hat{L}_0 = \bar{L}_0^{ \mathbf{u}, p }
$,
$
  L_1 = |B|_{ \Cb{1}( H, HS(U,H) ) }
$, 
$
  \hat{L}_1 = \bar{L}_1^{ \mathbf{u}, p }
$,
$
  W=W
$, 
$ A=A $,
$
  \mathbf{F} = \mathcal{G}^{ \mathbf{u} }_0
$,
$
  \mathbf{B} = \mathcal{G}^{ \mathbf{u} }_1
$,
$
  \delta = 0
$, 
$
  \lambda = 
  0
$, 
$
  \xi = 
  ( \Omega \ni \omega \mapsto 0 \in H )
$
for 
$
  \mathbf{u}
  \in
  H^{k+2}
$, 
$ p \in [2,\infty) $
in the notation of Theorem~2.9 in~\cite{AnderssonJentzenKurniawan2016arXiv}) 
to obtain that there exist up-to-modifications unique
$ ( \mathcal{F}_t )_{ t \in [0,T] } $/$ \mathcal{B}(H) $-predictable stochastic processes
$
  X^{k+1,\mathbf{u}} \colon [0,T] \times \Omega \to H
$, 
$
  \mathbf{u} \in H^{k+2}
$, 
which fulfill for all 
$ p \in [2,\infty) $, 
$
  \mathbf{u} = ( u_0, u_1, \ldots, u_{k+1} ) \in
    H^{k+2}
$, 
$ t \in [0,T] $ 
that 
$
  \sup_{ s \in [0,T] }
  \E[ \|X^{k+1,\mathbf{u}}_s\|^p_H ]
  < \infty
$ 
and 
\begin{equation}
\label{eq:SEE2_Proof}
\begin{split}
&
  [X_t^{k+1,\mathbf{u}}]_{ \P, \mathcal{B}(H) }
=
  \int_0^t
    e^{ ( t - s ) A }
    \mathcal{G}^{ \mathbf{u} }_0( s, X_s^{ k+1, \mathbf{u} } )
  \, {\bf ds}
  +
  \int_0^t
    e^{ ( t - s ) A }
    \mathcal{G}^{ \mathbf{u} }_1( s, X_s^{ k+1, \mathbf{u} } )
  \, \diffns W_s
\\ & 
  =
  \int_0^t
    e^{ ( t - s ) A }
      \sum_{ \varpi\in \Pi_{ k + 1 } }
      F^{ ( \#_\varpi ) }( X_s^{ 0, u_0 } )
      \big(
        X_s^{ \#_{I^\varpi_1}, [ \mathbf{u} ]_1^{ \varpi } }
        ,
        X_s^{ \#_{I^\varpi_2}, [ \mathbf{u} ]_2^{ \varpi } }
        ,
        \dots
        ,
        X_s^{ \#_{I^\varpi_{\#_\varpi}}, [\mathbf{u} ]_{ \#_\varpi }^{ \varpi } }
      \big)
  \, { \bf ds }
\\ & \quad
  +
  \int_0^t
    e^{ ( t - s ) A }
      \sum_{ \varpi\in \Pi_{ k + 1 } }
      B^{ ( \#_\varpi ) }( X_s^{ 0, u_0 } )
      \big(
        X_s^{ \#_{I^\varpi_1}, [ \mathbf{u} ]_1^{ \varpi } }
        ,
        X_s^{ \#_{I^\varpi_2}, [ \mathbf{u} ]_2^{ \varpi } }
        ,
        \dots
        ,
        X_s^{ \#_{I^\varpi_{\#_\varpi}}, [\mathbf{u} ]_{ \#_\varpi }^{ \varpi } }
      \big)
  \, \diffns W_s.
\end{split}
\end{equation}
This proves~\eqref{eq:SEE.3} in the case $k+1$. 
Induction hence establishes~\eqref{eq:SEE.3}. 
The proof of item~\eqref{item:lem_derivative:existence} is thus completed.

For our proof of
items~\eqref{item:lem_derivative:a_priori}--\eqref{item:time.derivative.representation}
we introduce further notation.
Let 
$ 
  X^{k,\mathbf{u}} \colon [0,T] \times \Omega \to H
$, 
$
  \mathbf{u} \in H^{k+1}
$, 
$ k \in \{0,1,\ldots,n\} $, 
be 
$
  (\mathcal{F}_t)_{ t \in [0,T] }
$/$ \mathcal{B}(H) $-predictable stochastic processes which fulfill for all 
$
  k \in \{ 0, 1, \dots, n \}
$, 
$ p \in [2,\infty) $, 
$
  \mathbf{u} = ( u_0, u_1, \ldots, u_k ) \in
    H^{k+1}
$, 
$ t \in [0,T] $ 
that 
$
  \sup_{ s \in [0,T] }
  \E[ \|X^{k,\mathbf{u}}_s\|^p_H ]
  < \infty
$ 
and 
\begin{equation}
\label{eq:proof.derivatives}
\begin{split}
&
  [
  X_t^{k,\mathbf{u}}
  -
  e^{tA}
  \, \mathbbm{1}_{ \{ 0, 1 \} }(k) \, u_k  
  ]_{ \P,\mathcal{B}(H) }
\\ &
  =
  \int_0^t
    e^{ ( t - s ) A }
    \bigg[
      \mathbbm{1}_{ \{ 0 \} }(k)
      \,
      F(X_s^{0,u_0})
      +
      \sum_{ \varpi\in \Pi_k }
      F^{ ( \#_\varpi ) }( X_s^{ 0, u_0 } )
      \big(
        X_s^{ \#_{I^\varpi_1}, [ \mathbf{u} ]_1^{ \varpi } }
        ,
        X_s^{ \#_{I^\varpi_2}, [ \mathbf{u} ]_2^{ \varpi } }
        ,
        \dots
        ,
        X_s^{ \#_{I^\varpi_{\#_\varpi}}, [\mathbf{u} ]_{ \#_\varpi }^{ \varpi } }
      \big)
    \bigg]
  \,{\bf ds}
\\ &
  +
  \int_0^t
    e^{ ( t - s ) A }
    \bigg[
      \mathbbm{1}_{ \{ 0 \} }(k)
      \,
      B(X_s^{0,u_0})
      +
      \sum_{ \varpi\in \Pi_k }
      B^{ ( \#_\varpi ) }( X_s^{ 0, u_0 } )
      \big(
        X_s^{ \#_{I^\varpi_1}, [ \mathbf{u} ]_1^{ \varpi } }
        ,
        X_s^{ \#_{I^\varpi_2}, [ \mathbf{u} ]_2^{ \varpi } }
        ,
        \dots
        ,
        X_s^{ \#_{I^\varpi_{\#_\varpi}}, [\mathbf{u} ]_{ \#_\varpi }^{ \varpi } }
      \big)
    \bigg]
  \, \diffns W_s
  ,
\end{split}
\end{equation}
let 
$
  L^{ \boldsymbol{\delta} }_{\varpi, p}
  \in [0,\infty]
$, 
$
  \varpi \in \mathcal{P}\big(
  \mathcal{P}(\{ 1, 2, \ldots, k \})
  \setminus \{\emptyset\}
  \big)
$, 
$
  \boldsymbol{\delta} \in \deltaset{k} 
$,
$ p \in (0,\infty) $, 
$
  k \in \{ 1, 2, \ldots, n \}
$, 
be the extended real numbers which satisfy for all 
$
  k \in \{ 1, 2, \ldots, n \}
$, 
$ p \in (0,\infty) $, 
$
  \boldsymbol{\delta}=
  (\delta_1, \delta_2, \ldots, \delta_k) \in \deltaset{k} 
$, 
$
  \varpi \in \mathcal{P}\big(
  \mathcal{P}(\{ 1, 2, \ldots, k \})
  \setminus \{\emptyset\}
  \big)
  \setminus \{\emptyset\}
$
that 
$
  L^{ \boldsymbol{\delta} }_{\emptyset,p}
  =
  1
$ 
and 
\begin{equation}
  L^{ \boldsymbol{\delta} }_{\varpi,p}
  =
  \prod_{ I\in\varpi }
    \sup_{\mathbf{u}=(u_i)_{i \in I\cup\{0\}}\in(\times_{i\in I\cup\{0\}}H^{[i]})}\,
    \sup_{ t\in (0,T] }
    \left[
    \frac{
      t^{ \iota^{ \boldsymbol{\delta} }_{ I 
      } }\,
      \|
        X^{ \#_I, \mathbf{u} }_t
      \|_{ \lpn{p}{\P}{H} }
    }{
      \prod_{ i\in I }
      \| u_i \|_{ H_{ -\delta_i } }
    }
    \right],
\end{equation}
let 
$
\tilde{L}_p
\in [0,\infty]
$, 
$ p \in (0,\infty) $, 
be the extended real numbers which satisfy for all 
$ p \in (0,\infty) $
that 
\begin{equation}
\begin{split}
\tilde{L}_p
=
\sup_{ u_0 \in H }
\sup_{ u_1 \in \nzspace{H} }
\sup_{ t\in (0,T] }
\left[
\frac{
	\|
	X^{0,u_0+u_1}_t - X^{0,u_0}_t
	\|_{ \lpn{p}{\P}{H} }
}{
\| u_1 \|_H
}
\right],
\end{split}
\end{equation} 
let 
$ \hat{L}^{ \boldsymbol{\delta}, \mathbf{u}, p }_{k,l} \in [0,\infty] $,
$
  \mathbf{u}
  \in H^{k+1}
$, 
$
  \boldsymbol{\delta} \in \deltaset{k}
$,
$ p \in (0,\infty) $,
$ l \in \{0,1\} $, 
$ k \in \{1,2,\ldots,n\} $, 
be the extended real numbers which satisfy for all 
$ k \in \{1,2,\ldots,n\} $, 
$ l \in \{0,1\} $, 
$ p \in (0,\infty) $,
$
  \boldsymbol{\delta} \in \deltaset{k}
$,
$
  \mathbf{u} = ( u_0, u_1, \dots, u_k )
  \in H^{k+1}
$ 
that 
\begin{equation}
\begin{split}
  \hat{L}^{ \boldsymbol{\delta}, \mathbf{u}, p }_{k,l}
  &=
    | T \vee 1 |^{
       \lfloor k / 2 \rfloor 
       \min\{ 1-\alpha, \nicefrac{1}{2}-\beta \}
    }
\\ & \quad
  \cdot
    \sum_{ \varpi \in \Pi_k^{ * } }
    | G_l |_{ 
      \Cb{ \#_\varpi }( H, V_{l,r_l} ) 
    }
      \prod_{ i = 1 }^{ \#_\varpi }
        \sup_{ t \in (0,T] }
      \Big[
        t^{
          \iota_{ I^{ \varpi }_i 
          }^{ \boldsymbol{\delta} }
        }
        \|
          X_t^{
            \#_{ I^\varpi_i }, 
            [ \mathbf{u} ]_i^{ \varpi }
            }
        \|_{
          \mathcal{L}^{ p \, \#_\varpi }( \P ; H )
        }
      \Big]
    ,
\end{split}
\end{equation}
for every 
$
  k \in \{ 1, 2, \ldots, n \}
$, 
$ l \in \{0,1\} $, 
$
  \mathbf{u} = ( u_0, u_1, \dots, u_k ) 
  \in H^{k+1}
$ 
let 
$ 
  \mathbf{G}_{ k, l }^{ \mathbf{u} } \colon [0,T] \times \Omega \times H \to V_{l,0} 
$ 
and 
$
  \mathbf{\bar{G}}^{ \mathbf{u} }_{ k, l }
  \colon
  [0,T] \times \Omega \times H
  \to
  V_{l,0}
$
be the functions which satisfy 
for all 
$ x \in H $,
$ t \in [0,T] $ 
that 
\begin{equation}
\label{eq:G.def}
  \mathbf{G}^{ \mathbf{u} }_{ k, l }(t,x) =
  G'_l( X^{0,u_0}_t ) x
  +
  {\smallsum_{ \varpi \in \Pi^*_k }}
  G^{ (\#_\varpi) }_l( X^{0,u_0}_t )
  \big( 
        X_t^{ \#_{I^\varpi_1}, [ \mathbf{u} ]_1^{ \varpi } }
        ,
        X_t^{ \#_{I^\varpi_2}, [ \mathbf{u} ]_2^{ \varpi } }
        ,
        \dots
        ,
        X_t^{ \#_{I^\varpi_{\#_\varpi}}, [\mathbf{u} ]_{ \#_\varpi }^{ \varpi } }
  \big)
\end{equation}
and 
\begin{equation}
\label{eq:barG.def}
\begin{split}
&
  \mathbf{\bar{G}}^{ \mathbf{u} }_{ k,l }(t,x)
  =
\\&
\begin{cases}
  \int^1_0
  G'_l\big( X^{0,u_0}_t + \rho [X^{ 0, u_0 + u_1 }_t-X^{ 0, u_0 }_t] \big) \,x
  \, d\rho
  &\colon k=1
\\\vspace{-10pt}
\\
\begin{split}
&
    G'_l( X_t^{ 0, u_0 } ) \, x
\\&+
    {\textstyle\int}_0^1 \,
    G_l''\big(
      X_t^{ 0, u_0 } + \rho [ X^{0,u_0+u_{k}}_t - X^{0,u_0}_t ]
    \big)
    \big(
      X_t^{ k-1, \theta^k_1( \mathbf{u} ) } ,
      X^{0,u_0+u_{k}}_t - X^{0,u_0}_t
    \big)
  \, \diffns\rho
\\
&
  +
  \textstyle
  \smallsum_{
    \varpi \in \Pi_{k-1}^{ * }
  }
  \Big[
    {\textstyle\int}^1_0 \,
    G_l^{ ( \#_\varpi + 1 ) }\big(
      X_t^{ 0, u_0 } 
      + 
      \rho [X^{0,u_0+u_{k}}_t - X^{0,u_0}_t]
    \big)
    \big(
        X_t^{ \#_{I^\varpi_1}, [ \theta^k_1(\mathbf{u}) ]_1^{ \varpi } }
        ,
\\&\quad
        X_t^{ \#_{I^\varpi_2}, [ \theta^k_1(\mathbf{u}) ]_2^{ \varpi } }
        ,
        \dots
        ,
        X_t^{ \#_{I^\varpi_{\#_\varpi}}, [ \theta^k_1(\mathbf{u}) ]_{ \#_\varpi }^{ \varpi } }
      ,
      X^{0,u_0+u_{k}}_t - X^{0,u_0}_t
    \big)
    \, \diffns\rho
\\&+
  G^{(\#_\varpi)}_l( X^{0,u_0}_t )\big(
        X_t^{ \#_{I^\varpi_1}, [ \theta^k_1(\mathbf{u}) ]_1^{ \varpi } }
        ,
        X_t^{ \#_{I^\varpi_2}, [ \theta^k_1(\mathbf{u}) ]_2^{ \varpi } }
        ,
        \dots
        ,
        X_t^{ \#_{I^\varpi_{\#_\varpi}}, [ \theta^k_1(\mathbf{u}) ]_{ \#_\varpi }^{ \varpi } }
  \big)
\\&-
  G^{(\#_\varpi)}_l( X^{0,u_0}_t )\big(
        X_t^{ \#_{I^\varpi_1}, [ \theta^k_0(\mathbf{u}) ]_1^{ \varpi } }
        ,
        X_t^{ \#_{I^\varpi_2}, [ \theta^k_0(\mathbf{u}) ]_2^{ \varpi } }
        ,
        \dots
        ,
        X_t^{ \#_{I^\varpi_{\#_\varpi}}, [ \theta^k_0(\mathbf{u}) ]_{ \#_\varpi }^{ \varpi } }
  \big)
  \Big]
\end{split}
  &\colon k > 1
\end{cases}
,
\end{split}
\end{equation}
and for every 
$
  k \in \{ 1, 2, \ldots, n \}
$, 
$
  p \in (0,\infty)
$ 
let 
$
  d_{k,p}
  \colon
  H^2\rightarrow[0,\infty]
$ 
and 
$
  \tilde{d}_{k,p}
  \colon
  H \times (\nzspace{H}) \rightarrow[0,\infty]
$
be the functions which satisfy for all 
$ x,y \in H $, 
$ v \in \nzspace{H} $ 
that 
\begin{equation}
\begin{split}
  d_{k,p}(x,y)
  =
  \sup\limits_{
    \mathbf{u}=(u_1,u_2,\ldots,u_k)
    \in (\nzspace{H})^k
  }
  \sup\limits_{ t\in(0,T] }
  \left[
    \frac{
      \|
        X^{ k, (x,\mathbf{u}) }_t - X^{ k, (y,\mathbf{u}) }_t
      \|_{ \lpn{ p }{ \P }{ H } }
    }{
      \prod^k_{i = 1}
      \| u_i \|_H
    }
  \right]
\end{split}
\end{equation}
and 
\begin{equation}
\begin{split}
&
  \tilde{d}_{k,p}(x,v)
\\ & =
\begin{cases}
  \sup\limits_{ t\in(0,T] }
  \left[
  \frac{
  	\|
  	X^{0,x+v}_t - X^{0,x}_t - X^{ 1, (x,v) }_t
  	\|_{ \lpn{ p }{ \P }{ H } }
  }{
  \| v \|_H
}
\right]  
&\colon k = 1
\\
  \sup\limits_{
    \mathbf{u}=(u_1,u_2,\ldots,u_{k-1})
    \in (\nzspace{H})^{k-1}
  }
  \sup\limits_{ t\in(0,T] }
  \left[
    \frac{
      \|
        X^{ k-1, (x+v,\mathbf{u}) }_t - X^{ k-1, (x,\mathbf{u}) }_t - X^{ k, (x,\mathbf{u},v) }_t
      \|_{ \lpn{ p }{ \P }{ H } }
    }{
      \|v\|_H
      \prod^{k-1}_{i = 1}
      \| u_i \|_H
    }
  \right]
  &\colon k > 1
\end{cases}
.
\end{split}
\end{equation}
In the next step we prove item~\eqref{item:lem_derivative:a_priori} and the fact that for all 
$ k \in \{1,2,\ldots,n\} $, 
$ p \in [2,\infty) $, 
$ x \in H $, 
$ t \in [0,T] $
it holds that
\begin{equation}
\label{eq:k.linear}
  H^k \ni \mathbf{u} 
  \mapsto [X^{k,(x,\mathbf{u})}_t]_{\P,\mathcal{B}(H)} 
  \in \lpnb{p}{\P}{H}
\end{equation}
is a $k$-linear function.

We prove item~\eqref{item:lem_derivative:a_priori} and~\eqref{eq:k.linear} by induction on $k\in\{1,2,\ldots,n\}$. 
Note that for all 
$ l \in \{0,1\} $, 
$ p \in [2,\infty) $,
$
  \mathbf{u}
  \in H^2
$,
$
  Y, Z \in \mathcal{L}^p( \P; H )
$,
$ t \in (0,T] $
it holds that
\begin{equation}
\label{eq:exist.cond}
\begin{split}
  \|
    \mathbf{G}^{ \mathbf{u} }_{1,l}(t, Y )
    -
    \mathbf{G}^{ \mathbf{u} }_{1,l}(t, Z )
  \|_{
    \mathcal{L}^p( \P; V_{l,r_l} )
  }
&\leq 
  \left| G_l \right|_{
    \Cb{1}( H, V_{l,r_l} )
  }
  \left\| Y - Z \right\|_{
    \mathcal{L}^p( \P; H )
  }
\qquad\text{and}
\\
    \|\mathbf{G}^{ \mathbf{u} }_{1,l}( t , 0 )\|_{
        \mathcal{L}^p( \P; V_{l,r_l} )
      }
&=
  0
  .
\end{split}
\end{equation}
Moreover, observe that~\eqref{eq:proof.derivatives} and~\eqref{eq:G.def} ensure that for all 
$
  \mathbf{u} = ( u_0, u_1 )
  \in H^2
$, 
$ t \in [0,T] $
it holds that
\begin{equation}
\label{eq:exist.sol}
\begin{split}
  [X_t^{1,\mathbf{u}}]_{ \P, \mathcal{B}(H) }
& =
  [e^{tA} u_1]_{ \P, \mathcal{B}(H) }
  +
  \int_0^t
    e^{ ( t - s ) A }
    \mathbf{G}^{ \mathbf{u} }_{1,0}( s, X_s^{ 1, \mathbf{u} } )
  \, {\bf ds}
  +
  \int_0^t
    e^{ ( t - s ) A }
    \mathbf{G}^{ \mathbf{u} }_{1,1}( s, X_s^{ 1, \mathbf{u} } )
  \, \diffns W_s
  .
\end{split}
\end{equation}
Combining~\eqref{eq:exist.cond}--\eqref{eq:exist.sol} with 
items~(i)--(ii) of Theorem~2.9 in~\cite{AnderssonJentzenKurniawan2016arXiv} 
(with 
$
  H=H
$, 
$
 U=U
$, 
$ T = T $, 
$ \eta = \eta $, 
$ p = p $, 
$ \alpha = \alpha $, 
$
  \hat{\alpha} = 0 
$, 
$ \beta = \beta $, 
$
  \hat{\beta} = 0 
$,
$
  L_0 = |F|_{ \Cb{1}( H, H_{-\alpha} ) }
$, 
$
  \hat{L}_0 = 0
$, 
$
  L_1 = |B|_{ \Cb{1}( H, HS( U, H_{-\beta} ) ) }
$, 
$
  \hat{L}_1 = 0
$, 
$
  W=W
$, 
$ A=A $,
$
  \mathbf{F} = 
  \big(
  [0,T]\times\Omega\times H \ni (t,\omega,x) \mapsto
  \mathbf{G}^{ \mathbf{u} }_{1,0}(t,\omega,x) \in H_{-\alpha}
  \big)
$,
$
  \mathbf{B} = 
  \big(
  [0,T]\times\Omega\times H \ni (t,\omega,x) \mapsto
  ( U \ni u \mapsto
  \mathbf{G}^{ \mathbf{u} }_{1,1}(t,\omega,x) u \in H_{-\beta} 
  ) \in HS(U,H_{-\beta})
  \big)
$,
$
  \delta = \delta
$,
$
  \lambda = 
  \delta
$,  
$
  \xi = ( \Omega \ni \omega \mapsto u_1 \in H_{-\delta} )
$
for  
$
  \mathbf{u} = (u_0,u_1)
  \in H^2 
$, 
$
  \delta \in [ 0 , \nicefrac{1}{2} )
$, 
$ p \in [2,\infty) $ 
in the notation of Theorem~2.9 in~\cite{AnderssonJentzenKurniawan2016arXiv}) 
implies that for all
$ p \in [2,\infty) $, 
$ \delta \in [ 0 , \nicefrac{1}{2} ) $
it holds that 
\begin{equation}
\label{eq:base_case_apriori}
\begin{split}
&  \sup_{ u_0 \in H }
  \sup_{ u_1 \in \nzspace{H} }
  \sup_{ t \in (0,T] }
  \left[
  \frac{
    t^{ 
      \delta
    }
    \|
      X_t^{ 1, (u_0,u_1) }
    \|_{
      \mathcal{L}^p( \P; H )
    }
  }{
    \| u_1 \|_{ H_{ -\delta } }
  }
  \right]
\\&\leq
  \Theta_{ A, \eta, p, T }^{
    \alpha , 
    \beta ,
    \delta
  }( |F|_{ \Cb{1}( H, H_{-\alpha} ) } , |B|_{\Cb{1}( H, HS( U, H_{-\beta} ) )} )
  \sup_{u_0\in H} \sup_{ u_1 \in \nzspace{H} }
  \left[
  \frac{
    \sup_{ t \in (0,T] }
    (
      t^\delta \, \|e^{tA} u_1\|_H
    )
  }{
    \| u_1 \|_{ H_{ -\delta } }
  }
  \right]
\\&\leq
  \bigg[
    \sup_{ t \in (0,T] }
    t^\delta \, \| (\eta-A)^\delta e^{tA} \|_{L(H)}
  \bigg] \,
  \Theta_{ A, \eta, p, T }^{
    \alpha , 
    \beta ,
    \delta
  }( |F|_{ \Cb{1}( H, H_{-\alpha} ) } , |B|_{\Cb{1}( H, HS( U, H_{-\beta} ) )} )
\\&=
      \chi_{ A, \eta }^{ \delta, T }\,
  \Theta_{ A, \eta, p, T }^{
    \alpha , 
    \beta ,
    \delta
  }( |F|_{ \Cb{1}( H, H_{-\alpha} ) } , |B|_{\Cb{1}( H, HS( U, H_{-\beta} ) )} )
  < \infty
  .
\end{split}
\end{equation}
This proves item~\eqref{item:lem_derivative:a_priori} in the base case $ k=1 $.
Next we observe that~\eqref{eq:proof.derivatives} shows that for all 
$ p \in [2,\infty) $, 
$ x, u, \tilde{u} \in H $, 
$ \lambda \in \R $, 
$ t \in [0,T] $ 
it holds that 
$
  \sup_{ s\in[0,T] }
  \E\big[
    \|X^{ 1, (x,u) }_s\|^p_H
    +
    \|X^{ 1, (x,\tilde{u}) }_s\|^p_H
  \big]
  < \infty
$
and 
\begin{equation}
\begin{split}
&
  [ X^{ 1, (x,u) }_t + \lambda X^{ 1, (x,\tilde{u}) }_t ]_{ \P, \mathcal{B}(H) }
  =
  [ e^{tA} ( u + \lambda \tilde{u} ) ]_{ \P, \mathcal{B}(H) }
\\&+
  \int^t_0
  e^{(t-s)A}
  F'( X^{0,x}_s ) ( X^{ 1, (x,u) }_s + \lambda X^{ 1, (x, \tilde{u}) }_s )
  \, { \bf ds }
+
  \int^t_0
  e^{(t-s)A}
  B'( X^{0,x}_s ) ( X^{ 1, (x,u) }_s + \lambda X^{ 1, (x, \tilde{u}) }_s )
  \, dW_s
  .
\end{split}
\end{equation}
Item~\eqref{item:lem_derivative:existence} therefore ensures for all 
$ x, u, \tilde{u} \in H $, 
$ \lambda \in \R $, 
$ t \in [0,T] $
that
\begin{equation}
  [X^{ 1, ( x, u + \lambda \tilde{u} ) }_t]_{\P,\mathcal{B}(H)}
  =
  [X^{ 1, (x,u) }_t + \lambda X^{ 1, (x,\tilde{u}) }_t]_{\P,\mathcal{B}(H)}.
\end{equation}
This proves~\eqref{eq:k.linear} in the base case $ k =1 $.
For the induction step $ \{1,2,\ldots,n-1\} \ni k \to k+1 \in \{2,3,\ldots,n\} $ of
item~\eqref{item:lem_derivative:a_priori} 
and~\eqref{eq:k.linear} 
assume that there exists a natural number $ k \in \{1,2,\ldots,n-1\} $ 
such that item~\eqref{item:lem_derivative:a_priori} 
and~\eqref{eq:k.linear} hold for 
$k=1$, $k=2$, $\ldots\,$, $k=k$. 
This ensures that for all 
$ l \in \{0,1\} $, 
$ p \in [2,\infty) $,
$
  \boldsymbol{\delta} \in \deltaset{k+1}
$,
$
  \mathbf{u}
  \in H^{k+2}
$ 
it holds that 
\begin{equation}
\label{eq:hatL.finite}
  \hat{L}^{\boldsymbol{\delta},\mathbf{u},p}_{k+1,l}
  < \infty.
\end{equation}
This and H\"{o}lder's inequality imply that for all 
$ l \in \{0,1\} $, 
$ p \in [2,\infty) $,
$
  \boldsymbol{\delta}=
  (\delta_1, \delta_2, \dots, \delta_{ k + 1 }) \in \deltaset{k+1}
$,
$
  \mathbf{u} = ( u_0, u_1, \dots, u_{ k + 1 } )
  \in H^{k+2}
$,
$
  Y, Z \in \mathcal{L}^p( \P; H )
$,
$ t \in (0,T] $
it holds that
\begin{equation}
\label{eq:L0_k}
  \big\|
    \mathbf{G}^{ \mathbf{u} }_{k+1,l}(t, Y )
    -
    \mathbf{G}^{ \mathbf{u} }_{k+1,l}(t, Z )
  \big\|_{
    \mathcal{L}^p( \P; V_{l,r_l} )
  }
\leq 
  \left| G_l \right|_{
    \Cb{1}( H, V_{l,r_l} )
  }
  \left\| Y - Z \right\|_{
    \mathcal{L}^p( \P; H )
  }
\end{equation}
and 
\begin{equation}
\label{eq:hatL0_k}
\begin{split}
& \big\|
    \mathbf{G}^{ \mathbf{u} }_{k+1,l}( t , 0 )
  \big\|_{\mathcal{L}^p(\P;V_{l,r_l})}
\\&\leq
  \sum_{ 
    \varpi \in \Pi_{ k + 1 }^{ * } 
  }
    \big\|
      G_l^{ ( \#_\varpi ) }( X_t^{ 0, u_0 } )
      \big(
        X_t^{ \#_{I^\varpi_1}, [ \mathbf{u} ]_1^{ \varpi } }
        ,
        X_t^{ \#_{I^\varpi_2}, [ \mathbf{u} ]_2^{ \varpi } }
        ,
        \dots
        ,
        X_t^{ \#_{I^\varpi_{\#_\varpi}}, [\mathbf{u} ]_{ \#_\varpi }^{ \varpi } }
      \big)
    \big\|_{
      \mathcal{L}^p( \P; V_{l,r_l} ) 
    }
\\
& \leq
  \sum_{ \varpi \in \Pi_{ k + 1 }^{ * } }
  | G_l |_{ 
    \Cb{ \#_\varpi }( H, V_{l,r_l} ) 
  }
    \prod_{ i = 1 }^{ \#_\varpi }
    \displaystyle
    \big\|
      X_t^{
        \#_{I^\varpi_i},
        [ \mathbf{u} ]_i^{ \varpi }
      }
    \big\|_{
      \mathcal{L}^{ p \, \#_\varpi }( \P ; H )
    }
\\ & =
  \sum_{ \varpi \in \Pi_{ k + 1 }^{ * } }
  \frac{
    | G_l |_{ 
      \Cb{ \#_\varpi }( H, V_{l,r_l} ) 
    }
  }{
    t^{ ( \delta_1 + \delta_2+ \ldots + \delta_{ k + 1 } ) }
  }
  \Bigg(
    \prod\limits_{ i = 1 }^{ \#_\varpi }
    \frac{
      t^{
        \iota_{ I^{ \varpi }_i 
        }^{ \boldsymbol{\delta} }
      }
      \|
      X_t^{
        \#_{I^\varpi_i},
        [ \mathbf{u} ]_i^{ \varpi }
      }
      \|_{
        \mathcal{L}^{ p \, \#_\varpi }( \P ; H )
      }
    }{
      t^{
        \iota_{ I^{ \varpi }_i 
        }^{ \boldsymbol{\delta} }
        -
        ( 
          \delta_{ I^{ \varpi }_{ i, 1 } } 
          + 
          \delta_{ I^{ \varpi }_{ i, 2 } } 
          +
          \ldots 
          +
          \delta_{ I^{ \varpi }_{ i, \#_{I^{ \varpi }_i} } } 
        )
      }
    }
  \Bigg)
\\ & =
  \sum_{ \varpi \in \Pi_{ k + 1 }^{ * } }
  \frac{
    | G_l |_{ 
      \Cb{ \#_\varpi }( H, V_{l,r_l} ) 
    }
  }{
    t^{ ( \delta_1 + \delta_2+ \ldots + \delta_{ k + 1 } ) }
  }
  \Bigg(
    \prod\limits_{ i = 1 }^{ \#_\varpi }
      t^{
        \iota_{ I^{ \varpi }_i 
        }^{ \boldsymbol{\delta} }
      }
      \|
      X_t^{
        \#_{I^\varpi_i},
        [ \mathbf{u} ]_i^{ \varpi }
      }
      \|_{
        \mathcal{L}^{ p \, \#_\varpi }( \P ; H )
      }
      \,
      t^{
        \1_{[2,\infty)}(\#_{I^\varpi_i})
        \min\{ 1-\alpha,\nicefrac{1}{2}-\beta \}
      }
  \Bigg)
\\ & \leq
  \sum_{ \varpi \in \Pi_{ k + 1 }^{ * } }
  \frac{
    | G_l |_{ 
      \Cb{ \#_\varpi }( H, V_{l,r_l} ) 
    }
  }{
    t^{ ( \delta_1 + \delta_2+ \ldots + \delta_{ k + 1 } ) }
  }
    \prod\limits_{ i = 1 }^{ \#_\varpi }
    \displaystyle
    \Bigg[
      t^{
        \iota_{ I^{ \varpi }_i
        }^{ \boldsymbol{\delta} }
      }
      \big\|
      X_t^{
        \#_{I^\varpi_i},
        [ \mathbf{u} ]_i^{ \varpi }
      }
      \big\|_{
        \mathcal{L}^{ p \, \#_\varpi }( \P ; H )
      } \,
      | T \vee 1 |^{
          \mathbbm{1}_{ [2,\infty) }( \#_{I^{ \varpi }_i} )
          \min\{1-\alpha,\nicefrac{1}{2}-\beta\}
        }
    \Bigg]
\\ & \leq
  \left[
  \sum_{ \varpi \in \Pi_{ k + 1 }^{ * } }
  | G_l |_{ 
    \Cb{ \#_\varpi }( H, V_{l,r_l} ) 
  }
    \prod\limits_{ i = 1 }^{ \#_\varpi }
    \displaystyle
      \sup_{ s \in (0,T] }
    \Big[
      s^{
        \iota_{ I^{ \varpi }_i
        }^{ \boldsymbol{\delta} }
      }
      \|
      X_s^{
        \#_{I^\varpi_i},
        [ \mathbf{u} ]_i^{ \varpi }
      }
      \|_{
        \mathcal{L}^{ p \, \#_\varpi }( \P ; H )
      }
    \Big]
  \right] 
\\ & \quad
  \cdot 
  | T \vee 1 |^{
     \lfloor ( k + 1 ) / 2 \rfloor
     \min\{1-\alpha,\nicefrac{1}{2}-\beta\}
  } \,
  t^{
    - ( \delta_1 + \delta_2 + \ldots + \delta_{ k + 1 } ) 
  }
=
  \hat{L}^{ \boldsymbol{\delta}, \mathbf{u}, p }_{k+1,l}
  \,
  t^{
    - ( \delta_1 + \delta_2 + \ldots + \delta_{k+1} ) 
  }
  < \infty.
\end{split}
\end{equation}
In addition, note that~\eqref{eq:proof.derivatives} and~\eqref{eq:G.def} ensure that for all 
$
  \mathbf{u} = ( u_0, u_1, \ldots, u_{k+1} )
  \in H^{k+2}
$, 
$ t \in [0,T] $
it holds that
\begin{equation}
\label{eq:exist.sol.higher}
\begin{split}
  [X_t^{k+1,\mathbf{u}}]_{ \P, \mathcal{B}(H) }
& =
  \int_0^t
    e^{ ( t - s ) A }
    \mathbf{G}^{ \mathbf{u} }_{k+1,0}( s, X_s^{ k+1, \mathbf{u} } )
  \, {\bf ds}
  +
  \int_0^t
    e^{ ( t - s ) A }
    \mathbf{G}^{ \mathbf{u} }_{k+1,1}( s, X_s^{ k+1, \mathbf{u} } )
  \, \diffns W_s
  .
\end{split}
\end{equation}
Combining~\eqref{eq:hatL.finite}--\eqref{eq:exist.sol.higher} with 
items~(i)--(ii) of Theorem~2.9 in~\cite{AnderssonJentzenKurniawan2016arXiv} 
(with 
$
  H=H
$, 
$
  U=U
$, 
$ T = T $, 
$ \eta = \eta $, 
$ p = p $, 
$ \alpha = \alpha $, 
$
  \hat{\alpha} = \sum_{ i = 1 }^{ k + 1 } \delta_i 
$,
$ \beta = \beta $, 
$
  \hat{\beta} 
 = 
  \sum_{ i = 1 }^{ k + 1 } \delta_i
$,
$
  L_0 = |F|_{ \Cb{1}( H, H_{-\alpha} ) }
$, 
$
  \hat{L}_0 = \hat{L}_{k+1,0}^{ \boldsymbol{\delta}, \mathbf{u}, p }
$,
$
  L_1 = |B|_{ \Cb{1}( H, HS( U, H_{-\beta} ) ) }
$, 
$
  \hat{L}_1 = \hat{L}_{k+1,1}^{ \boldsymbol{\delta}, \mathbf{u}, p }
$,
$
  W=W
$, 
$ A=A $,
$
  \mathbf{F} = 
  \big(
  [0,T]\times\Omega\times H \ni (t,\omega,x) \mapsto
  \mathbf{G}^{ \mathbf{u} }_{k+1,0}(t,\omega,x) \in H_{-\alpha}
  \big)
$,
$
  \mathbf{B} = 
  \big(
  [0,T]\times\Omega\times H \ni (t,\omega,x) \mapsto
  (U \ni u \mapsto \mathbf{G}^{ \mathbf{u} }_{k+1,1}(t,\omega,x)u \in H_{-\beta}) \in HS(U,H_{-\beta})
  \big)
$,
$
  \delta = -\nicefrac{1}{2}
$, 
$
  \lambda = 
  \iota^{ \boldsymbol{\delta} }_\N
$, 
$
  \xi = 
  ( \Omega \ni \omega \mapsto 0 \in H )
$
for 
$
  \mathbf{u}
  \in
  H^{k+2}
$, 
$
  \boldsymbol{\delta} = (\delta_1,\delta_2,\ldots,\delta_{k+1}) \in \deltaset{k+1}
$, 
$ p \in [2,\infty) $ 
in the notation of Theorem~2.9 in~\cite{AnderssonJentzenKurniawan2016arXiv}) 
ensures that for all 
$ p \in [2,\infty) $, 
$
  \boldsymbol{\delta}=
  (\delta_1, \delta_2, \dots, \delta_{ k + 1 }) \in \deltaset{k+1}
$, 
$ \mathbf{u} \in H^{k+2} $ 
it holds that 
\begin{equation}
\begin{split}
&
    \sup_{ t\in(0,T] }
    \big[
    t^{ 
  \iota^{ \boldsymbol{\delta} }_\N
    } \,
    \|
      X_t^{ k+1, \mathbf{u} }
    \|_{
      \mathcal{L}^p( \P; H )
    }
  \big]
\leq
  \Theta_{ A, \eta, p, T }^{
    \alpha , 
    \beta ,
  \iota^{ \boldsymbol{\delta} }_\N
  }( |F|_{ \Cb{1}( H, H_{-\alpha} ) } 
  , 
  |B|_{ \Cb{1}( H, HS( U, H_{-\beta} ) ) } )
\\ & \quad
  \cdot
  \bigg[
      \chi^{ \alpha, T }_{ A, \eta } \,
      \hat{L}_{k+1,0}^{ \boldsymbol{\delta}, \mathbf{u}, p }\,
      T^{ \max\{\beta - \alpha + \nicefrac{1}{2},0\} } \,
      \mathbbm{B}\big(
        1-\alpha
        ,
        1 - \smallsum_{ i = 1 }^{ k + 1 } \delta_i
      \big)
\\&\quad+
      \chi^{ \beta, T }_{ A, \eta } \,
      \hat{L}_{k+1,1}^{ \boldsymbol{\delta}, \mathbf{u}, p }\,
      T^{ \max\{\alpha - \beta - \nicefrac{1}{2},0\} }
        \sqrt{\tfrac{p \, ( p - 1 )}{2}} \,
      \big|
        \mathbbm{B}\big(
          1-2\beta
          ,
          1 
          - 
          2 \smallsum_{ i = 1 }^{ k + 1 } \delta_i
        \big)
      \big|^{ \nicefrac{1}{2} }
  \bigg]
\\ &
\leq
  | T \vee 1 |^{(k+1)}\,
  \Theta_{ A, \eta, p, T }^{
    \alpha , 
    \beta ,
  \iota^{ \boldsymbol{\delta} }_\N
  }( |F|_{ \Cb{1}( H, H_{-\alpha} ) } 
  , 
  |B|_{ \Cb{1}( H, HS( U, H_{-\beta} ) ) } )
\\ & \quad
  \cdot
    \sum_{ \varpi \in \Pi_{ k + 1 }^{ * } }
  \bigg[
    \chi^{ \alpha, T }_{ A, \eta } \,
    \mathbbm{B}\big(
      1 - \alpha
      ,
      \textstyle
      1 - \sum_{ i = 1 }^{ k + 1 } \delta_i
      \displaystyle
    \big)
    \,
    | F |_{ 
      \Cb{ \#_\varpi }( H, H_{-\alpha} ) 
    }
\\&\quad+
        \chi^{ \beta, T }_{ A, \eta } \,
        \sqrt{
        \tfrac{p \, ( p - 1 )}{2} \,
        \mathbbm{B}\big(
          1 - 2 \beta
          ,
          1
          - 
          2 \smallsum_{ i = 1 }^{ k + 1 } \delta_i
        \big)
      }
      \,
    | B |_{ 
      \Cb{ \#_\varpi }( H, HS( U, H_{-\beta} ) ) 
    }
  \bigg]
\\ & \quad
  \cdot
      \prod\limits_{ i = 1 }^{ \#_\varpi }
      \displaystyle
        \sup_{ t \in (0,T] }
      \Big[
        t^{
          \iota_{ I^{ \varpi }_i 
          }^{ \boldsymbol{\delta} }
        } \,
        \|
          X_t^{
            \#_{I^\varpi_i}, 
            [ \mathbf{u} ]_i^{ \varpi }
          }
        \|_{
          \mathcal{L}^{ p \, \#_\varpi }( \P ; H )
        }
      \Big]
  .
\end{split}
\end{equation}
This implies that for all 
$ p \in [2,\infty) $, 
$
  \boldsymbol{\delta}=
  (\delta_1, \delta_2, \dots, \delta_{ k + 1 }) \in \deltaset{k+1}
$ 
it holds that 
\begin{equation}
\label{eq:induction_step_a_priori0}
\begin{split}
&
    \sup_{ \mathbf{u}=(u_0,u_1,\ldots,u_{k+1}) \in (\times^{k+1}_{i=0} H^{[i]}) }
    \sup_{ t\in(0,T] }
    \left[
    \frac{
    t^{ 
  \iota^{ \boldsymbol{\delta} }_\N
    } \,
    \|
      X_t^{ k+1, \mathbf{u} }
    \|_{
      \mathcal{L}^p( \P; H )
    }
    }{
    \prod^{k+1}_{i=1}
    \|u_i\|_{H_{-\delta_i}}
    }
  \right]
\\ &
\leq
  \max\{T^{(k+1)},1\}\,
  \Theta_{ A, \eta, p, T }^{
    \alpha , 
    \beta ,
  \iota^{ \boldsymbol{\delta} }_\N
  }( |F|_{ \Cb{1}( H, H_{-\alpha} ) } 
  , 
  |B|_{ \Cb{1}( H, HS( U, H_{-\beta} ) ) } )
\\ & \quad
  \cdot
    \sum_{ \varpi \in \Pi_{ k + 1 }^{ * } }
  \bigg[
    \chi^{ \alpha, T }_{ A, \eta } \,
    \mathbbm{B}\big(
      1 - \alpha
      ,
      \textstyle
      1 - \sum_{ i = 1 }^{ k + 1 } \delta_i
      \displaystyle
    \big)
    \,
    | F |_{ 
      \Cb{ \#_\varpi }( H, H_{-\alpha} ) 
    }
\\&\quad+
        \chi^{ \beta, T }_{ A, \eta } \,
        \sqrt{
        \tfrac{p \, ( p - 1 )}{2} \,
        \mathbbm{B}\big(
          1 - 2 \beta
          ,
          1
          - 
          2 \smallsum_{ i = 1 }^{ k + 1 } \delta_i
        \big)
      }
      \,
    | B |_{ 
      \Cb{ \#_\varpi }( H, HS( U, H_{-\beta} ) ) 
    }
  \bigg]
\\ & \quad
  \cdot
      \prod\limits_{ I \in \varpi }
      \displaystyle
        \sup_{ \mathbf{u}=(u_i)_{i\in I \cup \{0\}} \in (\times_{i\in I \cup \{0\}} H^{[i]}) }
        \sup_{ t \in (0,T] }
      \bigg[
      \frac{
        t^{
          \iota_I
          ^{ \boldsymbol{\delta} }
        } \,
        \|
          X_t^{
            \#_I, 
            \mathbf{u}
          }
        \|_{
          \mathcal{L}^{ p \, \#_\varpi }( \P ; H )
        }
    }{
      \prod_{i\in I}
      \|u_i\|_{H_{-\delta_i}}
      }
      \bigg]
  .
\end{split}
\end{equation}
This and the induction hypothesis imply item~\eqref{item:lem_derivative:a_priori} in the case $ k+1 $
and thus complete the induction step for 
item~\eqref{item:lem_derivative:a_priori}.
In the next step we note that for all 
$ \lambda \in \R $, 
$ i \in \{1,2,\ldots,k+1\} $, 
$ \varpi \in \Pi^*_{k+1} $
and all 
$
\mathbf{u}^{(m)}=( {u_0}, u_1, \ldots,u_{i-1},
u^{(m)}_i,u_{i+1},
u_{i+2},\ldots,u_{k+1} )
\in H^{k+2}
$, 
$ m \in \{1,2,3\} $, 
with 
$ u^{(3)}_i = u^{(1)}_i + \lambda u^{(2)}_i  $
it holds that there exists a unique natural number
$ j \in \{1,2,\ldots,\#_\varpi\} $ 
such that there exists a natural number 
$ q \in \{1,2,\ldots,\#_{I^\varpi_j}\} $ 
such that for all 
$ l \in \{1,2,\ldots,\#_\varpi\} \setminus \{j\} $ 
it holds that 
\begin{equation}
\label{eq:multlinear.I}
  I^\varpi_{j,q}=i,\qquad
  [\mathbf{u}^{(1)}]^\varpi_l
  =
  [\mathbf{u}^{(2)}]^\varpi_l
  =
  [\mathbf{u}^{(3)}]^\varpi_l,
\end{equation}
and 
\begin{equation}
\label{eq:multlinear.II}
  [ \mathbf{u}^{(3)} ]^\varpi_j
  =
  ( u_0, u_{ I^\varpi_{ j,1 } }, u_{ I^\varpi_{ j,2 } }, \ldots, u_{ I^\varpi_{ j,q-1 } },
  u^{(1)}_i + \lambda u^{(2)}_i, u_{ I^\varpi_{ j,q+1 } }, u_{ I^\varpi_{ j,q+2 } }, 
  \ldots, 
  u_{ I^\varpi_{ j,\#_{ I^\varpi_j } } })
  .
\end{equation}
In addition, observe that for all 
$ \varpi \in \Pi^*_{k+1} $, 
$ j \in \{1,2,\ldots,\#_\varpi\} $ 
it holds that 
\begin{equation}
\label{eq:multlinear.III}
  \#_{I^\varpi_j} \in \{1,2,\ldots,k\}
  .
\end{equation}
Moreover, observe that the induction hypothesis establishes 
that for all 
$ m \in \{1,2,\ldots,k\} $, 
$ p \in [2,\infty) $, 
$ x \in H $, 
$ t \in [0,T] $ 
it holds that 
\begin{equation}
  H^m \ni \mathbf{u} 
  \mapsto [X^{m,(x,\mathbf{u})}_t]_{\P,\mathcal{B}(H)} 
  \in \lpnb{p}{\P}{H}
\end{equation}
is an $m$-linear function.
Combining~\eqref{eq:multlinear.I} and~\eqref{eq:multlinear.II} 
with~\eqref{eq:multlinear.III} hence assures that for all 
$ \lambda \in \R $, 
$ i \in \{1,2,\ldots,k+1\} $, 
$ \varpi \in \Pi^*_{k+1} $, 
$ t \in [0,T] $ 
and all 
$
\mathbf{u}^{(m)}=( {u_0}, u_1, \ldots,u_{i-1},
u^{(m)}_i,u_{i+1},
u_{i+2},\ldots,u_{k+1} )
\in H^{k+2}
$, 
$ m \in \{1,2,3\} $, 
with 
$ u^{(3)}_i = u^{(1)}_i + \lambda u^{(2)}_i  $
it holds that there exists a unique natural number
$ j \in \{ 1,2,\ldots,\#_\varpi \} $ 
such that for all 
$ l \in \{1,2,\ldots,\#_\varpi\} \setminus \{j\} $ 
it holds that 
\begin{equation}
  i \in I^\varpi_j, \qquad
  X^{\#_{I^\varpi_l},[\mathbf{u}^{(1)}]^\varpi_l}_t
  =
  X^{\#_{I^\varpi_l},[\mathbf{u}^{(2)}]^\varpi_l}_t
  =
  X^{\#_{I^\varpi_l},[\mathbf{u}^{(3)}]^\varpi_l}_t,
\end{equation}
and
\begin{equation}
  [
  X^{\#_{I^\varpi_j},[\mathbf{u}^{(1)}]^\varpi_j}_t
  +
  \lambda X^{\#_{I^\varpi_j},[\mathbf{u}^{(2)}]^\varpi_j}_t
  ]_{\P,\mathcal{B}(H)}
  =
  [
  X^{\#_{I^\varpi_j},[\mathbf{u}^{(3)}]^\varpi_j}_t
  ]_{\P,\mathcal{B}(H)}.
\end{equation}
This shows that for all  
$ \lambda \in \R $, 
$ l \in \{ 0,1 \} $, 
$ i \in \{ 1,2,\ldots,k+1 \} $, 
$ t \in [0,T] $ 
and all 
$
\mathbf{u}^{(m)}=( {u_0}, u_1, \ldots,u_{i-1},
$
$
u^{(m)}_i,u_{i+1},
u_{i+2},\ldots,u_{k+1} )
\in H^{k+2}
$, 
$ m \in \{1,2,3\} $, 
with 
$ u^{(3)}_i = u^{(1)}_i + \lambda u^{(2)}_i  $
it holds that there exist 
$ j_\varpi \in \{ 1,2,\ldots,\#_\varpi \} $, 
$ \varpi \in \Pi^*_{k+1} $, 
such that
\begin{equation}
\begin{split}
&\!\!\!\!
\big[
\mathbf{G}^{ \mathbf{u}^{(3)} }_{ k+1,l }
( t, X^{ k+1, \mathbf{u}^{(1)} }_t + \lambda X^{ k+1, \mathbf{u}^{(2)} }_t )
\big]_{\P,\mathcal{B}(V_{l,0})}
\\&\!\!\!\!=
\Big[
G'_l( X^{0,u_0}_t ) ( X^{ k+1, \mathbf{u}^{(1)} }_t + \lambda X^{ k+1, \mathbf{u}^{(2)} }_t )
+
\smallsum_{ \varpi \in \Pi^*_{k+1} }
G^{ (\#_\varpi) }_l( X^{0,u_0}_t )
\big( 
X^{ \#_{I^\varpi_1}, [\mathbf{u}^{(1)}]^\varpi_1 }_t, 
X^{ \#_{I^\varpi_2}, [\mathbf{u}^{(1)}]^\varpi_2 }_t, 
\ldots, 
\\&\!\!\!\!
X^{ \#_{I^\varpi_{j_\varpi-1}}, [\mathbf{u}^{(1)}]^\varpi_{j_\varpi-1} }_t,
X^{ \#_{I^\varpi_{j_\varpi}}, [\mathbf{u}^{(3)}]^\varpi_{j_\varpi} }_t,
X^{ \#_{I^\varpi_{j_\varpi+1}}, [\mathbf{u}^{(1)}]^\varpi_{j_\varpi+1} }_t,
X^{ \#_{I^\varpi_{j_\varpi+2}}, [\mathbf{u}^{(1)}]^\varpi_{j_\varpi+2} }_t,
\ldots,
X^{ \#_{I^\varpi_{\#_\varpi}}, [\mathbf{u}^{(1)}]^\varpi_{ \#_\varpi } }_t 
\big)  
\Big]_{\P,\mathcal{B}(V_{l,0})}
\\&\!\!\!\!=
\Big[
G'_l( X^{0,u_0}_t ) ( X^{ k+1, \mathbf{u}^{(1)} }_t + \lambda X^{ k+1, \mathbf{u}^{(2)} }_t )
+
\smallsum_{ \varpi \in \Pi^*_{k+1} }
G^{ (\#_\varpi) }_l( X^{0,u_0}_t )
\big( 
X^{ \#_{I^\varpi_1}, [\mathbf{u}^{(1)}]^\varpi_1 }_t, 
X^{ \#_{I^\varpi_2}, [\mathbf{u}^{(1)}]^\varpi_2 }_t, 
\ldots, 
\\&\!\!\!\!
X^{ \#_{I^\varpi_{j_\varpi-1}}, [\mathbf{u}^{(1)}]^\varpi_{j_\varpi-1} }_t,
X^{ \#_{I^\varpi_{j_\varpi}}, [\mathbf{u}^{(1)}]^\varpi_{j_\varpi} }_t
+
\lambda X^{ \#_{I^\varpi_{j_\varpi}}, [\mathbf{u}^{(2)}]^\varpi_{j_\varpi} }_t,
X^{ \#_{I^\varpi_{j_\varpi+1}}, [\mathbf{u}^{(1)}]^\varpi_{j_\varpi+1} }_t,
X^{ \#_{I^\varpi_{j_\varpi+2}}, [\mathbf{u}^{(1)}]^\varpi_{j_\varpi+2} }_t,
\ldots,
\\&\!\!\!\!
X^{ \#_{I^\varpi_{\#_\varpi}}, [\mathbf{u}^{(1)}]^\varpi_{ \#_\varpi } }_t 
\big)  
\Big]_{\P,\mathcal{B}(V_{l,0})}
\\&\!\!\!\!=
\big[
\mathbf{G}^{ \mathbf{u}^{(1)} }_{ k+1,l }( t, X^{ k+1, \mathbf{u}^{(1)} }_t )
+
\lambda \, \mathbf{G}^{ \mathbf{u}^{(2)} }_{ k+1,l }( t, X^{ k+1, \mathbf{u}^{(2)} }_t )
\big]_{\P,\mathcal{B}(V_{l,0})}
.
\end{split}
\end{equation}
This, \eqref{eq:exist.sol.higher}, and Lemma~3.1 in Jentzen \& Pu\v{s}nik~\cite{JentzenPusnik2016arXiv} 
(with 
$ (\Omega,\mathcal{F},\mu) = (\Omega,\mathcal{F},\P) $, 
$ T=t $, 
$
  Y(\omega,s)
  =
  \big\|
  e^{(t-s)A}
  \mathbf{G}^{ \mathbf{u}^{(3)} }_{ k+1,l }
  ( s, \omega, X^{ k+1, \mathbf{u}^{(1)} }_s(\omega) + \lambda X^{ k+1, \mathbf{u}^{(2)} }_s(\omega) )
  -
  e^{(t-s)A}
  \mathbf{G}^{ \mathbf{u}^{(1)} }_{ k+1,l }( s, \omega, X^{ k+1, \mathbf{u}^{(1)} }_s(\omega) )
  -
  \lambda e^{(t-s)A}
  \mathbf{G}^{ \mathbf{u}^{(2)} }_{ k+1,l }
$
$
  ( s, \omega, X^{ k+1, \mathbf{u}^{(2)} }_s(\omega) )
  \big\|^{(l+1)}_{V_{l,0}}
$,
$ Z(\omega,s) = 0 $ 
for 
$ s \in [0,t] $, 
$ t \in (0,T] $, 
$ \omega \in \Omega $, 
$ l \in \{0,1\} $, 
$
\mathbf{u}^{(3)}=( {u_0}, u_1, \ldots,u_{i-1},
u^{(1)}_i+\lambda u^{(2)}_i,u_{i+1},
u_{i+2},\ldots,
u_{k+1} )
$, 
$
\mathbf{u}^{(2)}=( {u_0}, u_1, \ldots,u_{i-1},
u^{(2)}_i,u_{i+1},
u_{i+2},\ldots,
u_{k+1} )
$, 
$
\mathbf{u}^{(1)}=( {u_0}, u_1, \ldots,u_{i-1},
u^{(1)}_i,u_{i+1},
u_{i+2},\ldots,
u_{k+1} )
\in H^{k+2}
$, 
$ i \in \{ 1,2,\ldots,k+1 \} $, 
$ \lambda \in \R $ 
in the notation of Lemma~3.1 in Jentzen \& Pu\v{s}nik~\cite{JentzenPusnik2016arXiv})
prove that for all 
$ \lambda \in \R $,  
$ i \in \{ 1,2,\ldots,k+1 \} $, 
$ t \in [0,T] $ 
and all 
$
\mathbf{u}^{(m)}=( {u_0}, u_1, \ldots,
u_{i-1},
$
$
u^{(m)}_i,u_{i+1},u_{i+2},\ldots,u_{k+1} )
\in H^{k+2}
$, 
$ m \in \{1,2,3\} $, 
with 
$ u^{(3)}_i = u^{(1)}_i + \lambda u^{(2)}_i  $
it holds that 
\begin{equation}
\begin{split}
[X^{ k+1, \mathbf{u}^{(1)} }_t
+
\lambda X^{ k+1, \mathbf{u}^{(2)} }_t]_{\P, \mathcal{B}(H)}
&=
\int^t_0
e^{(t-s)A}
\mathbf{G}^{\mathbf{u}^{(3)}}_{k+1,0}
( s, X^{ k+1, \mathbf{u}^{(1)} }_s + \lambda X^{ k+1, \mathbf{u}^{(2)} }_s )
\, {\bf ds}
\\&\quad+
\int^t_0
e^{(t-s)A}
\mathbf{G}^{\mathbf{u}^{(3)}}_{k+1,1}
( s, X^{ k+1, \mathbf{u}^{(1)} }_s + \lambda X^{ k+1, \mathbf{u}^{(2)} }_s )
\,dW_s
.
\end{split}
\end{equation}
This and item~\eqref{item:lem_derivative:existence} imply for all
$ \lambda \in \R $,  
$ i \in \{ 1,2,\ldots,k+1 \} $, 
$ t \in [0,T] $
and all 
$
\mathbf{u}^{(m)}=( {u_0}, u_1, \ldots,
u_{i-1},
$
$
u^{(m)}_i,
u_{i+1},u_{i+2},\ldots,u_{k+1} )
\in H^{k+2}
$, 
$ m \in \{1,2,3\} $, 
with 
$ u^{(3)}_i = u^{(1)}_i + \lambda u^{(2)}_i  $
that 
\begin{equation}
[X^{ k+1, \mathbf{u}^{(3)} }_t]_{\P,\mathcal{B}(H)}
=
[X^{ k+1, \mathbf{u}^{(1)} }_t
+
\lambda X^{ k+1, \mathbf{u}^{(2)} }_t]_{\P,\mathcal{B}(H)}.
\end{equation}
This proves~\eqref{eq:k.linear} in the case $ k+1 $ 
and hence completes the induction step for~\eqref{eq:k.linear}.
Induction thus completes the proof of 
item~\eqref{item:lem_derivative:a_priori} and~\eqref{eq:k.linear}.

Combining~\eqref{eq:k.linear} with item~\eqref{item:lem_derivative:a_priori} 
establishes item~\eqref{item:thm_derivative}.
Next we prove item~\eqref{item:lem_derivative:a_priori_Lip}. 
We first note that item~\eqref{item:lem_derivative:a_priori} implies that 
for all 
$ k \in \{1,2,\ldots,n\} $, 
$ l \in \{0,1\} $, 
$
  p \in (0,\infty)
$, 
$ \boldsymbol{\delta} \in \deltaset{k} $, 
$
  \varpi \in \mathcal{P}\big(
  \mathcal{P}(\{ 1, 2, \ldots, k \})
  \setminus \{\emptyset\}
  \big)
$, 
$ \mathbf{u} \in H^{k+1} $
it holds that 
\begin{equation}
\label{eq:L.finite}
  L^{ \boldsymbol{\delta} }_{\varpi,p}
  +
  \hat{L}^{\boldsymbol{\delta},\mathbf{u},p}_{k,l}
  < \infty.
\end{equation}
We next apply
the Burkholder-Davis-Gundy type inequality 
in Lemma~7.7 in Da Prato \& Zabczyk~\cite{dz92}, 
\eqref{eq:exist.cond}, \eqref{eq:exist.sol}, \eqref{eq:L0_k}, 
\eqref{eq:hatL0_k}, \eqref{eq:exist.sol.higher}, 
and 
Proposition~2.7 in~\cite{AnderssonJentzenKurniawan2016arXiv} 
(with
$
H=H
$, 
$
U=U
$, 
$ T = T $, 
$ \eta = \eta $, 
$ p = p $, 
$ \alpha = \alpha $, 
$
\hat{\alpha} = 0
$,
$ \beta = \beta $, 
$
\hat{\beta} 
= 
0
$,
$L_0=|F|_{\Cb{1}(H,H_{-\alpha})}$, 
$
  \hat{L}_0
  =
  \hat{L}^{\mathbf{0}_k,\mathbf{y},p}_{k,0}
$, 
$L_1=|B|_{\Cb{1}(H,HS(U,H_{-\beta}))}$, 
$
  \hat{L}_1
  =
  \hat{L}^{\mathbf{0}_k,\mathbf{y},p}_{k,1}
$, 
$
W=W
$, 
$ A=A $,
$
  \mathbf{F} = 
  \big(
  [0,T]\times\Omega\times H \ni (t,\omega,z) \mapsto
  \mathbf{G}_{k,0}^{\mathbf{y}}(t,\omega,z) \in H_{-\alpha}
  \big)
$, 
$
  \mathbf{B} = 
  \big(
  [0,T]\times\Omega\times H \ni (t,\omega,z) \mapsto
  (U\ni u \mapsto \mathbf{G}_{k,1}^{\mathbf{y}}(t,\omega,z)u \in H_{-\beta}) \in HS(U,H_{-\beta})
  \big)
$,  
$ \delta = 0 $, 
$ Y^1 = X^{k,\mathbf{x}} $, $ Y^2 = X^{k,\mathbf{y}} $,
$ \lambda = \lambda $  
for 
$
  \mathbf{x} = ( x, u_1, u_2, \ldots, u_k )
$, 
$
  \mathbf{y} = ( y, u_1, u_2, \ldots, u_k )
  \in H^{k+1}
$, 
$ \lambda \in ( -\infty, \nicefrac{1}{2} ) $, 
$ p \in [2,\infty) $, 
$ k \in \{ 1, 2, \dots, n \} $ 
in the notation of Proposition~2.7 in~\cite{AnderssonJentzenKurniawan2016arXiv})  
to obtain that for all
$ k \in \{ 1, 2, \dots, n \} $, 
$ p \in [2,\infty) $,
$ \lambda \in ( -\infty, \nicefrac{1}{2} ) $, 
$
  \mathbf{x} = ( x, u_1, u_2, \ldots, u_k )
$, 
$
  \mathbf{y} = ( y, u_1, u_2, \ldots, u_k )
  \in H^{k+1}
$
it holds that 
\begin{equation}
\label{eq:initial.perturb.gronwall}
\begin{split}
& 
  \sup_{ t \in (0,T] }
  t^\lambda
  \,
  \big\|
    X_t^{ k, \mathbf{x} }
    -
    X_t^{ k, \mathbf{y} }
  \big\|_{
    \mathcal{L}^p( \P ; H )
  }
\leq
  \Theta_{ A, \eta, p, T }^{ \alpha, \beta, \lambda }
  \big(
    | F |_{ \Cb{1}( H , H_{ - \alpha } ) }
    ,
    | B |_{ \Cb{1}( H , HS( U , H_{ - \beta } ) ) }
  \big)
\\ & \quad
  \cdot
  \sup_{ t \in (0,T] }
  \Bigg[
    t^\lambda \,
    \bigg\|
      \int_0^t
        e^{ ( t - s ) A }
        \left(
          \mathbf{G}_{ k, 0 }^{ \mathbf{x} }( s , X_s^{ k, \mathbf{x} } )
          -
          \mathbf{G}_{ k, 0 }^{ \mathbf{y} }( s , X_s^{ k, \mathbf{x} } )
        \right)
      { \bf ds }
\\ & \quad
      +
      \int_0^t
        e^{ ( t - s ) A }
        \left(
          \mathbf{G}_{ k, 1 }^{ \mathbf{x} }( s , X_s^{ k, \mathbf{x} } )
          -
          \mathbf{G}_{ k, 1 }^{ \mathbf{y} }( s , X_s^{ k, \mathbf{x} } )
        \right)
      \diffns W_s
    \bigg\|_{ L^p( \P; H ) }
  \Bigg]
\\&\leq
  \Theta_{ A, \eta, p, T }^{ 
    \alpha, \beta, \lambda
  }\big(
    | F |_{ \Cb{1}( H , H_{ - \alpha } ) }
    ,
    | B |_{ \Cb{1}( H , HS( U , H_{ - \beta } ) ) }
  \big)
\\ & \quad
  \cdot
  \Bigg[
    \sup_{ t \in (0,T] }
    \Bigg\{
    t^\lambda \,
    \chi^{ \alpha, T }_{ A, \eta }
    \int_0^t
      \tfrac{
      \|
        \mathbf{G}_{ k, 0 }^{ \mathbf{x} }( s , X_s^{ k, \mathbf{x} } )
        -
        \mathbf{G}_{ k, 0 }^{ \mathbf{y} }( s , X_s^{ k, \mathbf{x} } )
      \|_{\mathcal{L}^p(\P;H_{-\alpha})}
      }{
      ( t - s )^{ \alpha }
      }
    \, \diffns s
    \Bigg\}
\\ & \quad
    +
    \sup_{
      t \in (0,T]
    }
    \left\{
    t^\lambda \,
    \chi^{ \beta, T }_{ A, \eta }
    \left[
      \tfrac{ p \, ( p - 1 ) }{ 2 }  
      \int_0^t
      \tfrac{
      \|
        \mathbf{G}_{ k, 1 }^{ \mathbf{x} }( s , X_s^{ k, \mathbf{x} } )
        -
        \mathbf{G}_{ k, 1 }^{ \mathbf{y} }( s , X_s^{ k, \mathbf{x} } )
      \|_{
        \mathcal{L}^p( \P ; HS( U, H_{ - \beta } ) ) 
      }^2
      }{
      ( t - s )^{ 2 \beta }
      }
      \,\diffns s
    \right]^{ \nicefrac{1}{2} }
    \right\}
  \Bigg] .
\end{split}
\end{equation}
Moreover, observe that~\eqref{eq:G.def} ensures that for all 
$ k \in \{ 1, 2, \dots, n \} $,
$ l \in \{ 0, 1 \} $,
$
  \mathbf{x} = 
  ( x, u_1 , u_2, \dots, u_k )
$,
$
  \mathbf{y} = 
  ( y, u_1 , u_2, 
$
$
  \dots, u_k ) 
  \in H^{ k+1 }
$, 
$ t \in [0,T] $
it holds that
\begin{equation}
\label{eq:decompose.perturb}
\begin{split}
&\!\! 
  \mathbf{G}_{ k, l }^{ \mathbf{x} }( t , X_t^{ k, \mathbf{x} } )
  -
  \mathbf{G}_{ k, l }^{ \mathbf{y} }( t , X_t^{ k, \mathbf{x} } )
  =
  \big(
    G_l'( X_t^{ 0, x } ) 
    - 
    G_l'( X_t^{ 0, y } )
  \big)
  X_t^{ k, \mathbf{x} }
\\&\!\!+
  {\smallsum_{ \varpi \in \Pi_k^{ * } }}
  \displaystyle
  \Big[
    \big(
      G^{ ( \#_\varpi ) }_l( X_t^{ 0, x } )
      -
      G^{ ( \#_\varpi ) }_l( X_t^{ 0, y } )
    \big)
    \big(
      X_t^{ \#_{I^\varpi_1}, [ \mathbf{y} ]_1^\varpi}
      ,
      X_t^{ \#_{I^\varpi_2}, [ \mathbf{y} ]_2^\varpi}
      ,
      \dots
      ,
      X_t^{ \#_{I^\varpi_{\#_\varpi}}, [ \mathbf{y} ]_{ \#_\varpi }^{ \varpi } }
    \big)
\\ &    
    \!\!+
    \smallsum^{\#_\varpi}_{i=1}
    G^{ ( \#_\varpi ) }_l( X_t^{ 0, x } )
    \big(
      X_t^{
        \#_{I^\varpi_1}, 
        [ \mathbf{x} ]_1^\varpi
      }
      ,
      X_t^{
        \#_{I^\varpi_2}, 
        [ \mathbf{x} ]_2^\varpi
      }
      ,
      \ldots
      ,
      X_t^{
        \#_{I^\varpi_{i-1}}, 
        [ \mathbf{x} ]_{i-1}^\varpi
      }
      ,
      X_t^{
        \#_{I^\varpi_i}, 
        [ \mathbf{x} ]_i^\varpi
      }
      -
      X_t^{
        \#_{I^\varpi_i}, 
        [ \mathbf{y} ]_i^\varpi
      }
      ,
\\&\quad
      X_t^{
        \#_{I^\varpi_{i+1}}, 
        [ \mathbf{y} ]_{i+1}^\varpi
      }
      ,
      X_t^{ \#_{I^\varpi_{i+2}}, [ \mathbf{y} ]_{i+2}^\varpi}
      ,
      \dots
      ,
      X_t^{ 
        \#_{I^\varpi_{\#_\varpi}}, 
        [ \mathbf{y} ]_{ \#_\varpi }^\varpi
      }
    \big)
  \Big]
  .
\end{split}
\end{equation}
Next note that H\"{o}lder's inequality ensures that for all 
$ k \in \{ 1, 2, \dots, n \} $,
$ l \in \{ 0, 1 \} $,
$ p \in [ 2 , \infty ) $,  
$
  \boldsymbol{\delta}=
  (\delta_1, \delta_2, \dots, \delta_k) \in \deltaset{k} 
$,
$ \varpi \in \Pi_k $,
$
  x \in H
$, 
$
  \mathbf{y} = 
  ( y, u_1 , u_2, \dots, u_k ) 
  \in \times^k_{i=0} H^{[i]}
$, 
$
  t \in (0,T]
$
it holds that
\begin{equation}
\begin{split}
&
    \frac{
    \big\|
    \big(
      G^{ ( \#_\varpi ) }_l( X_t^{ 0, x } )
      -
      G^{ ( \#_\varpi ) }_l( X_t^{ 0, y } )
    \big)
    \big(
      X_t^{ \#_{I^\varpi_1}, [ \mathbf{y} ]_1^\varpi}
      ,
      X_t^{ \#_{I^\varpi_2}, [ \mathbf{y} ]_2^\varpi}
      ,
      \dots
      ,
      X_t^{ \#_{I^\varpi_{\#_\varpi}}, [ \mathbf{y} ]_{ \#_\varpi }^{ \varpi } }
    \big)
    \big\|_{\lpn{p}{\P}{V_{l,r_l}}}
    }{ \prod^k_{ i=1 } \|u_i\|_{ H_{-\delta_i} } }
\\&\leq
    \big\|
      G_l^{ ( \#_\varpi ) }( X_t^{ 0, x } )
      -
      G_l^{ ( \#_\varpi ) }( X_t^{ 0, y } )
    \big\|_{
      \mathcal{L}^{ p ( \#_\varpi + 1 ) }(
        \P ; L^{ ( \#_\varpi ) }( H, V_{l,r_l} ) 
      )
    }
    \prod^{ \#_\varpi }_{ i=1 }
    \frac{
      \|X^{\#_{I^\varpi_i},[\mathbf{y}]^\varpi_i}_t\|_{ \lpn{p(\#_\varpi+1)}{\P}{H} }
    }{
    \Big[
    \prod^{ \#_{I^\varpi_i} }_{ m=1 } \|u_{I^\varpi_{i,m}}\|_{ H_{-\delta_{I^\varpi_{i,m}}} }
    \Big]
    }
\\&=
    \big\|
      G_l^{ ( \#_\varpi ) }( X_t^{ 0, x } )
      -
      G_l^{ ( \#_\varpi ) }( X_t^{ 0, y } )
    \big\|_{
      \mathcal{L}^{ p ( \#_\varpi + 1 ) }(
        \P ; L^{ ( \#_\varpi ) }( H, V_{l,r_l} ) 
      )
    }
    \Bigg[
      \prod_{ I \in \varpi }
      \frac{1}{
        t^{ \iota^{\boldsymbol{\delta}}_I }
      }
    \Bigg]
    \prod^{ \#_\varpi }_{ i=1 }
    \frac{
      t^{ \iota^{\boldsymbol{\delta}}_{ I^\varpi_i } }
      \|X^{\#_{I^\varpi_i},[\mathbf{y}]^\varpi_i}_t\|_{ \lpn{p(\#_\varpi+1)}{\P}{H} }
    }{
    \Big[
    \prod^{ \#_{I^\varpi_i} }_{ m=1 } \|u_{I^\varpi_{i,m}}\|_{ H_{-\delta_{I^\varpi_{i,m}}} }
    \Big]
    }
\\&\leq
  \frac{
    | T \vee 1 |^{ \lfloor k/2 \rfloor \min\{ 1-\alpha, \nicefrac{1}{2} - \beta 
    \} }
  }{
    t^{
      ( \delta_1 + \delta_2 + \ldots + \delta_k )
    }
  }
  \,
    L^{ \boldsymbol{\delta} }_{\varpi, \, p( \#_\varpi + 1 )} \,
    \big\|
      G_l^{ ( \#_\varpi ) }( X_t^{ 0, x } )
      -
      G_l^{ ( \#_\varpi ) }( X_t^{ 0, y } )
    \big\|_{
      \mathcal{L}^{ p ( \#_\varpi + 1 ) }(
        \P ; L^{ ( \#_\varpi ) }( H, V_{l,r_l} ) 
      )
    }.
\end{split}
\end{equation}
In addition, H\"{o}lder's inequality establishes that for all 
$ k \in \{ 1, 2, \dots, n \} $,
$ l \in \{ 0, 1 \} $,
$ p \in [ 2 , \infty ) $,  
$ \gamma \in [0,\infty) $, 
$
  \boldsymbol{\delta}=
  (\delta_1, \delta_2, \dots, \delta_k) \in \deltaset{k} 
$,
$ \varpi \in \Pi^*_k $,
$ j \in \{1,2,\ldots,\#_\varpi\} $, 
$
  \mathbf{x} = 
  ( x, u_1 , u_2, \dots, u_k ) 
$, 
$
  \mathbf{y} = 
  ( y, u_1 , u_2, \dots, u_k ) 
  \in \times^k_{i=0} H^{[i]}
$, 
$
  t \in (0,T]
$ 
it holds that 
\begin{equation}
\label{eq:decompose.perturb.bound}
\begin{split}
&
    \tfrac{
    1
    }{ \prod^k_{ i=1 } \|u_i\|_{ H_{-\delta_i} } } \,
    \big\|
    G^{ ( \#_\varpi ) }_l( X_t^{ 0, x } )
    \big(
      X_t^{
        \#_{I^\varpi_1},
        [ \mathbf{x} ]_1^\varpi
      }
      ,
      X_t^{
        \#_{I^\varpi_2},
        [ \mathbf{x} ]_2^\varpi
      }
      ,
      \ldots
      ,
      X_t^{
        \#_{I^\varpi_{j-1}},
        [ \mathbf{x} ]_{j-1}^\varpi
      }
      ,
      X_t^{
        \#_{I^\varpi_j},
        [ \mathbf{x} ]_j^\varpi
      }
      -
      X_t^{
        \#_{I^\varpi_j},
        [ \mathbf{y} ]_j^\varpi
      }
      ,
\\&\quad
      X_t^{
        \#_{I^\varpi_{j+1}},
        [ \mathbf{y} ]_{j+1}^\varpi
      }
      ,
      X_t^{ \#_{I^\varpi_{j+2}}, [ \mathbf{y} ]_{j+2}^\varpi}
      ,
      \dots
      ,
      X_t^{ 
        \#_{I^\varpi_{\#_\varpi}},
        [ \mathbf{y} ]_{ \#_\varpi }^\varpi
      }
    \big)
    \big\|_{\lpn{p}{\P}{V_{l,r_l}}}
\\&\leq
  |G_l|_{ \Cb{\#_\varpi}( H, V_{l,r_l} ) }
  \Bigg[
    \prod^{j-1}_{ i=1 }
    \frac{
      \|X^{\#_{I^\varpi_i},[\mathbf{x}]^\varpi_i}_t\|_{ \lpn{p \#_\varpi}{\P}{H} }
    }{
    \prod^{ \#_{I^\varpi_i} }_{ m=1 } \|u_{I^\varpi_{i,m}}\|_{ H_{-\delta_{I^\varpi_{i,m}}} }
    }
  \Bigg]
  \Bigg[
    \prod^{ \#_\varpi }_{i=j+1}
    \frac{
      \|X^{\#_{I^\varpi_i},[\mathbf{y}]^\varpi_i}_t\|_{ \lpn{p \#_\varpi}{\P}{H} }
    }{
    \prod^{ \#_{I^\varpi_i} }_{ m=1 } \|u_{I^\varpi_{i,m}}\|_{ H_{-\delta_{I^\varpi_{i,m}}} }
    }
  \Bigg]
\\&\quad\cdot
    \frac{
      \|X^{\#_{I^\varpi_j},[\mathbf{x}]^\varpi_j}_t-X^{\#_{I^\varpi_j},[\mathbf{y}]^\varpi_j}_t\|_{ \lpn{p \#_\varpi}{\P}{H} }
    }{
    \prod^{ \#_{I^\varpi_j} }_{ m=1 } \|u_{I^\varpi_{j,m}}\|_{ H_{-\delta_{I^\varpi_{j,m}}} }
    }
\\&=
  |G_l|_{ \Cb{\#_\varpi}( H, V_{l,r_l} ) }
  \Bigg[
    \frac{1}{
      t^{ \gamma + \iota^{ (\boldsymbol{\delta},0) }_{ I^\varpi_j \cup \{k+1\} } }
    }
    \prod_{ I \in \varpi \setminus \{I^\varpi_j\} }
    \frac{1}{
    t^{ \iota^{ \boldsymbol{\delta} }_I }
    }
  \Bigg]
  \Bigg[
    \prod^{j-1}_{ i=1 }
    \frac{
      t^{ \iota^{ \boldsymbol{\delta} }_{ I^\varpi_i } }
      \|X^{\#_{I^\varpi_i},[\mathbf{x}]^\varpi_i}_t\|_{ \lpn{p \#_\varpi}{\P}{H} }
    }{
    \prod^{ \#_{I^\varpi_i} }_{ m=1 } \|u_{I^\varpi_{i,m}}\|_{ H_{-\delta_{I^\varpi_{i,m}}} }
    }
  \Bigg]
\\&\quad\cdot
  \Bigg[
    \prod^{ \#_\varpi }_{i=j+1}
    \frac{
      t^{ \iota^{ \boldsymbol{\delta} }_{ I^\varpi_i } }
      \|X^{\#_{I^\varpi_i},[\mathbf{y}]^\varpi_i}_t\|_{ \lpn{p \#_\varpi}{\P}{H} }
    }{
    \prod^{ \#_{I^\varpi_i} }_{ m=1 } \|u_{I^\varpi_{i,m}}\|_{ H_{-\delta_{I^\varpi_{i,m}}} }
    }
  \Bigg]
    \frac{
      t^{ \gamma + \iota^{ (\boldsymbol{\delta},0) }_{ I^\varpi_j \cup \{k+1\} } }
      \|X^{\#_{I^\varpi_j},[\mathbf{x}]^\varpi_j}_t-X^{\#_{I^\varpi_j},[\mathbf{y}]^\varpi_j}_t\|_{ \lpn{p \#_\varpi}{\P}{H} }
    }{
    \prod^{ \#_{I^\varpi_j} }_{ m=1 } \|u_{I^\varpi_{j,m}}\|_{ H_{-\delta_{I^\varpi_{j,m}}} }
    }
\\&\leq
  \frac{
    | T \vee 1 |^{ \lceil k/2 \rceil \min\{ 1-\alpha, \nicefrac{1}{2} - \beta 
    \} }
  }{
    t^{
      ( \gamma + \delta_1 + \delta_2 + \ldots + \delta_k )
    }
  }
    \sup_{\mathbf{v}=(v_i)_{i\in I^\varpi_j}\in (\nzspace{H})^{\#_{I^\varpi_j}} }\,
      \sup_{ s \in (0,T] }
      \Bigg\{
  |G_l|_{ \Cb{\#_\varpi}( H, V_{l,r_l} ) } \,
    L^{ \boldsymbol{\delta} }_{ \varpi \setminus \{I^\varpi_j\}, \, p\#_\varpi }
\\&\quad\cdot
      \frac{
        s^{
          \gamma + \iota_{ I^\varpi_j \cup \{k+1\} }^{ (\boldsymbol{\delta}, 0) } 
        }
        \|
          X_s^{ \#_{I^\varpi_j}, ( x, \mathbf{v} ) }
          -
          X_s^{ \#_{I^\varpi_j}, ( y, \mathbf{v} ) }
        \|_{
          \mathcal{L}^{ p \#_\varpi }( \P ; H )
        }
      }{
        \prod_{ i \in I^\varpi_j }
        \| v_i \|_{ H_{ - \delta_i } }
      }
    \Bigg\}
    .
\end{split}
\end{equation}
Combining~\eqref{eq:decompose.perturb}--\eqref{eq:decompose.perturb.bound} yields that for all 
$ k \in \{ 1, 2, \dots, n \} $,
$ l \in \{ 0, 1 \} $,
$ p \in [ 2 , \infty ) $,  
$ \gamma \in [0,\infty) $, 
$
  \boldsymbol{\delta}=
  (\delta_1, \delta_2, \dots, \delta_k) \in \deltaset{k} 
$,
$
  \mathbf{x} = 
  ( x, u_1 , u_2, \dots, u_k ) 
$, 
$
  \mathbf{y} = 
  ( y, u_1 , u_2, \dots, u_k ) 
  \in \times^k_{i=0} H^{[i]}
$, 
$
  t \in (0,T]
$
it holds that
\begin{equation}
\begin{split}
& 
  \frac{
    \|
      \mathbf{G}_{ k, l }^{ \mathbf{x} }( t , X_t^{ k, \mathbf{x} } )
      -
      \mathbf{G}_{ k, l }^{ \mathbf{y} }( t , X_t^{ k, \mathbf{x} } )
    \|_{
      \mathcal{L}^p( \P; V_{l,r_l} )
    }
  }{
    \prod_{ i = 1 }^k \| u_i \|_{ H_{ - \delta_i } }
  }
  \leq
  \frac{
    | T \vee 1 |^{ \lceil k/2 \rceil \min\{ 1-\alpha, \nicefrac{1}{2} - \beta 
    \} }
  }{
    t^{
      ( \delta_1 + \delta_2 + \ldots + \delta_k )
    }
  }
\\&\quad\cdot  
  \Bigg(
  \sum_{
    \varpi \in \Pi_k
  }
  \Big[
    L^{ \boldsymbol{\delta} }_{\varpi, \, p( \#_\varpi + 1 )} \,
    \big\|
      G_l^{ ( \#_\varpi ) }( X_t^{ 0, x } )
      -
      G_l^{ ( \#_\varpi ) }( X_t^{ 0, y } )
    \big\|_{
      \mathcal{L}^{ p ( \#_\varpi + 1 ) }(
        \P ; L^{ ( \#_\varpi ) }( H, V_{l,r_l} ) 
      )
    }
    \Big]
\\&\quad+
    \frac{1}{t^\gamma}
  \sum_{
    \varpi \in \Pi^*_k
  }
    \sum\limits_{
      I \in \varpi
    }
    \sup_{\mathbf{v}=(v_i)_{i\in I}\in (\nzspace{H})^{\#_I} }\,
      \sup_{ s \in (0,T] }
      \Bigg\{
    |
      G_l
    |_{
      \Cb{ \#_\varpi }( H, V_{l,r_l} ) 
      } \,
    L^{ \boldsymbol{\delta} }_{ \varpi \setminus \{I\}, \, p\#_\varpi } \,
\\&\quad\cdot
      \frac{
        s^{
          \gamma + \iota_{ I \cup \{k+1\} }^{ (\boldsymbol{\delta}, 0) } 
        }
        \|
          X_s^{ \#_I, ( x, \mathbf{v} ) }
          -
          X_s^{ \#_I, ( y, \mathbf{v} ) }
        \|_{
          \mathcal{L}^{ p \#_\varpi }( \P ; H )
        }
      }{
        \prod_{ i \in I }
        \| v_i \|_{ H_{ - \delta_i } }
      }
    \Bigg\}\Bigg)
  .
\end{split}
\end{equation}
This and Minkowski's inequality imply that for all 
$ k \in \{ 1, 2, \dots, n \} $,
$ l \in \{ 0, 1 \} $,
$ p \in [ 2 , \infty ) $,  
$
  \boldsymbol{\delta}=
  (\delta_1, \delta_2, \dots, \delta_k) \in \deltaset{k} 
$,
$ \gamma \in [0,\nicefrac{1}{2}-\sum^k_{i=1} \delta_i) $, 
$
  \mathbf{x} = 
  ( x, u_1 , u_2, \dots, u_k ) 
$, 
$
  \mathbf{y} = 
  ( y, u_1 , u_2, \dots, u_k ) 
  \in \times^k_{i=0} H^{[i]}
$, 
$ t \in (0,T] $
it holds that
\begin{equation}
\begin{split}
&
  \Bigg[
  \int^t_0
  \Bigg(
  \frac{
    \|
      \mathbf{G}_{ k, l }^{ \mathbf{x} }( s , X_s^{ k, \mathbf{x} } )
      -
      \mathbf{G}_{ k, l }^{ \mathbf{y} }( s , X_s^{ k, \mathbf{x} } )
    \|_{
      \mathcal{L}^p( \P; V_{l,r_l} )
    }
  }{
    (t-s)^{r_l}
    \prod^k_{i=1}
    \|u_i\|_{H_{-\delta_i}}
  }
  \Bigg)^{ \!(l+1) }
  \, ds
  \Bigg]^{\nicefrac{1}{(l+1)}}
  \leq
  | T \vee 1 |^{ \lceil k/2 \rceil \min\{ 1-\alpha, \nicefrac{1}{2} - \beta 
      \} }
\\&\cdot
  \Bigg[
  \int^t_0
  \Bigg(
    \sum_{ \varpi \in \Pi_k }
    \frac{
    L^{ \boldsymbol{\delta} }_{\varpi, \, p( \#_\varpi + 1 )} \,
    \|
      G_l^{ ( \#_\varpi ) }( X_s^{ 0, x } )
      -
      G_l^{ ( \#_\varpi ) }( X_s^{ 0, y } )
    \|_{
      \mathcal{L}^{ p ( \#_\varpi + 1 ) }(
        \P ; L^{ ( \#_\varpi ) }( H, V_{l,r_l} ) 
      )
    }
    }{
      s^{(\delta_1+\delta_2+\ldots+\delta_k)} \,
      (t-s)^{r_l}
    }
\\&+
    \sum_{ \varpi \in \Pi^*_k }
    \sum_{
      I \in \varpi
    }
    \frac{
    1
    }{
      s^{(\gamma+\delta_1+\delta_2+\ldots+\delta_k)} \,
      (t-s)^{r_l}
    }
    \sup_{\mathbf{v}=(v_i)_{i\in I}\in (\nzspace{H})^{\#_I} }\,
      \sup_{ w \in (0,T] }
      \Bigg\{
    |G_l|_{ \Cb{\#_\varpi}( H, V_{l,r_l} ) } \,
    L^{ \boldsymbol{\delta} }_{ \varpi \setminus \{I\}, \, p\#_\varpi } \,
\\&\cdot
      \frac{
        w^{
          \gamma + \iota_{ I \cup \{k+1\} }^{ (\boldsymbol{\delta}, 0) } 
        }
        \|
          X_w^{ \#_I, ( x, \mathbf{v} ) }
          -
          X_w^{ \#_I, ( y, \mathbf{v} ) }
        \|_{
          \mathcal{L}^{ p \#_\varpi }( \P ; H )
        }
      }{
        \prod_{ i \in I }
        \| v_i \|_{ H_{ - \delta_i } }
      }
    \Bigg\}
  \Bigg)^{\!(l+1)}
  \, ds
  \Bigg]^{\nicefrac{1}{(l+1)}}
\\&\leq
  | T \vee 1 |^{ \lceil k/2 \rceil \min\{ 1-\alpha, \nicefrac{1}{2} - \beta 
        \} }
\\&\cdot
  \Bigg(
    \sum_{ \varpi \in \Pi_k }
    L^{ \boldsymbol{\delta} }_{\varpi, \, p( \#_\varpi + 1 )}
  \Bigg[
  \int^t_0
  \Bigg(
    \frac{
    \|
      G_l^{ ( \#_\varpi ) }( X_s^{ 0, x } )
      -
      G_l^{ ( \#_\varpi ) }( X_s^{ 0, y } )
    \|_{
      \mathcal{L}^{ p ( \#_\varpi + 1 ) }(
        \P ; L^{ ( \#_\varpi ) }( H, V_{l,r_l} ) 
      )
    }
    }{
      s^{(\delta_1+\delta_2+\ldots+\delta_k)} \,
      (t-s)^{r_l}
    }
  \Bigg)^{\!(l+1)}
  \, ds
  \Bigg]^{\nicefrac{1}{(l+1)}}
\\&+
    \sum_{ \varpi \in \Pi^*_k }
    \sum_{ I \in \varpi }
    t^{(
      \nicefrac{1}{(l+1)} - r_l - \gamma - \sum^k_{i=1} \delta_i
    )}
    \Big[
    \mathbb{B}\big( 1-(l+1)r_l, 1-(l+1)(\gamma+\smallsum^k_{i=1} \delta_i) \big)
    \Big]^{\nicefrac{1}{(l+1)}}
\\&\cdot
    \sup_{\mathbf{v}=(v_i)_{i\in I}\in (\nzspace{H})^{\#_I} }\,
      \sup_{ w \in (0,T] }
      \Bigg\{
    |G_l|_{ \Cb{\#_\varpi}( H, V_{l,r_l} ) } \,
    L^{ \boldsymbol{\delta} }_{\varpi\setminus\{I\}, \, p \#_\varpi } \,
      \frac{
        w^{
          \gamma + \iota_{ I \cup \{k+1\} }^{ (\boldsymbol{\delta}, 0) } 
        }
        \|
          X_w^{ \#_I, ( x, \mathbf{v} ) }
          -
          X_w^{ \#_I, ( y, \mathbf{v} ) }
        \|_{
          \mathcal{L}^{ p \#_\varpi }( \P ; H )
        }
      }{
        \prod_{ i \in I }
        \| v_i \|_{ H_{ - \delta_i } }
      }
    \Bigg\}
  \Bigg)
  .
\end{split}
\end{equation}
Hence, we obtain that for all 
$ k \in \{ 1, 2, \dots, n \} $,
$ l \in \{ 0, 1 \} $,
$ p \in [ 2 , \infty ) $,  
$
  \boldsymbol{\delta}=
  (\delta_1, \delta_2, \dots, \delta_k) \in \deltaset{k} 
$,
$ \lambda \in [ \iota^{ ( \boldsymbol{\delta}, 0 ) }_\N, \nicefrac{1}{2} ) $, 
$
\gamma \in
[0,
\lambda - \iota^{ (\boldsymbol{\delta},0) }_\N]
\cap
[0,
\nicefrac{1}{2}
-
\sum^k_{ i=1 } \delta_i  
)
$, 
$
  \mathbf{x} = 
  ( x, u_1 , u_2, \dots, u_k ) 
$, 
$
  \mathbf{y} = 
  ( y, u_1 , u_2, \dots, u_k ) 
  \in \times^k_{i=0} H^{[i]}
$
it holds that
\begin{equation}
\begin{split}
&
  \sup_{ t \in (0,T] }
  \Bigg\{
  t^\lambda \, \chi^{r_l,T}_{A,\eta} \,
  \Big[
    \tfrac{p \, (p-1)}{2}
  \Big]^{\nicefrac{l}{2}} \,
  \Bigg[
  \int^t_0
  \Bigg(
  \frac{
    \|
      \mathbf{G}_{ k, l }^{ \mathbf{x} }( s , X_s^{ k, \mathbf{x} } )
      -
      \mathbf{G}_{ k, l }^{ \mathbf{y} }( s , X_s^{ k, \mathbf{x} } )
    \|_{
      \mathcal{L}^p( \P; V_{l,r_l} )
    }
  }{
    (t-s)^{r_l}
    \prod^k_{i=1}
    \|u_i\|_{H_{-\delta_i}}
  }
  \Bigg)^{ \!(l+1) }
  \, ds
  \Bigg]^{\nicefrac{1}{(l+1)}}
  \Bigg\}
\\&\leq
  | T \vee 1 |^{ \lceil k/2 \rceil \min\{ 1-\alpha, \nicefrac{1}{2} - \beta 
        \} }
  \Bigg(
    \sum_{ \varpi \in \Pi_k }
    L^{ \boldsymbol{\delta} }_{\varpi, \, p( \#_\varpi + 1 )} \,
    \chi^{r_l,T}_{A,\eta} 
\\&\cdot
  \sup_{ t \in (0,T] }
  \Bigg\{
  t^\lambda \,
  \Big[
    \tfrac{p \, (p-1)}{2}
  \Big]^{\nicefrac{l}{2}} \,
  \Bigg[
  \int^t_0
    \frac{
    \|
      G_l^{ ( \#_\varpi ) }( X_s^{ 0, x } )
      -
      G_l^{ ( \#_\varpi ) }( X_s^{ 0, y } )
    \|^{(l+1)}_{
      \mathcal{L}^{ p ( \#_\varpi + 1 ) }(
        \P ; L^{ ( \#_\varpi ) }( H, V_{l,r_l} ) 
      )
    }
    }{
      s^{(l+1)(\delta_1+\delta_2+\ldots+\delta_k)} \,
      (t-s)^{(l+1)r_l}
    }
  \, ds
  \Bigg]^{\nicefrac{1}{(l+1)}}
  \Bigg\}
\\&+
    \sum_{ \varpi \in \Pi^*_k }
    \sum_{ I \in \varpi }
    \chi^{r_l,T}_{A,\eta} \,
    T^{(
      \lambda+\nicefrac{1}{(l+1)} - r_l - \gamma - \sum^k_{i=1} \delta_i
    )} \,
  \Big[
    \tfrac{p \, (p-1)}{2}
  \Big]^{\nicefrac{l}{2}} \,
    \Big[
    \mathbb{B}\big( 1-(l+1)r_l, 1-(l+1)(\gamma+\smallsum^k_{i=1} \delta_i) \big)
    \Big]^{\nicefrac{1}{(l+1)}}
\\&\cdot
    \sup_{\mathbf{v}=(v_i)_{i\in I}\in (\nzspace{H})^{\#_I} }\,
      \sup_{ w \in (0,T] }
      \Bigg\{
    |G_l|_{ \Cb{\#_\varpi}( H, V_{l,r_l} ) } \,
    L^{ \boldsymbol{\delta} }_{\varpi\setminus\{I\}, \, p \#_\varpi } \,
      \frac{
        w^{
          \gamma + \iota_{ I \cup \{k+1\} }^{ (\boldsymbol{\delta}, 0) } 
        }
        \|
          X_w^{ \#_I, ( x, \mathbf{v} ) }
          -
          X_w^{ \#_I, ( y, \mathbf{v} ) }
        \|_{
          \mathcal{L}^{ p \#_\varpi }( \P ; H )
        }
      }{
        \prod_{ i \in I }
        \| v_i \|_{ H_{ - \delta_i } }
      }
    \Bigg\}
  \Bigg)
  .
\end{split}
\end{equation}
This shows that for all 
$ k \in \{ 1, 2, \dots, n \} $,
$ p \in [ 2 , \infty ) $,  
$
  \boldsymbol{\delta}=
  (\delta_1, \delta_2, \dots, \delta_k) \in \deltaset{k} 
$,
$ \lambda \in [ \iota^{ ( \boldsymbol{\delta}, 0 ) }_\N, \nicefrac{1}{2} ) $, 
$
\gamma \in
[0,
\lambda - \iota^{ (\boldsymbol{\delta},0) }_\N]
\cap
[0,
\nicefrac{1}{2}
-
\sum^k_{ i=1 } \delta_i  
)
$, 
$
  x, y \in H
$
it holds that
\begin{equation}
\label{eq:nonlinearity.perturb}
\begin{split}
&
  \sum^1_{l=0}
  \sup_{ \mathbf{u}=(u_1,u_2,\ldots,u_k) \in (\nzspace{H})^k }
  \sup_{ t\in(0,T] }
  \Bigg\{
  t^\lambda \, \chi^{r_l,T}_{A,\eta} \,
  \Big[
    \tfrac{p \, (p-1)}{2}
  \Big]^{\nicefrac{l}{2}}
\\&\cdot
  \Bigg[
  \int^t_0
  \Bigg(
  \frac{
    \|
      \mathbf{G}_{ k, l }^{ (x,\mathbf{u}) }( s , X_s^{ k, (x,\mathbf{u}) } )
      -
      \mathbf{G}_{ k, l }^{ (y,\mathbf{u}) }( s , X_s^{ k, (x,\mathbf{u}) } )
    \|_{
      \mathcal{L}^p( \P; V_{l,r_l} )
    }
  }{
    (t-s)^{r_l}
    \prod^k_{i=1}
    \|u_i\|_{H_{-\delta_i}}
  }
  \Bigg)^{ \!(l+1) }
  \, ds
  \Bigg]^{\nicefrac{1}{(l+1)}}
  \Bigg\}
\\&\leq
  | T \vee 1 |^{ \lceil k/2 \rceil \min\{ 1-\alpha, 1/2 - \beta 
  \} }
\\&\cdot
    \Bigg(
  \sum_{
    \varpi \in \Pi_k
  }
    L^{ \boldsymbol{\delta} }_{\varpi, p( \#_\varpi + 1 )}
    \Bigg[
    \chi^{ \alpha, T }_{ A, \eta }
    \sup_{ t\in( 0,T ] }
    \Bigg\{
    t^\lambda  
    \int_0^t
      \tfrac{
    \|
      F^{ ( \#_\varpi ) }( X_s^{ 0, x } )
      -
      F^{ ( \#_\varpi ) }( X_s^{ 0, y } )
    \|_{
      \mathcal{L}^{ p ( \#_\varpi + 1 ) }(
        \P ; L^{ ( \#_\varpi ) }( H, H_{ -\alpha } ) 
      )
    }
      }{
      ( t - s )^\alpha\,
      s^{ ( \delta_1 + \delta_2 + \ldots + \delta_k ) }
      }
    \, \diffns s
    \Bigg\}
\\ &
    +
    \chi^{ \beta, T }_{ A, \eta }
    \sup_{
      t \in (0,T]
    }
    \left\{
    t^\lambda
    \bigg[
      \tfrac{ p \, ( p - 1 ) }{ 2 }  
      \int_0^t
      \tfrac{
    \|
      B^{ ( \#_\varpi ) }( X_s^{ 0, x } )
      -
      B^{ ( \#_\varpi ) }( X_s^{ 0, y } )
    \|^2_{
      \mathcal{L}^{ p ( \#_\varpi + 1 ) }(
        \P ; L^{ ( \#_\varpi ) }( H, HS( U, H_{ -\beta } ) ) 
      )
    }
      }{
      ( t - s )^{ 2 \beta }\,
      s^{ 
        2( \delta_1 + \delta_2 + \ldots + \delta_k ) 
      }
      }
      \,\diffns s
    \bigg]^{ \nicefrac{1}{2} }
      \right\}\Bigg]
\\&+
    \sum_{ \varpi \in \Pi^*_k }
          \sum\limits_{
            I \in \varpi
          }
    \sup_{\mathbf{u}=(u_i)_{i\in I}\in (\nzspace{H})^{\#_I}}\,
            \sup_{ t \in (0,T] }
            \bigg\{
          L^{ \boldsymbol{\delta} }_{ \varpi \setminus \{I\}, p \, \#_\varpi }
            \tfrac{
              t^{
                \gamma + \iota_{ I \cup \{k+1\} }^{ (\boldsymbol{\delta}, 0) } 
              }
              \|
                X_t^{ \#_I, ( x, \mathbf{u} ) }
                -
                X_t^{ \#_I, ( y, \mathbf{u} ) }
              \|_{
                \mathcal{L}^{ p \#_\varpi }( \P ; H )
              }
            }{
              \prod_{ i \in I }
              \| u_i \|_{ H_{ - \delta_i } }
            }
\\&\cdot
      \bigg[
        \chi^{ \alpha, T }_{ A, \eta }\,
        T^{(
        \lambda + 1 - \alpha - \gamma - \sum^k_{ i=1 } \delta_i
        )}\,
        |F|_{  \Cb{ \#_\varpi }( H, H_{ -\alpha } )  }\,
        \mathbb{B}\big( 1-\alpha, 1 - \gamma - \smallsum^k_{ i=1 } \delta_i \big)
\\&+
        \chi^{ \beta, T }_{ A, \eta } \,
        T^{(
        \lambda + \nicefrac{1}{2} - \beta - \gamma - \sum^k_{ i=1 } \delta_i
        )} \,
        |B|_{  \Cb{ \#_\varpi }( H, HS( U, H_{ -\beta } ) )  }\,
        \sqrt{
        \tfrac{
          p \, (p-1)
        }{ 2 }\,
          \mathbb{B}\big( 1 - 2\beta
          ,
          1 - 2\gamma
          -
          2 \smallsum^k_{ i=1 } \delta_i  
          \big)
        }
      \bigg]
                  \bigg\}
      \Bigg).
\end{split}
\end{equation}
Combining~\eqref{eq:initial.perturb.gronwall} with~\eqref{eq:nonlinearity.perturb} yields that for all  
$ k \in \{ 1, 2, \dots, n \} $,
$ p \in [2,\infty) $, 
$
  \boldsymbol{\delta}=
  (\delta_1, \delta_2, \dots, \delta_k) \in \deltaset{k} 
$,
$ \lambda \in [ \iota^{ ( \boldsymbol{\delta}, 0 ) }_\N, \nicefrac{1}{2} ) $, 
$
\gamma \in
[0,
\lambda - \iota^{ (\boldsymbol{\delta},0) }_\N]
\cap
[0,
\nicefrac{1}{2}
-
\sum^k_{ i=1 } \delta_i  
)
$, 
$
  x,y
  \in H
$ 
it holds that 
\begin{equation}
\label{eq:derivative.difference}
\begin{split}
& 
  \sup_{\mathbf{u}=(u_1,u_2,\dots,u_k)\in (\nzspace{H})^k }
  \sup_{ t \in (0,T] }
  \frac{
  t^\lambda
    \,
    \|
      X_t^{ k, ( x, \mathbf{u} ) }
      -
      X_t^{ k, ( y, \mathbf{u} ) }
    \|_{
      \mathcal{L}^p( \P; H )
    }
  }{
    \prod^k_{ i=1 }
    \|u_i\|_{ H_{ -\delta_i } }
  }
\\&\leq
  | T \vee 1 |^{ \lceil k/2 \rceil \min\{ 1-\alpha, 1/2 - \beta 
  \} }
  \,
  \Theta_{ A, \eta, p, T }^{ 
    \alpha, \beta, \lambda
  }\big(
    | F |_{ \Cb{1}( H , H_{ - \alpha } ) }
    ,
    | B |_{ \Cb{1}( H , HS( U , H_{ - \beta } ) ) }
  \big)
\\&\cdot
    \Bigg(
  \sum_{
    \varpi \in \Pi_k
  }
    L^{ \boldsymbol{\delta} }_{\varpi, p( \#_\varpi + 1 )}
    \Bigg[
    \chi^{ \alpha, T }_{ A, \eta }
    \sup_{ t\in( 0,T ] }
    \Bigg\{
    t^\lambda  
    \int_0^t
      \tfrac{
    \|
      F^{ ( \#_\varpi ) }( X_s^{ 0, x } )
      -
      F^{ ( \#_\varpi ) }( X_s^{ 0, y } )
    \|_{
      \mathcal{L}^{ p ( \#_\varpi + 1 ) }(
        \P ; L^{ ( \#_\varpi ) }( H, H_{ -\alpha } ) 
      )
    }
      }{
      ( t - s )^\alpha\,
      s^{ ( \delta_1 + \delta_2 + \ldots + \delta_k ) }
      }
    \, \diffns s
    \Bigg\}
\\ &
    +
    \chi^{ \beta, T }_{ A, \eta }
    \sup_{
      t \in (0,T]
    }
    \left\{
    t^\lambda
    \bigg[
      \tfrac{ p \, ( p - 1 ) }{ 2 }  
      \int_0^t
      \tfrac{
    \|
      B^{ ( \#_\varpi ) }( X_s^{ 0, x } )
      -
      B^{ ( \#_\varpi ) }( X_s^{ 0, y } )
    \|^2_{
      \mathcal{L}^{ p ( \#_\varpi + 1 ) }(
        \P ; L^{ ( \#_\varpi ) }( H, HS( U, H_{ -\beta } ) ) 
      )
    }
      }{
      ( t - s )^{ 2 \beta }\,
      s^{ 
        2( \delta_1 + \delta_2 + \ldots + \delta_k ) 
      }
      }
      \,\diffns s
    \bigg]^{ \nicefrac{1}{2} }
      \right\}\Bigg]
\\&+
    \sum_{ \varpi \in \Pi^*_k }
          \sum\limits_{
            I \in \varpi
          }
    \sup_{\mathbf{u}=(u_i)_{i\in I}\in (\nzspace{H})^{\#_I}}\,
            \sup_{ t \in (0,T] }
            \bigg\{
          L^{ \boldsymbol{\delta} }_{ \varpi \setminus \{I\}, p \, \#_\varpi }
            \tfrac{
              t^{
                \gamma + \iota_{ I \cup \{k+1\} }^{ (\boldsymbol{\delta}, 0) } 
              }
              \|
                X_t^{ \#_I, ( x, \mathbf{u} ) }
                -
                X_t^{ \#_I, ( y, \mathbf{u} ) }
              \|_{
                \mathcal{L}^{ p \#_\varpi }( \P ; H )
              }
            }{
              \prod_{ i \in I }
              \| u_i \|_{ H_{ - \delta_i } }
            }
\\&\cdot
      \bigg[
        \chi^{ \alpha, T }_{ A, \eta }\,
        T^{(
        \lambda + 1 - \alpha - \gamma - \sum^k_{ i=1 } \delta_i
        )}\,
        |F|_{  \Cb{ \#_\varpi }( H, H_{ -\alpha } )  }\,
        \mathbb{B}\big( 1-\alpha, 1 - \gamma - \smallsum^k_{ i=1 } \delta_i \big)
\\&+
        \chi^{ \beta, T }_{ A, \eta } \,
        T^{(
        \lambda + \nicefrac{1}{2} - \beta - \gamma - \sum^k_{ i=1 } \delta_i
        )} \,
        |B|_{  \Cb{ \#_\varpi }( H, HS( U, H_{ -\beta } ) )  }\,
        \sqrt{
        \tfrac{
          p \, (p-1)
        }{ 2 }\,
          \mathbb{B}\big( 1 - 2\beta
          ,
          1 - 2\gamma
          -
          2 \smallsum^k_{ i=1 } \delta_i  
          \big)
        }
      \bigg]
                  \bigg\}
      \Bigg).
\end{split}
\end{equation}
In particular, this shows that for all 
$ k \in \{ 1, 2, \dots, n \} $,
$ p \in [2,\infty) $, 
$
  \boldsymbol{\delta}=
  (\delta_1, \delta_2, \dots, \delta_k) \in \deltaset{k} 
$,
$
  x,y
  \in H
$ 
it holds that 
\begin{equation}
\label{eq:derivative.difference.Lip}
\begin{split}
& 
  \sup_{\mathbf{u}=(u_1,u_2,\dots,u_k)\in (\nzspace{H})^k }
  \sup_{ t \in (0,T] }
  \frac{
  t^{ \iota^{(\boldsymbol{\delta},0)}_\N }
    \,
    \|
      X_t^{ k, ( x, \mathbf{u} ) }
      -
      X_t^{ k, ( y, \mathbf{u} ) }
    \|_{
      \mathcal{L}^p( \P; H )
    }
  }{
    \prod^k_{ i=1 }
    \|u_i\|_{ H_{ -\delta_i } }
  }
\\&\leq
  | T \vee 1 |^{ \lceil k/2 \rceil \min\{ 1-\alpha, 1/2 - \beta 
  \} }
  \,
  \Theta_{ A, \eta, p, T }^{ 
    \alpha, \beta, \iota^{(\boldsymbol{\delta},0)}_\N
  }\!\big(
    | F |_{ \Cb{1}( H , H_{ - \alpha } ) }
    ,
    | B |_{ \Cb{1}( H , HS( U , H_{ - \beta } ) ) }
  \big)
\\&\cdot
    \Bigg(
  \sum_{
    \varpi \in \Pi_k
  }
    L^{ \boldsymbol{\delta} }_{\varpi, p( \#_\varpi + 1 )}
    \Bigg[
    \chi^{ \alpha, T }_{ A, \eta }
    \sup_{ t\in( 0,T ] }
    \Bigg\{
    t^{\iota^{(\boldsymbol{\delta},0)}_\N}
    \int_0^t
      \tfrac{
    \|
      F^{ ( \#_\varpi ) }( X_s^{ 0, x } )
      -
      F^{ ( \#_\varpi ) }( X_s^{ 0, y } )
    \|_{
      \mathcal{L}^{ p ( \#_\varpi + 1 ) }(
        \P ; L^{ ( \#_\varpi ) }( H, H_{ -\alpha } ) 
      )
    }
      }{
      ( t - s )^\alpha\,
      s^{ ( \delta_1 + \delta_2 + \ldots + \delta_k ) }
      }
    \, \diffns s
    \Bigg\}
\\ &
    +
    \chi^{ \beta, T }_{ A, \eta }
    \sup_{
      t \in (0,T]
    }
    \left\{
    t^{\iota^{(\boldsymbol{\delta},0)}_\N}
    \bigg[
      \tfrac{ p \, ( p - 1 ) }{ 2 }  
      \int_0^t
      \tfrac{
    \|
      B^{ ( \#_\varpi ) }( X_s^{ 0, x } )
      -
      B^{ ( \#_\varpi ) }( X_s^{ 0, y } )
    \|^2_{
      \mathcal{L}^{ p ( \#_\varpi + 1 ) }(
        \P ; L^{ ( \#_\varpi ) }( H, HS( U, H_{ -\beta } ) ) 
      )
    }
      }{
      ( t - s )^{ 2 \beta }\,
      s^{ 
        2( \delta_1 + \delta_2 + \ldots + \delta_k ) 
      }
      }
      \,\diffns s
    \bigg]^{ \nicefrac{1}{2} }
      \right\}\Bigg]
\\&+
    \sum_{ \varpi \in \Pi^*_k }
          \sum\limits_{
            I \in \varpi
          }
    \sup_{\mathbf{u}=(u_i)_{i\in I}\in (\nzspace{H})^{\#_I}}\,
            \sup_{ t \in (0,T] }
            \bigg\{
          L^{ \boldsymbol{\delta} }_{ \varpi \setminus \{I\}, p \, \#_\varpi }
            \tfrac{
              t^{
                \iota_{ I \cup \{k+1\} }^{ (\boldsymbol{\delta}, 0) } 
              }
              \|
                X_t^{ \#_I, ( x, \mathbf{u} ) }
                -
                X_t^{ \#_I, ( y, \mathbf{u} ) }
              \|_{
                \mathcal{L}^{ p \#_\varpi }( \P ; H )
              }
            }{
              \prod_{ i \in I }
              \| u_i \|_{ H_{ - \delta_i } }
            }
\\&\cdot
      \bigg[
        \chi^{ \alpha, T }_{ A, \eta }\,
        T^{(
        \iota^{(\boldsymbol{\delta},0)}_\N + 1 - \alpha - \sum^k_{ i=1 } \delta_i
        )}\,
        |F|_{  \Cb{ \#_\varpi }( H, H_{ -\alpha } )  }\,
        \mathbb{B}\big( 1-\alpha, 1 - \smallsum^k_{ i=1 } \delta_i \big)
\\&+
        \chi^{ \beta, T }_{ A, \eta } \,
        T^{(
        \iota^{(\boldsymbol{\delta},0)}_\N + \nicefrac{1}{2} - \beta - \sum^k_{ i=1 } \delta_i
        )} \,
        |B|_{  \Cb{ \#_\varpi }( H, HS( U, H_{ -\beta } ) )  }\,
        \sqrt{
        \tfrac{
          p \, (p-1)
        }{ 2 }\,
          \mathbb{B}\big( 1 - 2\beta
          ,
          1
          -
          2 \smallsum^k_{ i=1 } \delta_i  
          \big)
        }
      \bigg]
                  \bigg\}
      \Bigg).
\end{split}
\end{equation}
Furthermore, we note that Corollary~2.8 in~\cite{AnderssonJentzenKurniawan2016arXiv} 
(with 
$
H=H
$, 
$
U=U
$, 
$ T = T $, 
$ \eta = \eta $, 
$ p = p $, 
$ \alpha = \alpha $, 
$
\hat{\alpha} = 0
$,
$ \beta = \beta $, 
$
\hat{\beta} 
= 
0
$,
$L_0=|F|_{\Cb{1}(H,H_{-\alpha})}$,  
$ \hat{L}_0 = \|F(0)\|_{H_{-\alpha}} $, 
$ L_1 = |B|_{\Cb{1}(H,HS(U,H_{-\beta}))} $, 
$
\hat{L}_1
=\|B(0)\|_{HS(U,H_{-\beta})}
$, 
$
W=W
$, 
$ A=A $,
$\mathbf{F}=\big([0,T]\times\Omega\times H\ni(t,\omega,z)\mapsto F(z) \in H_{-\alpha}\big)$, 
$\mathbf{B}=\big([0,T]\times\Omega\times H\ni(t,\omega,z)\mapsto (U\ni u \mapsto B(z)u \in H_{-\beta}) \in HS(U,H_{-\beta})\big)$, 
$ \delta = 0 $, 
$ X^1 = X^{0,x} $, 
$ X^2 = X^{0,y} $, 
$ \lambda = 0 $ 
for 
$ x,y \in H $, 
$ p \in [2,\infty) $
in the notation of Corollary~2.8 in~\cite{AnderssonJentzenKurniawan2016arXiv}) 
and~\eqref{eq:proof.derivatives} show that for all 
$ p \in [2,\infty) $, 
$ x, y \in H $ 
it holds that 
\begin{equation}
\label{eq:initial.perturb}
  \sup_{ t \in (0,T] }
  \| X^{0,x}_t - X^{0,y}_t \|_{ \lpn{p}{\P}{H} }
  \leq
  \chi^{0,T}_{ A, \eta } \,
  \|x-y\|_H \,
  \Theta^{ \alpha,\beta,0 }_{ A, \eta, p, T }
  \big(
    | F |_{ \Cb{1}( H , H_{ - \alpha } ) }
    ,
    | B |_{ \Cb{1}( H , HS( U , H_{ - \beta } ) ) }
  \big)
  < \infty
  .
\end{equation}
This implies that for all 
$ k \in \{ 1, 2, \dots, n \} $, 
$ m \in \{ 1, 2, \ldots, k \} $,  
$ p \in [2,\infty) $,
$ l \in \{ 0, 1 \} $,  
$ 
  \boldsymbol{\delta}=
  (\delta_1, \delta_2, \dots, \delta_k) \in \deltaset{k} 
$,
$
  x, y \in H
$, 
$ t \in (0,T] $ 
with $ x \neq y $
it holds that 
\begin{equation}
\begin{split}
&
  t^{
    \iota^{ (\boldsymbol{\delta},0) }_\N
  }
  \left[
  \int^t_0
    \tfrac{
      \|
      G^{(m)}_l( X^{0,x}_s ) - G^{(m)}_l( X^{0,y}_s )
      \|^{ (l+1) }_{ \lpn{p}{\P}{ L^{ (m) }( H, V_{l,r_l} ) } }
    }{
      ( t-s )^{ (l+1)r_l }\,
      s^{ (l+1)( \delta_1 + \delta_2 + \ldots + \delta_k 
      ) }
    }
  \,\diffns{s}
  \right]^{ \nicefrac{1}{( l+1 )} }
\\&\leq
  t^{
  	\iota^{ (\boldsymbol{\delta},0) }_\N
  }
  \left[
  \int^t_0
  \tfrac{
  	1
  }{
  ( t-s )^{ (l+1)r_l }\,
  s^{ (l+1)( \delta_1+\delta_2 + \ldots + \delta_k ) }
}
\,\diffns{s}
\right]^{ \nicefrac{1}{( l+1 )} }
\\&\quad\cdot
  \sup_{ s\in(0,T] }
  \|
  G^{(m)}_l( X^{0,x}_s ) - G^{(m)}_l( X^{0,y}_s )
  \|_{ \lpn{p}{\P}{ L^{ (m) }( H, V_{l,r_l} ) } }
\\&\leq
  T^{
    (\nicefrac{1}{( l+1 )} - r_l - \min\{1-\alpha,\nicefrac{1}{2}-\beta\})
  }\,
  \big|\mathbb{B}\big(
    1-(l+1)r_l
    ,
    1-(l+1) \smallsum^k_{ i=1 } \delta_i
  \big)\big|^{ \nicefrac{1}{(l+1)} }
\\&\quad\cdot
	  \sup_{ s\in(0,T] }
	  \|
	  G^{(m)}_l( X^{0,x}_s ) - G^{(m)}_l( X^{0,y}_s )
	  \|_{ \lpn{p}{\P}{ L^{ (m) }( H, V_{l,r_l} ) } }
\\&\leq
  T^{
    (\nicefrac{1}{( l+1 )} - r_l - \min\{1-\alpha,\nicefrac{1}{2}-\beta\})
  }\,
  \big|\mathbb{B}\big(
    1-(l+1)r_l
    ,
    1-(l+1) \smallsum^k_{ i=1 } \delta_i
  \big)\big|^{ \nicefrac{1}{(l+1)} }
\\&\quad\cdot
  |G_l|_{\operatorname{Lip}^m(H,V_{l,r_l})}
	  \sup_{ s\in(0,T] }
	  \|
	  X^{0,x}_s - X^{0,y}_s
	  \|_{ \lpn{p}{\P}{ H } }
\\&\leq
  T^{
    (\nicefrac{1}{( l+1 )} - r_l - \min\{1-\alpha,\nicefrac{1}{2}-\beta\})
  }\,
  \Theta^{ \alpha,\beta,0 }_{ A, \eta, p, T }
  \big(
    | F |_{ \Cb{1}( H , H_{ - \alpha } ) }
    ,
    | B |_{ \Cb{1}( H , HS( U , H_{ - \beta } ) ) }
  \big)
\\&\quad\cdot
  \big|\mathbb{B}\big(
    1-(l+1)r_l
    ,
    1-(l+1) \smallsum^k_{ i=1 } \delta_i
  \big)\big|^{ \nicefrac{1}{(l+1)} }\,
  \chi^{ 0, T }_{ A, \eta }\,
  |G_l|_{ \operatorname{Lip}^m( H,V_{l,r_l} ) }\,
  \| x-y \|_H
  .
\end{split}
\end{equation}
Combining this with~\eqref{eq:derivative.difference.Lip} establishes that for all 
$ k \in \{1,2,\ldots,n\} $, 
$ p \in [2,\infty) $, 
$ \boldsymbol{\delta}=(\delta_1,\delta_2,\ldots,\delta_k) \in \deltaset{k} $
with 
$
  |F|_{\operatorname{Lip}^k(H,H_{-\alpha})}
  +
  |B|_{\operatorname{Lip}^k(H,HS(U,H_{-\beta}))}
  < \infty
$ 
it holds that 
\begin{equation}
\label{eq:lipschitz.arg.F}
\begin{split}
&
  \sup_{\substack{x,y\in H,\\ x\neq y}}
  \sup_{ \mathbf{u}=(u_1,u_2,\ldots,u_k) \in (\nzspace{H})^k }
  \sup_{ t \in (0,T] }
  \left[
  \frac{
  t^{ \iota^{(\boldsymbol{\delta},0)}_\N }
    \,
    \|
      X_t^{ k, ( x, \mathbf{u} ) }
      -
      X_t^{ k, ( y, \mathbf{u} ) }
    \|_{
      \mathcal{L}^p( \P; H )
    }
  }{
    \|x-y\|_H
    \prod^k_{ i=1 }
    \|u_i\|_{ H_{ -\delta_i } }
  }
  \right]
\\&\leq
  | T \vee 1 |^{ \lceil k/2 \rceil \min\{ 1-\alpha, 1/2 - \beta 
  \} }
  \,
  \Theta_{ A, \eta, p, T }^{ 
    \alpha, \beta, \iota^{(\boldsymbol{\delta},0)}_\N
  }\big(
    | F |_{ \Cb{1}( H , H_{ - \alpha } ) }
    ,
    | B |_{ \Cb{1}( H , HS( U , H_{ - \beta } ) ) }
  \big)
\\&\cdot
    \Bigg(
  \sum_{
    \varpi \in \Pi_k
  }
    L^{ \boldsymbol{\delta} }_{\varpi, p( \#_\varpi + 1 )} \,
  \chi^{0,T}_{A,\eta} \,
  \Theta^{ \alpha,\beta,0 }_{ A, \eta, p, T }
  \big(
    | F |_{ \Cb{1}( H , H_{ - \alpha } ) }
    ,
    | B |_{ \Cb{1}( H , HS( U , H_{ - \beta } ) ) }
  \big)
\\&\cdot
      \bigg[
        \chi^{ \alpha, T }_{ A, \eta }\,
        T^{(
        1 - \alpha - \min\{1-\alpha,\nicefrac{1}{2}-\beta\}
        )}\,
        |F|_{  \operatorname{Lip}^{ \#_\varpi }( H, H_{ -\alpha } )  }\,
        \mathbb{B}\big( 1-\alpha, 1 - \smallsum^k_{ i=1 } \delta_i \big)
\\&+
        \chi^{ \beta, T }_{ A, \eta } \,
        T^{(
        \nicefrac{1}{2} - \beta - \min\{1-\alpha,\nicefrac{1}{2}-\beta\}
        )} \,
        |B|_{  \operatorname{Lip}^{ \#_\varpi }( H, HS( U, H_{ -\beta } ) )  }\,
        \sqrt{
        \tfrac{
          p \, (p-1)
        }{ 2 }\,
          \mathbb{B}\big( 1 - 2\beta
          ,
          1
          -
          2 \smallsum^k_{ i=1 } \delta_i  
          \big)
        }
      \bigg]
\\&+
    \sum_{ \varpi \in \Pi^*_k }
          \sum\limits_{
            I \in \varpi
          }
  \sup_{\substack{x,y\in H,\\ x\neq y}}
    \sup_{\mathbf{u}=(u_i)_{i\in I}\in (\nzspace{H})^{\#_I}}\,
            \sup_{ t \in (0,T] }
            \bigg\{
          L^{ \boldsymbol{\delta} }_{ \varpi \setminus \{I\}, p \, \#_\varpi }
            \tfrac{
              t^{
                \iota_{ I \cup \{k+1\} }^{ (\boldsymbol{\delta}, 0) } 
              }
              \|
                X_t^{ \#_I, ( x, \mathbf{u} ) }
                -
                X_t^{ \#_I, ( y, \mathbf{u} ) }
              \|_{
                \mathcal{L}^{ p \#_\varpi }( \P ; H )
              }
            }{
              \|x-y\|_H
              \prod_{ i \in I }
              \| u_i \|_{ H_{ - \delta_i } }
            }
\\&\cdot
      \bigg[
        \chi^{ \alpha, T }_{ A, \eta }\,
        T^{(
        1 - \alpha - \min\{1-\alpha,\nicefrac{1}{2}-\beta\}
        )}\,
        |F|_{  \Cb{ \#_\varpi }( H, H_{ -\alpha } )  }\,
        \mathbb{B}\big( 1-\alpha, 1 - \smallsum^k_{ i=1 } \delta_i \big)
\\&+
        \chi^{ \beta, T }_{ A, \eta } \,
        T^{(
        \nicefrac{1}{2} - \beta - \min\{1-\alpha,\nicefrac{1}{2}-\beta\}
        )} \,
        |B|_{  \Cb{ \#_\varpi }( H, HS( U, H_{ -\beta } ) )  }\,
        \sqrt{
        \tfrac{
          p \, (p-1)
        }{ 2 }\,
          \mathbb{B}\big( 1 - 2\beta
          ,
          1
          -
          2 \smallsum^k_{ i=1 } \delta_i  
          \big)
        }
      \bigg]
                  \bigg\}
      \Bigg).
\end{split}
\end{equation}
Induction and~\eqref{eq:L.finite} hence imply that for all 
$ k \in \{1,2,\ldots,n\} $, 
$ p \in [2,\infty) $, 
$ \boldsymbol{\delta}=(\delta_1,\delta_2,\ldots,\delta_k) \in \deltaset{k} $
with 
$
  |F|_{\operatorname{Lip}^k(H,H_{-\alpha})}
  +
  |B|_{\operatorname{Lip}^k(H,HS(U,H_{-\beta}))}
  < \infty
$ 
it holds that 
\begin{equation}
\begin{split}
&
  \sup_{\substack{x,y\in H,\\ x\neq y}}
  \sup_{ \mathbf{u}=(u_1,u_2,\ldots,u_k) \in (\nzspace{H})^k }
  \sup_{ t \in (0,T] }
  \left[
  \frac{
  t^{ \iota^{(\boldsymbol{\delta},0)}_\N }
    \,
    \|
      X_t^{ k, ( x, \mathbf{u} ) }
      -
      X_t^{ k, ( y, \mathbf{u} ) }
    \|_{
      \mathcal{L}^p( \P; H )
    }
  }{
    \|x-y\|_H
    \prod^k_{ i=1 }
    \|u_i\|_{ H_{ -\delta_i } }
  }
  \right]
\\&\leq
  | T \vee 1 |^k
  \,
  \Theta_{ A, \eta, p, T }^{ 
    \alpha, \beta, \iota^{(\boldsymbol{\delta},0)}_\N
  }\big(
    | F |_{ \Cb{1}( H , H_{ - \alpha } ) }
    ,
    | B |_{ \Cb{1}( H , HS( U , H_{ - \beta } ) ) }
  \big)
\\&\cdot
    \Bigg(
  \sum_{
    \varpi \in \Pi_k
  }
    L^{ \boldsymbol{\delta} }_{\varpi, p( \#_\varpi + 1 )} \,
  \chi^{0,T}_{A,\eta} \,
  \Theta^{ \alpha,\beta,0 }_{ A, \eta, p, T }
  \big(
    | F |_{ \Cb{1}( H , H_{ - \alpha } ) }
    ,
    | B |_{ \Cb{1}( H , HS( U , H_{ - \beta } ) ) }
  \big)
\\&\cdot
      \bigg[
        \chi^{ \alpha, T }_{ A, \eta }\,
        |F|_{  \operatorname{Lip}^{ \#_\varpi }( H, H_{ -\alpha } )  }\,
        \mathbb{B}\big( 1-\alpha, 1 - \smallsum^k_{ i=1 } \delta_i \big)
\\&+
        \chi^{ \beta, T }_{ A, \eta } \,
        |B|_{  \operatorname{Lip}^{ \#_\varpi }( H, HS( U, H_{ -\beta } ) )  }\,
        \sqrt{
        \tfrac{
          p \, (p-1)
        }{ 2 }\,
          \mathbb{B}\big( 1 - 2\beta
          ,
          1
          -
          2 \smallsum^k_{ i=1 } \delta_i  
          \big)
        }
      \bigg]
\\&+
    \sum_{ \varpi \in \Pi^*_k }
          \sum\limits_{
            I \in \varpi
          }
          L^{ \boldsymbol{\delta} }_{ \varpi \setminus \{I\}, p \, \#_\varpi }
  \sup_{\substack{x,y\in H,\\ x\neq y}}
    \sup_{\mathbf{u}=(u_i)_{i\in I}\in (\nzspace{H})^{\#_I}}\,
            \sup_{ t \in (0,T] }
            \bigg[
            \tfrac{
              t^{
                \iota_{ I \cup \{k+1\} }^{ (\boldsymbol{\delta}, 0) } 
              }
              \|
                X_t^{ \#_I, ( x, \mathbf{u} ) }
                -
                X_t^{ \#_I, ( y, \mathbf{u} ) }
              \|_{
                \mathcal{L}^{ p \#_\varpi }( \P ; H )
              }
            }{
              \|x-y\|_H
              \prod_{ i \in I }
              \| u_i \|_{ H_{ - \delta_i } }
            }
                  \bigg]
\\&\cdot
      \bigg[
        \chi^{ \alpha, T }_{ A, \eta }\,
        |F|_{  \Cb{ \#_\varpi }( H, H_{ -\alpha } )  }\,
        \mathbb{B}\big( 1-\alpha, 1 - \smallsum^k_{ i=1 } \delta_i \big)
\\&+
        \chi^{ \beta, T }_{ A, \eta } \,
        |B|_{  \Cb{ \#_\varpi }( H, HS( U, H_{ -\beta } ) )  }\,
        \sqrt{
        \tfrac{
          p \, (p-1)
        }{ 2 }\,
          \mathbb{B}\big( 1 - 2\beta
          ,
          1
          -
          2 \smallsum^k_{ i=1 } \delta_i  
          \big)
        }
      \bigg]
      \Bigg)
      < \infty.
\end{split}
\end{equation}
This implies~\eqref{eq:a_priori_Lip} and thus completes the proof of item~\eqref{item:lem_derivative:a_priori_Lip}. 
To prove item~\eqref{item:thm_derivative.continuous} we first observe that~\eqref{eq:initial.perturb} ensures that for all 
$ x \in H $, 
$ t \in [0,T] $
it holds that 
\begin{equation}
  \limsup\nolimits_{ H \ni y \to x }
  \E\big[\!\min\{1,
    \|X^{0,x}_t-X^{0,y}_t\|_H
  \}\big]
  =0.
\end{equation}
This implies for all 
$ x \in H $, 
$ \rho \in [0,1] $, 
$ t \in [0,T] $
that 
\begin{equation}
  \limsup\nolimits_{ H \ni y \to x }
  \E\big[\!\min\{1,
    \|
    (X^{0,x}_t + \rho[ X^{0,y}_t - X^{0,x}_t ])
    -
    X^{0,x}_t
    \|_H
  \}\big]
  =0.
\end{equation}
The fact that 
$ 
  \forall\, 
  k \in \{1,2,\ldots,n\}, \,
  l \in \{ 0, 1 \} 
  \colon
  G^{(k)}_l \in 
  \mathcal{C}( H, L^{(k)}(H,V_{l,0}) )
$, 
e.g., Lemma~4.2 in 
Hutzenthaler et al.~\cite{HutzenthalerJentzenSalimova2016arXiv},
and, e.g., item~(ii) of Theorem~6.12 in Klenke~\cite{k08b}
hence ensure that for all 
$ k \in \{ 1,2,\ldots,n \} $,
$ l \in \{ 0, 1 \} $,  
$ x \in H $, 
$ \rho\in[0,1] $, 
$t\in[0,T]$, 
$ (x_m)_{m\in\N_0} \subseteq H $ 
with 
$ \limsup_{m\to\infty} \|x_m-x_0\|_H = 0 $
it holds that 
\begin{equation}
  \limsup\nolimits_{ m \to \infty }
  \E\big[\!\min\{
  1
  ,
  \|
  G^{(k)}_l( X^{0,x_0}_t + \rho[ X^{0,x_m}_t - X^{0,x_0}_t ] )
  -
  G^{(k)}_l( X^{0,x_0}_t )
  \|_{ L^{(k)}( H, V_{l,0} ) }
  \}\big]
  = 0
  .
\end{equation}
Combining this and, e.g., Lemma~4.2 in 
Hutzenthaler et al.~\cite{HutzenthalerJentzenSalimova2016arXiv} 
(with 
$ I = \{\emptyset\} $, 
$ c=1 $, 
$
  X^m(\emptyset,\omega) = 
  \|G^{(k)}_l( X^{0,x}_t(\omega) + \rho [ X^{0,x_m}_t(\omega) - X^{0,x}_t(\omega) ] ) - G^{(k)}_l( X^{0,x}_t(\omega) )\|_{L^{(k)}(H,V_{l,0})}
$
for 
$ \omega \in \Omega $, 
$ t \in [0,T] $, 
$ \rho \in [0,1] $, 
$ l \in \{0,1\} $, 
$ k \in \{1,2,\ldots,n\} $, 
$
  (x_j)_{ j \in \N }
  \in 
  \{
    y \in \mathbb{M}( \N, H )
    \colon
    \limsup_{ j \to \infty }
    \|y_j-x\|_H
    =0
  \}
$, 
$ m \in \N $, 
$ x \in H $ 
in the notation of Lemma~4.2 in 
Hutzenthaler et al.~\cite{HutzenthalerJentzenSalimova2016arXiv}) 
establishes that for all 
$ \varepsilon \in (0,\infty) $, 
$ k \in \{1,2,\ldots,n\} $, 
$ l \in \{0,1\} $, 
$ x \in H $, 
$ \rho \in [0,1] $, 
$ t \in [0,T] $, 
$ (x_m)_{ m \in \N } \subseteq H $ 
with 
$
    \limsup_{ m \to \infty }
    \|x_m-x\|_H
    =0
$ 
it holds that 
\begin{multline}
  \limsup\nolimits_{ m \to \infty } \,
  \P\big(\big\{
    \omega \in \Omega \colon
    \|G^{(k)}_l( X^{0,x}_t(\omega) + \rho [ X^{0,x_m}_t(\omega) 
    - X^{0,x}_t(\omega) ] ) 
\\
    - G^{(k)}_l( X^{0,x}_t(\omega) )\|_{L^{(k)}(H,V_{l,0})}
    \geq \varepsilon
  \big\}\big)
  =0.
\end{multline}
This, the fact that 
$
  \forall\, 
  k \in \{1,2,\ldots,n\}, \,
  l \in \{ 0,1 \}
  \colon
  \sup_{ x \in H }
  \|G^{(k)}_l(x)\|_{L^{(k)}(H,V_{l,0})}
  < \infty
$,
and, e.g., 
Proposition~4.5 in 
Hutzenthaler et al.~\cite{HutzenthalerJentzenSalimova2016arXiv} 
(with 
$ I = \{\emptyset\} $, 
$ p = p $, 
$ V = \R $, 
$
  X^m(\emptyset,\omega) = 
    \|G^{(k)}_l( X^{0,x_0}_t(\omega) + \rho[ X^{0,x_m}_t(\omega) - X^{0,x_0}_t(\omega) ] )
    -
    G^{(k)}_l( X^{0,x_0}_t(\omega) )\|_{L^{(k)}(H,V_{l,0})}
$
for 
$ \omega \in \Omega $, 
$ t \in [0,T] $, 
$ \rho \in [0,1] $, 
$ p \in (0,\infty) $, 
$ l \in \{0,1\} $, 
$ k \in \{1,2,\ldots,n\} $, 
$
  (x_j)_{ j \in \N_0 }
  \in 
  \{
    y \in \mathbb{M}( \N_0, H )
    \colon
    \limsup_{ j \to \infty }
    \|y_j-y_0\|_H
    =0
  \}
$, 
$ m \in \N_0 $ 
in the notation of Proposition~4.5 in 
Hutzenthaler et al.~\cite{HutzenthalerJentzenSalimova2016arXiv})
ensure that for all 
$ k \in \{ 1,2,\ldots,n \} $, 
$ l \in \{ 0, 1 \} $, 
$ p \in (0,\infty) $, 
$ \rho\in[0,1] $,
$t\in[0,T]$, 
$
  (x_m)_{ m \in \N_0 } \subseteq H
$ 
with 
$
    \limsup_{ m \to \infty }
    \|x_m-x_0\|_H
    =0
$
it holds that 
\begin{equation}
\label{eq:1st.derivative.conv1}
\begin{split}
  \limsup\nolimits_{ m \to \infty }
  \E\Big[
  \big\|
    G^{(k)}_l( X^{0,x_0}_t + \rho[ X^{0,x_m}_t - X^{0,x_0}_t ] )
    -
    G^{(k)}_l( X^{0,x_0}_t )
  \big\|^p_{L^{(k)}(H,V_{l,0})}
  \Big]
  = 0
  .
\end{split}
\end{equation}
Combining H\"{o}lder's inequality and  
Lebesgue's theorem of dominated convergence 
with~\eqref{eq:1st.derivative.conv1} 
(with $\rho=1$ in the notation of~\eqref{eq:1st.derivative.conv1})
yields that for all 
$ k \in \{ 1, 2, \dots, n \} $, 
$ l \in \{ 0,1 \} $, 
$ p \in [2,\infty) $, 
$ q \in ( 1,  \frac{1}{\max\{\alpha,2\beta,\nicefrac{1}{2}\}} ) $, 
$ \lambda \in [\max\{\alpha-\nicefrac{1}{q},\beta-\nicefrac{1}{(2q)}\},\infty) $, 
$ x \in H $
it holds that 
\begin{equation}
\label{eq:UI.cont}
\begin{split}
&
  \limsup_{ H\ni y \rightarrow x }
  \sup_{ t\in(0,T] }
  \bigg\{
      t^\lambda
      \bigg[
      \int_0^t
      \frac{
    \|
      G_l^{ ( k ) }( X_s^{ 0, x } )
      -
      G_l^{ ( k ) }( X_s^{ 0, y } )
    \|^{(l+1)}_{
      \mathcal{L}^p(
        \P ; L^{ ( k ) }( H, V_{l,r_l} ) 
      )
    }
      }{
      ( t - s )^{ (l+1) r_l }\,
      }
      \,\diffns s
      \bigg]^{\nicefrac{1}{(l+1)}}
      \bigg\}
\\&\leq
  \limsup_{ H\ni y \rightarrow x }
  \sup_{ t\in(0,T] }
      \bigg\{
      t^\lambda
      \left[\int_0^t
      \frac{
        1
      }{
      ( t - s )^{ q (l+1) r_l }\,
      }
      \,\diffns s  \right]^{\nicefrac{1}{[q(l+1)]}} 
\\&\quad\cdot
     \bigg[ \int_0^t
    \|
    G_l^{ ( k ) }( X_s^{ 0, x } )
    -
    G_l^{ ( k ) }( X_s^{ 0, y } )
    \|^{q(l+1)/(q-1)}_{
    	\mathcal{L}^p(
    	\P ; L^{ ( k ) }( H, V_{l,r_l} ) 
    	)
    }
      \,\diffns s   \bigg]^{\nicefrac{(q-1)}{[q(l+1)]}}
      \bigg\}
\\&=
  \limsup_{ H\ni y \rightarrow x }
  \sup_{ t\in(0,T] }
      \bigg\{
      \frac{t^{(\lambda+\nicefrac{1}{[q(l+1)]}-r_l)}}{[1-q(l+1)r_l]^{\nicefrac{1}{[q(l+1)]}}}
\\&\quad\cdot
     \bigg[ \int_0^t
    \|
      G_l^{ ( k ) }( X_s^{ 0, x } )
      -
      G_l^{ ( k ) }( X_s^{ 0, y } )
    \|^{q(l+1)/(q-1)}_{
      \mathcal{L}^p(
        \P ; L^{ ( k ) }( H, V_{l,r_l} ) 
      )
    }
      \,\diffns s   \bigg]^{\nicefrac{(q-1)}{[q(l+1)]}} 
      \bigg\}
\\&=
      \frac{T^{(\lambda+\nicefrac{1}{[q(l+1)]}-r_l)}}{[1-q(l+1)r_l]^{\nicefrac{1}{[q(l+1)]}}}
\\&\quad\cdot
  \bigg[
  \limsup_{ H\ni y \rightarrow x }
    \int_0^T
    \|
      G_l^{ ( k ) }( X_s^{ 0, x } )
      -
      G_l^{ ( k ) }( X_s^{ 0, y } )
    \|^{q(l+1)/(q-1)}_{
      \mathcal{L}^p(
        \P ; L^{ ( k ) }( H, V_{l,r_l} ) 
      )
    }
      \,\diffns s
  \bigg]^{\nicefrac{(q-1)}{[q(l+1)]}}
= 0
  .
\end{split}
\end{equation}
Moreover, observe that
the fact that 
$
\forall \, q \in ( 1,  \frac{1}{\max\{\alpha,2\beta,\nicefrac{1}{2}\}} )
\colon
0<\min\{\nicefrac{1}{q}-\alpha,\nicefrac{1}{(2q)}-\beta\}
<
\min\{1-\alpha,\nicefrac{1}{2}-\beta\}\leq \nicefrac{1}{2}
$
and~\eqref{eq:derivative.difference} 
(with 
$ k = k $, 
$ p = p $, 
$ \boldsymbol{\delta} = \mathbf{0}_k $, 
$ 
  \lambda = -\min\{\nicefrac{1}{q}-\alpha,\nicefrac{1}{(2q)}-\beta\}
$,  
$ 
  \gamma 
  = 
  \min\{1-\alpha,\nicefrac{1}{2}-\beta\} - 
  \min\{\nicefrac{1}{q}-\alpha,\nicefrac{1}{(2q)}-\beta\}
$, 
$ x = x $, 
$ y = y $
for 
$ x,y \in H $, 
$ q \in ( 1,  \frac{1}{\max\{\alpha,2\beta,\nicefrac{1}{2}\}} ) $, 
$ p \in [2,\infty) $, 
$ k \in \{ 1, 2, \dots, n \} $ 
in the notation of~\eqref{eq:derivative.difference})
imply that for all 
$ k \in \{ 1, 2, \dots, n \} $,
$ p \in [2,\infty) $, 
$ q \in ( 1,  \frac{1}{\max\{\alpha,2\beta,\nicefrac{1}{2}\}} ) $, 
$
  x,y
  \in H
$
it holds that 
\begin{equation}
\label{eq:derivative.difference.cont}
\begin{split}
& 
  \sup_{\mathbf{u}=(u_1,u_2,\dots,u_k)\in (\nzspace{H})^k }
  \sup_{ t \in (0,T] }
  \left[
  \frac{
    \|
      X_t^{ k, ( x, \mathbf{u} ) }
      -
      X_t^{ k, ( y, \mathbf{u} ) }
    \|_{
      \mathcal{L}^p( \P; H )
    }
  }{
  t^{\min\{\nicefrac{1}{q}-\alpha,\nicefrac{1}{(2q)}-\beta\}}
    \prod^k_{ i=1 }
    \|u_i\|_{ H }
  }
  \right]
\\&\leq
  | T \vee 1 |^{ \lceil k/2 \rceil \min\{ 1-\alpha, \nicefrac{1}{2} - \beta 
  \} }
  \,
  \Theta_{ A, \eta, p, T }^{ 
    \alpha, \beta, -\min\{\nicefrac{1}{q}-\alpha,\nicefrac{1}{(2q)}-\beta\}
  }\big(
    | F |_{ \Cb{1}( H , H_{ - \alpha } ) }
    ,
    | B |_{ \Cb{1}( H , HS( U , H_{ - \beta } ) ) }
  \big)
\\&\cdot
    \Bigg(
  \sum_{
    \varpi \in \Pi_k
  }
    L^{ \mathbf{0}_k }_{\varpi, p( \#_\varpi + 1 )}
    \Bigg[
    \chi^{ \alpha, T }_{ A, \eta }
    \sup_{ t\in( 0,T ] }
    \Bigg\{
    \frac{1}{t^{\min\{\nicefrac{1}{q}-\alpha,\nicefrac{1}{(2q)}-\beta\}}}
    \int_0^t
      \tfrac{
    \|
      F^{ ( \#_\varpi ) }( X_s^{ 0, x } )
      -
      F^{ ( \#_\varpi ) }( X_s^{ 0, y } )
    \|_{
      \mathcal{L}^{ p ( \#_\varpi + 1 ) }(
        \P ; L^{ ( \#_\varpi ) }( H, H_{ -\alpha } ) 
      )
    }
      }{
      ( t - s )^\alpha
      }
    \, \diffns s
    \Bigg\}
\\ &
    +
    \chi^{ \beta, T }_{ A, \eta }
    \sup_{
      t \in (0,T]
    }
    \left\{
    \frac{1}{t^{\min\{\nicefrac{1}{q}-\alpha,\nicefrac{1}{(2q)}-\beta\}}}
    \bigg[
      \tfrac{ p \, ( p - 1 ) }{ 2 }  
      \int_0^t
      \tfrac{
    \|
      B^{ ( \#_\varpi ) }( X_s^{ 0, x } )
      -
      B^{ ( \#_\varpi ) }( X_s^{ 0, y } )
    \|^2_{
      \mathcal{L}^{ p ( \#_\varpi + 1 ) }(
        \P ; L^{ ( \#_\varpi ) }( H, HS( U, H_{ -\beta } ) ) 
      )
    }
      }{
      ( t - s )^{ 2 \beta }
      }
      \,\diffns s
    \bigg]^{ \nicefrac{1}{2} }
      \right\}\Bigg]
\\&+
    \sum_{ \varpi \in \Pi^*_k }
          \sum\limits_{
            I \in \varpi
          }
    \sup_{\mathbf{u}=(u_i)_{i\in I}\in (\nzspace{H})^{\#_I}}\,
            \sup_{ t \in (0,T] }
            \bigg\{
          L^{ \mathbf{0}_k }_{ \varpi \setminus \{I\}, p \, \#_\varpi }
            \frac{
              \|
                X_t^{ \#_I, ( x, \mathbf{u} ) }
                -
                X_t^{ \#_I, ( y, \mathbf{u} ) }
              \|_{
                \mathcal{L}^{ p \#_\varpi }( \P ; H )
              }
            }{
              t^{\min\{\nicefrac{1}{q}-\alpha,\nicefrac{1}{(2q)}-\beta\}}
              \prod_{ i \in I }
              \| u_i \|_{ H }
            }
\\&\cdot
      \bigg[
        \chi^{ \alpha, T }_{ A, \eta }\,
        T^{(
        1 - \alpha - \min\{1-\alpha,\nicefrac{1}{2}-\beta\}
        )}\,
        |F|_{  \Cb{ \#_\varpi }( H, H_{ -\alpha } )  }\,
        \mathbb{B}\big( 1-\alpha, 1 - 
        \min\{1-\alpha,\nicefrac{1}{2}-\beta\} +
        \min\{\nicefrac{1}{q}-\alpha,\nicefrac{1}{(2q)}-\beta\}
        \big)
\\&+
        \chi^{ \beta, T }_{ A, \eta } \,
        T^{(
        \nicefrac{1}{2} - \beta - \min\{1-\alpha,\nicefrac{1}{2}-\beta\}
        )} \,
        |B|_{  \Cb{ \#_\varpi }( H, HS( U, H_{ -\beta } ) )  }
\\&\cdot
        \Big[
        \tfrac{
          p \, (p-1)
        }{ 2 } \,
          \mathbb{B}\big( 1 - 2\beta
          ,
          1 - 
          2\min\{1-\alpha,\nicefrac{1}{2}-\beta\} +
          2\min\{\nicefrac{1}{q}-\alpha,\nicefrac{1}{(2q)}-\beta\}
          \big)
        \Big]^{\nicefrac{1}{2}}
      \bigg]
                  \bigg\}
      \Bigg).
\end{split}
\end{equation}
Induction and \eqref{eq:UI.cont}--\eqref{eq:derivative.difference.cont} 
hence ensure that for all
$ k \in \{ 1, 2, \ldots, n \} $,  
$ p \in [2,\infty) $, 
$ q \in ( 1,  \frac{1}{\max\{\alpha,2\beta,\nicefrac{1}{2}\}} ) $, 
$ x \in H $
it holds that 
\begin{equation}
  \limsup_{ H \ni y \rightarrow x }
  \sup_{\mathbf{u}=(u_1,u_2,\dots,u_k)\in (\nzspace{H})^k }
  \sup_{t\in(0,T]}
  \left[
  \frac{
    \|
      X_t^{k,(x,\mathbf{u})}
      -
      X_t^{k,(y,\mathbf{u})}
    \|_{\mathcal{L}^p(\P;H)}
  }{
    t^{\min\{\nicefrac{1}{q}-\alpha,\nicefrac{1}{(2q)}-\beta\}}
    \prod_{i=1}^k\|u_i\|_H
  }
  \right]
  = 0
  .
\end{equation}
This and~\eqref{eq:proof.derivatives} show that for all 
$ k \in \{ 1, 2, \ldots, n \} $,  
$ p \in [2,\infty) $, 
$ q \in ( 1,  \frac{1}{\max\{\alpha,2\beta,\nicefrac{1}{2}\}} ) $, 
$ x \in H $
it holds that 
\begin{equation}
\label{eq:derivative.continuity}
\begin{split}
&
  \limsup_{ H \ni y \rightarrow x }
  \sup_{\mathbf{u}=(u_1,u_2,\dots,u_k)\in (\nzspace{H})^k }
  \sup_{t\in[0,T]}
  \left[
  \frac{
    \|
      X_t^{k,(x,\mathbf{u})}
      -
      X_t^{k,(y,\mathbf{u})}
    \|_{\mathcal{L}^p(\P;H)}
  }{
    \prod_{i=1}^k\|u_i\|_H
  }
  \right]
\\&\leq
  T^{\min\{\nicefrac{1}{q}-\alpha,\nicefrac{1}{(2q)}-\beta\}}
  \limsup_{ H \ni y \rightarrow x }
  \sup_{\mathbf{u}=(u_1,u_2,\dots,u_k)\in (\nzspace{H})^k }
  \sup_{t\in(0,T]}
  \frac{
    \|
      X_t^{k,(x,\mathbf{u})}
      -
      X_t^{k,(y,\mathbf{u})}
    \|_{\mathcal{L}^p(\P;H)}
  }{
    t^{\min\{\nicefrac{1}{q}-\alpha,\nicefrac{1}{(2q)}-\beta\}}
    \prod_{i=1}^k\|u_i\|_H
  }
  = 0
  .
\end{split}
\end{equation}
Combining~\eqref{eq:derivative.continuity} with item~\eqref{item:thm_derivative} proves item~\eqref{item:thm_derivative.continuous}.

We now prove item~\eqref{item:lem_derivative:frechet} by induction on $k\in\{1,2,\ldots,n\}$. 
Note that~\eqref{eq:initial.perturb} ensures that for all 
$ p \in (0,\infty) $ 
it holds that 
\begin{equation}
\label{eq:tildeL.finite}
  \tilde{L}_p < \infty.
\end{equation}
Furthermore, observe that 
for all 
$ l \in \{ 0, 1 \} $, 
$ \mathbf{u} = ( u_0, u_1 ) \in H^2 $, 
$ t\in[0,T] $ 
it holds that 
\begin{equation}
\label{eq:1st.diff.coeff}
    G_l( X^{ 0, u_0 + u_1 }_t )
    -
    G_l( X^{ 0, u_0 }_t )
  =
  \mathbf{\bar{G}}^{ \mathbf{u} }_{ 1,l }
  (
    t
    ,
      X^{ 0,u_0 + u_1 }_t
      -
      X^{ 0, u_0 }_t
  ).
\end{equation}
This and~\eqref{eq:proof.derivatives} imply that for all 
$ \mathbf{u}=(u_0,u_1) \in H^2 $, 
$ t \in [0,T] $ 
it holds that 
\begin{equation}
\begin{split}
&
 [X^{0,u_0+u_1}_t-X^{0,u_0}_t]_{ \P, \mathcal{B}(H) }
 =
  [e^{ t A } u_1]_{ \P, \mathcal{B}(H) }
\\&+
    \int_0^t
    e^{ ( t - s ) A }
    \,
    \mathbf{\bar{G}}^{ \mathbf{u} }_{ 1, 0 }( s, X^{0,u_0+u_1}_s-X^{0,u_0}_s )
  \, {\bf ds}
  +
  \int_0^t
    e^{ ( t - s ) A }
    \,
    \mathbf{\bar{G}}^{ \mathbf{u} }_{ 1, 1 }( s, X^{0,u_0+u_1}_s-X^{0,u_0}_s )
  \, \diffns W_s
  .  
\end{split}
\end{equation}
Combining this with the Burkholder-Davis-Gundy type inequality
in Lemma~7.7 in Da Prato \& Zabczyk~\cite{dz92}, \eqref{eq:Lip.G.1}, \eqref{eq:proof.derivatives}, and Proposition~2.7 in~\cite{AnderssonJentzenKurniawan2016arXiv} 
(with
$
H=H
$, 
$
U=U
$, 
$ T = T $, 
$ \eta = \eta $, 
$ p = p $, 
$ \alpha = 0 $, 
$
\hat{\alpha} = 0
$,
$ \beta = 0 $, 
$
\hat{\beta} 
= 
0
$,
$L_0=|F|_{\Cb{1}(H,H)}$, 
$
\hat{L}_0
=
0
$, 
$L_1=|B|_{\Cb{1}(H,HS(U,H))}$, 
$
\hat{L}_1
=
0
$, 
$
W=W
$, 
$ A=A $, 
$ \mathbf{F}=\mathbf{G}_{1,0}^{\mathbf{u}} $,
$ \mathbf{B}=\mathbf{G}_{1,1}^{\mathbf{u}} $, 
$
  \delta = 0
$, 
$ Y^1 = X^{ 0, \theta^1_1(\mathbf{u}) } - X^{ 0, \theta^1_0(\mathbf{u}) } $, 
$ Y^2 = X^{1,\mathbf{u}} $,
$ \lambda = 0 $ 
for 
$
  \mathbf{u}=(u_0,u_1)
  \in H^2
$, 
$ p \in [2,\infty) $ 
in the notation of Proposition~2.7 in~\cite{AnderssonJentzenKurniawan2016arXiv})  
ensures that for all  
$ p \in [ 2, \infty ) $, 
$
  \mathbf{u}
  \in H^2
$
it holds that 
\begin{equation}
\label{calc21}
\begin{split}
& \! \sup_{t\in[0,T]}
  \big\|
    X^{ 0, \theta^1_1(\mathbf{u}) }_t - X^{ 0, \theta^1_0(\mathbf{u}) }_t
    - X_t^{ 1, \mathbf{u} }
  \big\|_{
    \mathcal{L}^p ( \P ; H )
  }
  \leq
  \Theta_{A,\eta,p,T}^{0,0,0}
  \big(
    |F|_{\Cb{1}(H,H)},
    |B|_{\Cb{1}(H,HS(U,H))}
  \big)\\
& \! \cdot
  \sup_{t\in(0,T]}
  \Bigg[
  \Bigg\|
      \int_0^t
        e^{(t-s)A}
        \big(
          \mathbf{\bar{G}}_{1,0}^{\mathbf{u}}( s , X^{ 0,\theta^1_1(\mathbf{u}) }_s - X^{ 0,\theta^1_0(\mathbf{u}) }_s )
          -
          \mathbf{G}_{1,0}^{\mathbf{u}}( s , X^{ 0,\theta^1_1(\mathbf{u}) }_s - X^{ 0,\theta^1_0(\mathbf{u}) }_s )
        \big)
      \,{ \bf ds }\\
& \! +
      \int_0^t
        e^{(t-s)A}
        \big(
          \mathbf{\bar{G}}_{1,1}^{\mathbf{u}}( s , X^{ 0,\theta^1_1(\mathbf{u}) }_s - X^{ 0,\theta^1_0(\mathbf{u}) }_s )
          -
          \mathbf{G}_{1,1}^{\mathbf{u}}( s , X^{ 0,\theta^1_1(\mathbf{u}) }_s - X^{ 0,\theta^1_0(\mathbf{u}) }_s )
        \big)
      \,\diffns W_s
    \Bigg\|_{L^p(\P;H)}\Bigg]\\
& \! \leq
  \chi_{A,\eta}^{0,T} \,
  \Theta_{A,\eta,p,T}^{0,0,0}
  \big(
    |F|_{\Cb{1}(H,H)},
    |B|_{\Cb{1}(H,HS(U,H))}
  \big)
  \\
& \! \cdot
  \Bigg[
    \int_0^T    
        \|
          \mathbf{\bar{G}}_{1,0}^{\mathbf{u}}( s , X^{ 0,\theta^1_1(\mathbf{u}) }_s - X^{ 0,\theta^1_0(\mathbf{u}) }_s )
          -
          \mathbf{G}_{1,0}^{\mathbf{u}}( s , X^{ 0,\theta^1_1(\mathbf{u}) }_s - X^{ 0,\theta^1_0(\mathbf{u}) }_s )
        \|_{\mathcal{L}^p(\P;H)}
      \,\diffns s
      \\
& \! +
    \left[
      \tfrac{ p\, (p-1) }{ 2 }
      \int_0^T
        \|
          \mathbf{\bar{G}}_{1,1}^{\mathbf{u}}( s , X^{ 0,\theta^1_1(\mathbf{u}) }_s - X^{ 0,\theta^1_0(\mathbf{u}) }_s )
          -
          \mathbf{G}_{1,1}^{\mathbf{u}}( s , X^{ 0,\theta^1_1(\mathbf{u}) }_s - X^{ 0,\theta^1_0(\mathbf{u}) }_s )
        \|^2_{\mathcal{L}^p(\P;HS(U,H))}
      \,\diffns s
    \right]^{ \nicefrac{1}{2} }
  \Bigg].
\end{split}
\end{equation}
In addition, H\"{o}lder's inequality
yields that for all  
$ p \in [2,\infty) $, 
$ l \in \{ 0,1 \} $, 
$
  \mathbf{u}=( u_0, u_1 )
  \in H \times (\nzspace{H})
$,
$ t \in (0,T] $
it holds that 
\begin{equation}
\label{eq:1st.derivative.nonlinear.difference}
\begin{split}
& 
  \frac{
  \|
    \mathbf{\bar{G}}^{ \mathbf{u} }_{1,l}( t, X^{0,u_0+u_1}_t - X^{0,u_0}_t )
    -
    \mathbf{G}^{ \mathbf{u} }_{1,l}( t, X^{0,u_0+u_1}_t - X^{0,u_0}_t )
  \|_{ \lpn{p}{\P}{V_{l,0}} }
  }
  {
    \| u_1 \|_H
  }
\\&=
  \frac{1}{\|u_1\|_H}
  \left\|
    \int^1_0
    [
    G'_l( X^{0,u_0}_t + \rho[ X^{0,u_0+u_1}_t - X^{0,u_0}_t ] )
    -
    G'_l( X^{0,u_0}_t )
    ]
    (X^{0,u_0+u_1}_t - X^{0,u_0}_t)
    \,\diffns\rho
  \right\|_{ \lpn{p}{\P}{V_{l,0}} }  
\\&\leq
  \tilde{L}_{2p}
    \int^1_0
    \|
    G'_l( X^{0,u_0}_t + \rho[ X^{0,u_0+u_1}_t - X^{0,u_0}_t ] )
    -
    G'_l( X^{0,u_0}_t )
    \|_{ \lpn{2p}{\P}{L(H,V_{l,0})} }  
    \,\diffns\rho  
    .
\end{split}
\end{equation}
In the next step we combine~\eqref{calc21} with~\eqref{eq:1st.derivative.nonlinear.difference} 
and Jensen's inequality to obtain that for all 
$ p\in[2,\infty) $, 
$ \mathbf{u}=(u_0,u_1) \in H \times (\nzspace{H}) $
it holds that 
\begin{equation}
\label{eq:1st.infinitesimal.convergence}
\begin{split}
& \sup_{t\in[0,T]}
  \frac{
  \|
    X^{0,\theta^1_1(\mathbf{u})}_t - X^{0,\theta^1_0(\mathbf{u})}_t - X_t^{ 1,\mathbf{u} }
  \|_{
    \mathcal{L}^p ( \P ; H )
  }
  }{
    \| u_1 \|_{ H }
  }
  \leq
  \tilde{L}_{2p} \,
  \chi_{A,\eta}^{0,T} \,
  \Theta_{A,\eta,p,T}^{0,0,0}
  \big(
    |F|_{\Cb{1}(H,H)},
    |B|_{\Cb{1}(H,HS(U,H))}
  \big)
\\&\cdot
  \Bigg[
    \int_0^T 
    \int^1_0   
    \|
    F'( X^{0,u_0}_s + \rho[ X^{0,u_0+u_1}_s - X^{0,u_0}_s ] )
    - F'( X^{0,u_0}_s )
    \|_{ \lpn{2p}{\P}{L(H,H)} } 
      \,\diffns\rho 
      \,\diffns s\\
&   
    +
    \bigg[
      \tfrac{ p\, (p-1) }{ 2 }
      \int_0^T
      \int^1_0
    \|
    B'( X^{0,u_0}_s + \rho[ X^{0,u_0+u_1}_s - X^{0,u_0}_s ] )
    - B'( X^{0,u_0}_s )
    \|^2_{ \lpn{2p}{\P}{L(H,HS(U,H))} } 
      \,\diffns\rho
      \,\diffns s
    \bigg]^{ \nicefrac{1}{2} }
  \Bigg]
  .
\end{split}
\end{equation}
Furthermore, Lebesgue's theorem of dominated convergence 
and~\eqref{eq:1st.derivative.conv1}
yield that for all 
$ m \in \{ 1,2,\ldots,n \} $, 
$ l \in \{ 0,1 \} $,  
$ p \in [2,\infty) $,
$ u_0 \in H $
it holds that 
\begin{equation}
\label{eq:1st.derivative.integral.limit}
\begin{split}
&
  \limsup_{ H \ni u_1 \rightarrow 0 }
  \int^T_0
  \int^1_0
      \|
      G^{(m)}_l( X^{0,u_0}_s + \rho[ X^{0,u_0+u_1}_s - X^{0,u_0}_s ] )
      -
      G^{(m)}_l( X^{0,u_0}_s )
      \|^{ (l+1) }_{ \lpn{p}{\P}{ L^{(m)}( H, V_{l,0} ) } }
  \,\diffns\rho
  \,\diffns{s}
  = 0
  .
\end{split}
\end{equation}
Combining~\eqref{eq:1st.infinitesimal.convergence} 
with~\eqref{eq:tildeL.finite} and~\eqref{eq:1st.derivative.integral.limit} establishes  item~\eqref{item:lem_derivative:frechet} in the base case $ k=1 $.
For the induction step 
$
  \{ 1,2,\ldots,n-1 \} \ni k
  \to k+1 \in \{ 2,3,\ldots,n \}
$ 
assume that there exists a natural number 
$ k \in\{ 1,2,\ldots,n-1 \} $
such that item~\eqref{item:lem_derivative:frechet} holds 
for 
$k=1$, $k=2$, $\ldots\,$, $k=k$. 
Note that item~\eqref{item:lem_derivative:a_priori} ensures that for all 
$
  m \in \{ 1, 2, \ldots, n \}
$,  
$
  p \in (0,\infty)
$, 
$ x, y \in H $, 
$ v \in \nzspace{H} $
it holds that 
$
  d_{ m,p }(x,y) + \tilde{d}_{ m,p }(x,v) < \infty
$.
We also note that item~\eqref{item:thm_derivative.continuous} and the induction hypothesis assure that for all 
$
  m \in \{ 1, 2, \ldots, k \}
$, 
$
  p \in (0,\infty)
$, 
$
  x \in H
$ 
it holds that 
\begin{equation}
\label{eq:distance.convergence}
  \limsup_{ H \ni y \rightarrow x }
  d_{m,p}(x,y)
  =0
  \qquad
  \text{and}
  \qquad
  \limsup_{ \nzspace{H} \ni v \rightarrow 0 }
  \tilde{d}_{m,p}(x,v)
  =
  0
  .
\end{equation}
Next observe that~\eqref{eq:G.def} shows that
for all 
$ l \in \{ 0, 1 \} $,
$
  \mathbf{u}=(u_0,u_1,\ldots,u_{k+1}) 
  \in H^{k+2}
$,
$ t\in [0,T] $
it holds that 
\begin{equation}
\begin{split}
&
  \mathbf{G}^{\theta^{k+1}_1(\mathbf{u})}_{k,l}( t,X^{ k,\theta^{k+1}_1(\mathbf{u}) }_t ) =
  G'_l( X^{0,u_0+u_{k+1}}_t ) X^{ k,\theta^{k+1}_1(\mathbf{u}) }_t
\\&\quad+
  {\sum_{ \varpi \in \Pi^*_k }}
  G^{ (\#_\varpi) }_l( X^{0,u_0+u_{k+1}}_t )
  \big( 
        X_t^{ \#_{I^\varpi_1}, [ \theta^{k+1}_1(\mathbf{u}) ]_1^{ \varpi } }
        ,
        X_t^{ \#_{I^\varpi_2}, [ \theta^{k+1}_1(\mathbf{u}) ]_2^{ \varpi } }
        ,
        \dots
        ,
        X_t^{ \#_{I^\varpi_{\#_\varpi}}, [ \theta^{k+1}_1(\mathbf{u}) ]_{ \#_\varpi }^{ \varpi } }
  \big)
\end{split}
\end{equation}
and
\begin{equation}
\begin{split}
&
  \mathbf{G}^{\theta^{k+1}_0(\mathbf{u})}_{k,l}( t,X^{ k,\theta^{k+1}_0(\mathbf{u}) }_t ) 
  =
  G'_l( X^{0,u_0}_t ) X^{ k,\theta^{k+1}_0(\mathbf{u}) }_t
\\&\quad+
  {\sum_{ \varpi \in \Pi^*_k }}
  G^{ (\#_\varpi) }_l( X^{0,u_0}_t )
  \big( 
        X_t^{ \#_{I^\varpi_1}, [ \theta^{k+1}_0(\mathbf{u}) ]_1^{ \varpi } }
        ,
        X_t^{ \#_{I^\varpi_2}, [ \theta^{k+1}_0(\mathbf{u}) ]_2^{ \varpi } }
        ,
        \dots
        ,
        X_t^{ \#_{I^\varpi_{\#_\varpi}}, [ \theta^{k+1}_0(\mathbf{u}) ]_{ \#_\varpi }^{ \varpi } }
  \big).
\end{split}
\end{equation}
This implies that for all 
$ l \in \{ 0, 1 \} $,
$
  \mathbf{u}=(u_0,u_1,\ldots,u_{k+1}) 
  \in H^{k+2}
$,
$ t\in [0,T] $
it holds that 
\begin{equation}
\begin{split}
&
  \mathbf{G}^{\theta^{k+1}_1(\mathbf{u})}_{k,l}( t,X^{ k,\theta^{k+1}_1(\mathbf{u}) }_t )
  -
  \mathbf{G}^{\theta^{k+1}_0(\mathbf{u})}_{k,l}( t,X^{ k,\theta^{k+1}_0(\mathbf{u}) }_t )
\\&=
  G'_l( X^{0,u_0+u_{k+1}}_t ) X^{ k,\theta^{k+1}_1(\mathbf{u}) }_t
  -
  G'_l( X^{0,u_0}_t ) X^{ k,\theta^{k+1}_0(\mathbf{u}) }_t
\\&\quad+
  {\sum_{ \varpi \in \Pi^*_k }}
  \Big[
  G^{ (\#_\varpi) }_l( X^{0,u_0+u_{k+1}}_t )
  \big( 
        X_t^{ \#_{I^\varpi_1}, [ \theta^{k+1}_1(\mathbf{u}) ]_1^{ \varpi } }
        ,
        X_t^{ \#_{I^\varpi_2}, [ \theta^{k+1}_1(\mathbf{u}) ]_2^{ \varpi } }
        ,
        \dots
        ,
        X_t^{ \#_{I^\varpi_{\#_\varpi}}, [ \theta^{k+1}_1(\mathbf{u}) ]_{ \#_\varpi }^{ \varpi } }
  \big)
\\&\quad-
  G^{ (\#_\varpi) }_l( X^{0,u_0}_t )
  \big( 
        X_t^{ \#_{I^\varpi_1}, [ \theta^{k+1}_0(\mathbf{u}) ]_1^{ \varpi } }
        ,
        X_t^{ \#_{I^\varpi_2}, [ \theta^{k+1}_0(\mathbf{u}) ]_2^{ \varpi } }
        ,
        \dots
        ,
        X_t^{ \#_{I^\varpi_{\#_\varpi}}, [ \theta^{k+1}_0(\mathbf{u}) ]_{ \#_\varpi }^{ \varpi } }
  \big)
  \Big]
\\&=
  G'_l( X^{0,u_0}_t ) (X^{ k,\theta^{k+1}_1(\mathbf{u}) }_t-X^{ k,\theta^{k+1}_0(\mathbf{u}) }_t)
  +
  [G'_l( X^{0,u_0+u_{k+1}}_t )-G'_l( X^{0,u_0}_t )] X^{ k,\theta^{k+1}_1(\mathbf{u}) }_t
\\&\quad+
  {\sum_{ \varpi \in \Pi^*_k }}
  \Big[
  [G^{ (\#_\varpi) }_l( X^{0,u_0+u_{k+1}}_t )-G^{ (\#_\varpi) }_l( X^{0,u_0}_t )]
  \big( 
        X_t^{ \#_{I^\varpi_1}, [ \theta^{k+1}_1(\mathbf{u}) ]_1^{ \varpi } }
        ,
        X_t^{ \#_{I^\varpi_2}, [ \theta^{k+1}_1(\mathbf{u}) ]_2^{ \varpi } }
        ,
        \dots
        ,
\\&\quad
        X_t^{ \#_{I^\varpi_{\#_\varpi}}, [ \theta^{k+1}_1(\mathbf{u}) ]_{ \#_\varpi }^{ \varpi } }
  \big)
+
  G^{ (\#_\varpi) }_l( X^{0,u_0}_t )
  \big( 
        X_t^{ \#_{I^\varpi_1}, [ \theta^{k+1}_1(\mathbf{u}) ]_1^{ \varpi } }
        ,
        X_t^{ \#_{I^\varpi_2}, [ \theta^{k+1}_1(\mathbf{u}) ]_2^{ \varpi } }
        ,
        \dots
        ,
        X_t^{ \#_{I^\varpi_{\#_\varpi}}, [ \theta^{k+1}_1(\mathbf{u}) ]_{ \#_\varpi }^{ \varpi } }
  \big)
\\&\quad-
  G^{ (\#_\varpi) }_l( X^{0,u_0}_t )
  \big( 
        X_t^{ \#_{I^\varpi_1}, [ \theta^{k+1}_0(\mathbf{u}) ]_1^{ \varpi } }
        ,
        X_t^{ \#_{I^\varpi_2}, [ \theta^{k+1}_0(\mathbf{u}) ]_2^{ \varpi } }
        ,
        \dots
        ,
        X_t^{ \#_{I^\varpi_{\#_\varpi}}, [ \theta^{k+1}_0(\mathbf{u}) ]_{ \#_\varpi }^{ \varpi } }
  \big)
  \Big]
  .
\end{split}
\end{equation}
The fundamental theorem of calculus and~\eqref{eq:barG.def} hence yield that for all 
$ l \in \{ 0, 1 \} $,
$
  \mathbf{u}
  \in H^{k+2}
$,
$ t\in [0,T] $
it holds that 
\begin{equation}
\label{eq:kth.diff.coeff}
\begin{split}
&
  \mathbf{G}^{\theta^{k+1}_1(\mathbf{u})}_{k,l}( t,X^{ k,\theta^{k+1}_1(\mathbf{u}) }_t )
  -
  \mathbf{G}^{\theta^{k+1}_0(\mathbf{u})}_{k,l}( t,X^{ k,\theta^{k+1}_0(\mathbf{u}) }_t )
=
  \mathbf{\bar{G}}^{\mathbf{u}}_{k+1,l}
  ( 
  t,
  X^{ k,\theta^{k+1}_1(\mathbf{u}) }_t - X^{ k,\theta^{k+1}_0(\mathbf{u}) }_t 
  )
  .
\end{split}
\end{equation}
This, \eqref{eq:proof.derivatives}, and~\eqref{eq:G.def} imply that for all 
$
  \mathbf{u} \in H^{k+2}
$, 
$ t\in[0,T] $
it holds that 
\begin{equation}
\label{eq:derivative.processes}
\begin{split}
&
 \big[ X^{k,\theta^{k+1}_1(\mathbf{u})}_t-X^{k,\theta^{k+1}_0(\mathbf{u})}_t \big]_{ \P, \mathcal{B}(H) }
\\&=
    \int_0^t
    e^{ ( t - s ) A }
    \big[
  \mathbf{G}^{\theta^{k+1}_1(\mathbf{u})}_{k,0}( s,X^{ k,\theta^{k+1}_1(\mathbf{u}) }_s )
  -
  \mathbf{G}^{\theta^{k+1}_0(\mathbf{u})}_{k,0}( s,X^{ k,\theta^{k+1}_0(\mathbf{u}) }_s )
    \big]
  \, {\bf ds}
\\&\quad+
  \int_0^t
    e^{ ( t - s ) A }
    \big[
  \mathbf{G}^{\theta^{k+1}_1(\mathbf{u})}_{k,1}( s,X^{ k,\theta^{k+1}_1(\mathbf{u}) }_s )
  -
  \mathbf{G}^{\theta^{k+1}_0(\mathbf{u})}_{k,1}( s,X^{ k,\theta^{k+1}_0(\mathbf{u}) }_s )
    \big]
  \, \diffns W_s
\\&=
    \int_0^t
    e^{ ( t - s ) A }
    \,
    \mathbf{\bar{G}}^{ \mathbf{u} }_{ k+1, 0 }( s, X^{k,\theta^{k+1}_1(\mathbf{u})}_s-X^{k,\theta^{k+1}_0(\mathbf{u})}_s )
  \, {\bf ds}
\\&\quad+
  \int_0^t
    e^{ ( t - s ) A }
    \,
    \mathbf{\bar{G}}^{ \mathbf{u} }_{ k+1, 1 }( s, X^{k,\theta^{k+1}_1(\mathbf{u})}_s-X^{k,\theta^{k+1}_0(\mathbf{u})}_s )
  \, \diffns W_s
  .  
\end{split}
\end{equation}
Combining this with~\eqref{eq:Lip.G.n}, \eqref{eq:singular.G.n}, \eqref{eq:proof.derivatives}, and 
Proposition~2.7 in~\cite{AnderssonJentzenKurniawan2016arXiv} (with
$
H=H
$, 
$
U=U
$, 
$ T = T $, 
$ \eta = \eta $, 
$ p = p $, 
$ \alpha = 0 $, 
$
\hat{\alpha} = 0
$,
$ \beta = 0 $, 
$
\hat{\beta} 
= 
0
$,
$L_0=|F|_{\Cb{1}(H,H)}$, 
$
\hat{L}_0
=
    \smallsum_{ \varpi \in \Pi_{ k + 1 }^{ * } }
    | F |_{ 
      \Cb{ \#_\varpi }( H, H ) 
    }
      \smallprod_{ i = 1 }^{ \#_\varpi }
      \big\|\stochval{X^{
                  \#_{I^\varpi_i}, 
                  [ \mathbf{u} ]_i^{ \varpi }
                }}\big\|_{ \mathbb{L}^{ p \#_\varpi } }
$, 
$L_1=|B|_{\Cb{1}(H,HS(U,H))}$, 
$
\hat{L}_1
=
    \smallsum_{ \varpi \in \Pi_{ k + 1 }^{ * } }
    | B |_{ 
      \Cb{ \#_\varpi }( H, HS(U,H) ) 
    }
      \smallprod_{ i = 1 }^{ \#_\varpi }
      \big\|\stochval{X^{
                  \#_{I^\varpi_i}, 
                  [ \mathbf{u} ]_i^{ \varpi }
                }}\big\|_{ \mathbb{L}^{ p \#_\varpi } }
$, 
$
W=W
$, 
$ A=A $, 
$ \mathbf{F}=\mathbf{G}_{k+1,0}^{\mathbf{u}} $,
$ \mathbf{B}=\mathbf{G}_{k+1,1}^{\mathbf{u}} $, 
$
  \delta = 0
$, 
$ Y^1 = X^{ k,\theta^{k+1}_1(\mathbf{u}) } - X^{ k,\theta^{k+1}_0(\mathbf{u}) } $, 
$ Y^2 = X^{k+1,\mathbf{u}} $,
$ \lambda = 0 $ 
for 
$
  \mathbf{u}
  \in H^{k+2}
$, 
$ p \in [2,\infty) $
in the notation of Proposition~2.7 in~\cite{AnderssonJentzenKurniawan2016arXiv})  
implies that for all  
$ p \in [ 2, \infty ) $, 
$
  \mathbf{u}
  \in H^{k+2}
$
it holds that 
\begin{equation}
\begin{split}
& \sup_{t\in[0,T]}
  \big\|
    X^{ k,\theta^{k+1}_1(\mathbf{u}) }_t - X^{ k,\theta^{k+1}_0(\mathbf{u}) }_t
    - X_t^{ k+1,\mathbf{u} }
  \big\|_{
    \mathcal{L}^p ( \P ; H )
  }
  \leq
  \Theta_{A,\eta,p,T}^{0,0,0}
  \big(
    |F|_{\Cb{1}(H,H)},
    |B|_{\Cb{1}(H,HS(U,H))}
  \big)\\
& \cdot
  \sup_{t\in(0,T]}
  \Bigg[
  \Bigg\|
      \int_0^t
        e^{(t-s)A}
        \big(
          \mathbf{\bar{G}}_{k+1,0}^{\mathbf{u}}( s , X^{ k,\theta^{k+1}_1(\mathbf{u}) }_s - X^{ k,\theta^{k+1}_0(\mathbf{u}) }_s )
          -
          \mathbf{G}_{k+1,0}^{\mathbf{u}}( s , X^{ k,\theta^{k+1}_1(\mathbf{u}) }_s - X^{ k,\theta^{k+1}_0(\mathbf{u}) }_s )
        \big)
      \,{ \bf ds }\\
& +
      \int_0^t
        e^{(t-s)A}
        \big(
          \mathbf{\bar{G}}_{k+1,1}^{\mathbf{u}}( s , X^{ k,\theta^{k+1}_1(\mathbf{u}) }_s - X^{ k,\theta^{k+1}_0(\mathbf{u}) }_s )
          -
          \mathbf{G}_{k+1,1}^{\mathbf{u}}( s , X^{ k,\theta^{k+1}_1(\mathbf{u}) }_s - X^{ k,\theta^{k+1}_0(\mathbf{u}) }_s )
        \big)
      \,\diffns W_s
    \Bigg\|_{L^p(\P;H)}\Bigg]
  .
\end{split}
\end{equation}
The Burkholder-Davis-Gundy type inequality
in Lemma~7.7 in Da Prato \& Zabczyk~\cite{dz92} hence shows that 
for all  
$ p \in [ 2, \infty ) $, 
$
  \mathbf{u}
  \in H^{k+2}
$
it holds that 
\begin{equation}
\label{calc2}
\begin{split}
& \sup_{t\in[0,T]}
  \big\|
    X^{ k,\theta^{k+1}_1(\mathbf{u}) }_t - X^{ k,\theta^{k+1}_0(\mathbf{u}) }_t
    - X_t^{ k+1,\mathbf{u} }
  \big\|_{
    \mathcal{L}^p ( \P ; H )
  }
  \leq
  \chi_{A,\eta}^{0,T} \,
  \Theta_{A,\eta,p,T}^{0,0,0}
  \big(
    |F|_{\Cb{1}(H,H)},
    |B|_{\Cb{1}(H,HS(U,H))}
  \big)
  \\
& \cdot
  \Bigg[
    \int_0^T    
        \|
          \mathbf{\bar{G}}_{k+1,0}^{\mathbf{u}}( t , X^{ k,\theta^{k+1}_1(\mathbf{u}) }_t - X^{ k,\theta^{k+1}_0(\mathbf{u}) }_t )
          -
          \mathbf{G}_{k+1,0}^{\mathbf{u}}( t , X^{ k,\theta^{k+1}_1(\mathbf{u}) }_t - X^{ k,\theta^{k+1}_0(\mathbf{u}) }_t )
        \|_{\mathcal{L}^p(\P;H)}
      \,\diffns t
+
    \left[
      \tfrac{ p\, (p-1) }{ 2 }
    \right]^{ \nicefrac{1}{2} }
\\&\cdot
    \left[
      \int_0^T
        \|
          \mathbf{\bar{G}}_{k+1,1}^{\mathbf{u}}( t , X^{ k,\theta^{k+1}_1(\mathbf{u}) }_t - X^{ k,\theta^{k+1}_0(\mathbf{u}) }_t )
          -
          \mathbf{G}_{k+1,1}^{\mathbf{u}}( t , X^{ k,\theta^{k+1}_1(\mathbf{u}) }_t - X^{ k,\theta^{k+1}_0(\mathbf{u}) }_t )
        \|^2_{\mathcal{L}^p(\P;HS(U,H))}
      \,\diffns t
    \right]^{ \nicefrac{1}{2} }
  \Bigg]
  .
\end{split}
\end{equation}
Next observe that for all $m\in\N$ 
it holds that 
\begin{equation}
\begin{split}
&
  \Pi_{m+1}
  =
  \Big\{ 
  \varpi \cup \big\{\{m+1\}\big\}
  \colon 
  \varpi \in \Pi_m 
  \Big\}
\\&
  \biguplus
  \Big\{
    \big\{
      I^\varpi_1, I^\varpi_2, \ldots, I^\varpi_{i-1}, 
      I^\varpi_i \cup \{m+1\}, I^\varpi_{i+1}, I^\varpi_{i+2}, 
      \ldots, I^\varpi_{\#_\varpi}
    \big\}
    \colon
    i \in \{1,2,\ldots,\#_\varpi\}, \,
    \varpi \in \Pi_m
  \Big\}
  .
\end{split}
\end{equation}
This implies that for all $m\in\N$ 
it holds that 
\begin{equation}
\label{eq:recursive.rep}
\begin{split}
&
  \Pi^*_{m+1}
  =
  \Big\{ 
  \varpi \cup \big\{\{m+1\}\big\}
  \colon 
  \varpi \in \Pi_m 
  \Big\}
\\&
  \biguplus
  \Big\{
    \big\{
      I^\varpi_1, I^\varpi_2, \ldots, I^\varpi_{i-1}, 
      I^\varpi_i \cup \{m+1\}, I^\varpi_{i+1}, I^\varpi_{i+2}, 
      \ldots, I^\varpi_{\#_\varpi}
    \big\}
    \colon
    i \in \{1,2,\ldots,\#_\varpi\}, \,
    \varpi \in \Pi^*_m
  \Big\}
\\&=
  \Big\{\big\{\{1,2,\ldots,m\},\{m+1\}\big\}\Big\}
  \biguplus
  \bigg[
  \bigcup_{\varpi\in\Pi^*_m}
  \Big(
  \Big\{ 
  \varpi \cup \big\{\{m+1\}\big\}
  \Big\}
\\&
  \biguplus
  \Big\{
    \big\{
      I^\varpi_1, I^\varpi_2, \ldots, I^\varpi_{i-1}, 
      I^\varpi_i \cup \{m+1\}, I^\varpi_{i+1}, I^\varpi_{i+2}, 
      \ldots, I^\varpi_{\#_\varpi}
    \big\}
    \colon
    i \in \{1,2,\ldots,\#_\varpi\}
  \Big\}\Big)\bigg]
  .
\end{split}
\end{equation}
This and~\eqref{eq:G.def} prove that for all 
$ l\in\{0,1\} $, 
$
  \mathbf{u} 
  = ( u_0, u_1, \dots, u_{k+1} ) 
  \in H^{k+2}
$,
$ x \in H $, 
$ t \in [0,T] $ 
it holds that
\begin{equation}
\begin{split}
&\!\!\!\!
  \mathbf{G}^{\mathbf{u}}_{ k+1,l }( t, x )
  =
  G'_l( X^{0,u_0}_t ) x
  +
  {\sum_{ \varpi \in \Pi^*_{k+1} }}
  G^{ (\#_\varpi) }_l( X^{0,u_0}_t )
  \big( 
        X_t^{ \#_{I^\varpi_1}, [ \mathbf{u} ]_1^{ \varpi } }
        ,
        X_t^{ \#_{I^\varpi_2}, [ \mathbf{u} ]_2^{ \varpi } }
        ,
        \dots
        ,
        X_t^{ \#_{I^\varpi_{\#_\varpi}}, [\mathbf{u} ]_{ \#_\varpi }^{ \varpi } }
  \big)
\\&\!\!\!\!=
  G'_l( X^{0,u_0}_t ) \, x
  +
  G''_l( X^{0,u_0}_t )\big(
    X^{ k,\theta^{k+1}_0(\mathbf{u}) }_t
    ,
    X^{ 1,( u_0, u_{k+1} ) }_t
  \big)
\\&\!\!\!\!+
  \sum_{ \varpi \in \Pi^*_k }
  \Bigg[
    G^{ (\#_\varpi + 1) }_l( X^{0,u_0}_t )\Big(
      X^{ \#_{I^\varpi_1}, [\theta^{k+1}_0(\mathbf{u})]^\varpi_1 }_t
      ,
      X^{ \#_{I^\varpi_2}, [\theta^{k+1}_0(\mathbf{u})]^\varpi_2 }_t
      ,
      \ldots
      ,
      X^{ \#_{I^\varpi_{\#_\varpi}}, [\theta^{k+1}_0(\mathbf{u})]^\varpi_{\#_\varpi} }_t
      ,
      X^{ 1,(u_0,u_{k+1}) }_t
    \Big)
\\&\!\!\!\!+\sum^{\#_\varpi}_{ i=1 }
    G^{ (\#_\varpi) }_l( X^{0,u_0}_t )\Big(
      X^{ \#_{I^\varpi_1}, [\theta^{k+1}_0(\mathbf{u})]^\varpi_1 }_t
      ,
      X^{ \#_{I^\varpi_2}, [\theta^{k+1}_0(\mathbf{u})]^\varpi_2 }_t
      ,
      \dots
      ,
      X^{ \#_{I^\varpi_{i-1}}, [\theta^{k+1}_0(\mathbf{u})]^\varpi_{i-1} }_t
      ,
\\&\!\!\!\!\quad
      X^{ \#_{I^\varpi_i}+1, ([\theta^{k+1}_0(\mathbf{u})]^\varpi_i, u_{k+1}) }_t
      ,
      X^{ \#_{I^\varpi_{i+1}}, [\theta^{k+1}_0(\mathbf{u})]^\varpi_{i+1} }_t
      ,
      X^{ \#_{I^\varpi_{i+2}}, [\theta^{k+1}_0(\mathbf{u})]^\varpi_{i+2} }_t
      ,
      \ldots
      ,
      X^{ \#_{I^\varpi_{\#_\varpi}}, [\theta^{k+1}_0(\mathbf{u})]^\varpi_{\#_\varpi} }_t
    \Big)
  \Bigg].
\end{split}
\end{equation}
Moreover, observe that~\eqref{eq:barG.def} shows that for all 
$ l\in\{0,1\} $, 
$
  \mathbf{u} 
  = ( u_0, u_1, \dots, u_{k+1} ) 
  \in H^{k+2}
$,
$ x \in H $, 
$ t \in [0,T] $ 
it holds that
\begin{equation}
\begin{split}
&
  \mathbf{\bar{G}}^{ \mathbf{u} }_{
    k+1, l
  }( t, x )
  =
    G'_l( X_t^{ 0, u_0 } ) \, x
\\&+
    \int^1_0
    G_l''\big(
      X_t^{ 0, u_0 } + \rho [ X^{0,u_0+u_{k+1}}_t - X^{0,u_0}_t ]
    \big)
    \big(
      X_t^{ k, \theta^{k+1}_1( \mathbf{u} ) } ,
      X^{0,u_0+u_{k+1}}_t - X^{0,u_0}_t
    \big)
  \, \diffns\rho
\\
&
  +
  \sum_{
    \varpi \in \Pi_k^{ * }
  }
  \Bigg[
    \int^1_0
    G_l^{ ( \#_\varpi + 1 ) }\big(
      X_t^{ 0, u_0 } 
      + 
      \rho [X^{0,u_0+u_{k+1}}_t - X^{0,u_0}_t]
    \big)
    \big(
        X_t^{ \#_{I^\varpi_1}, [ \theta^{k+1}_1(\mathbf{u}) ]_1^{ \varpi } }
        ,
        X_t^{ \#_{I^\varpi_2}, [ \theta^{k+1}_1(\mathbf{u}) ]_2^{ \varpi } }
        ,
        \dots
        ,
\\&\quad
        X_t^{ \#_{I^\varpi_{\#_\varpi}}, [ \theta^{k+1}_1(\mathbf{u}) ]_{ \#_\varpi }^{ \varpi } }
      ,
      X^{0,u_0+u_{k+1}}_t - X^{0,u_0}_t
    \big)
    \, \diffns\rho
\\&+
  \sum^{\#_\varpi}_{ i=1 }
  \displaystyle
    G^{ ( \#_\varpi ) }_l( X_t^{ 0, u_0 } )\big(
      X_t^{ \#_{I^\varpi_1}, [ \theta^{k+1}_0(\mathbf{u}) ]_1^{ \varpi } }
      ,
      X_t^{ \#_{I^\varpi_2}, [ \theta^{k+1}_0(\mathbf{u}) ]_2^{ \varpi } }
      ,
      \dots
      ,
      X_t^{ \#_{I^\varpi_{i-1}}, [ \theta^{k+1}_0(\mathbf{u}) ]_{i-1}^{ \varpi } }
      ,
      X_t^{ \#_{I^\varpi_i}, [\theta^{k+1}_1(\mathbf{u})]^\varpi_i }
\\&\quad-
      X_t^{ \#_{I^\varpi_i}, [ \theta^{k+1}_0(\mathbf{u}) ]_i^{ \varpi } }
      ,
      X_t^{ \#_{I^\varpi_{i+1}}, [ \theta^{k+1}_1( \mathbf{u} ) ]_{i+1}^\varpi }
      ,
      X_t^{ \#_{I^\varpi_{i+2}}, [ \theta^{k+1}_1( \mathbf{u} ) ]_{i+2}^\varpi }
      ,
      \dots
      ,
      X_t^{ \#_{I^\varpi_{\#_\varpi}}, [ \theta^{k+1}_1( \mathbf{u} ) ]_{ \#_\varpi }^\varpi }
    \big)
  \Bigg]
  .
\end{split}
\end{equation}
This implies that for all 
$ l\in\{0,1\} $, 
$
  \mathbf{u} 
  = ( u_0, u_1, \dots, u_{k+1} ) 
  \in H^{k+2}
$,
$ t \in [0,T] $
it holds that 
\begin{equation}
\begin{split}
&\!\!\!\!
  \mathbf{\bar{G}}^{\mathbf{u}}_{ k+1,l }( t, X^{k,\theta^{k+1}_1(\mathbf{u})}_t - X^{k,\theta^{k+1}_0(\mathbf{u})}_t )
  -
  \mathbf{G}^{\mathbf{u}}_{ k+1,l }( t, X^{k,\theta^{k+1}_1(\mathbf{u})}_t - X^{k,\theta^{k+1}_0(\mathbf{u})}_t )
\\&\!\!\!\!=
  \sum_{
    \varpi \in \Pi_k^{ * }
  }
  \sum^{\#_\varpi}_{ i=1 }
    \Big[
    G^{ ( \#_\varpi ) }_l( X_t^{ 0,u_0 } )\big(
      X_t^{ \#_{I^\varpi_1}, [ \theta^{k+1}_0(\mathbf{u}) ]_1^{ \varpi } }
      ,
      X_t^{ \#_{I^\varpi_2}, [ \theta^{k+1}_0(\mathbf{u}) ]_2^{ \varpi } }
      ,
      \dots
      ,
      X_t^{ \#_{I^\varpi_{i-1}}, [ \theta^{k+1}_0(\mathbf{u}) ]_{i-1}^{ \varpi } }
      ,
\\&\!\!\!\!
      X_t^{ \#_{I^\varpi_i}, [\theta^{k+1}_1(\mathbf{u})]^\varpi_i }
      -
      X_t^{ \#_{I^\varpi_i}, [ \theta^{k+1}_0(\mathbf{u}) ]_i^{ \varpi } }
      -
      X_t^{ \#_{I^\varpi_i}+1, ( [ \theta^{k+1}_0(\mathbf{u}) ]_i^{ \varpi }, u_{k+1} ) }
      ,
      X_t^{ \#_{I^\varpi_{i+1}}, [ \theta^{k+1}_1( \mathbf{u} ) ]_{i+1}^\varpi }
      ,
\\&\!\!\!\!
      X_t^{ \#_{I^\varpi_{i+2}}, [ \theta^{k+1}_1( \mathbf{u} ) ]_{i+2}^\varpi }
      ,
      \dots
      ,
      X_t^{ \#_{I^\varpi_{\#_\varpi}}, [ \theta^{k+1}_1( \mathbf{u} ) ]_{ \#_\varpi }^\varpi }
    \big)
\\&\!\!\!\!+
    G^{ ( \#_\varpi ) }_l( X_t^{ 0,u_0 } )
    \big(
      X_t^{ \#_{I^\varpi_1}, [ \theta^{k+1}_0(\mathbf{u}) ]_1^{ \varpi } }
      ,
      X_t^{ \#_{I^\varpi_2}, [ \theta^{k+1}_0(\mathbf{u}) ]_2^{ \varpi } }
      ,
      \dots
      ,
      X_t^{ \#_{I^\varpi_{i-1}}, [ \theta^{k+1}_0(\mathbf{u}) ]_{i-1}^{ \varpi } }
      ,
\\&\!\!\!\!
      X_t^{ \#_{I^\varpi_i}+1, ( [ \theta^{k+1}_0(\mathbf{u}) ]_i^{ \varpi }, u_{k+1} ) }
      ,
      X_t^{ \#_{I^\varpi_{i+1}}, [ \theta^{k+1}_1( \mathbf{u} ) ]_{i+1}^{ \varpi } }
      ,
      X_t^{ \#_{I^\varpi_{i+2}}, [ \theta^{k+1}_1( \mathbf{u} ) ]_{i+2}^{ \varpi } }
      ,
      \dots
      ,
      X_t^{ \#_{I^\varpi_{\#_\varpi}}, [ \theta^{k+1}_1( \mathbf{u} ) ]_{ \#_\varpi }^{ \varpi } } 
    \big) 
\\&\!\!\!\!-
    G^{ ( \#_\varpi ) }_l( X_t^{ 0,u_0 } )
    \big(
      X_t^{ \#_{I^\varpi_1}, [ \theta^{k+1}_0(\mathbf{u}) ]_1^{ \varpi } }
      ,
      X_t^{ \#_{I^\varpi_2}, [ \theta^{k+1}_0(\mathbf{u}) ]_2^{ \varpi } }
      ,
      \dots
      ,
      X_t^{ \#_{I^\varpi_{i-1}}, [ \theta^{k+1}_0(\mathbf{u}) ]_{i-1}^{ \varpi } }
      ,
\\&\!\!\!\!
      X_t^{ \#_{I^\varpi_i}+1, ( [ \theta^{k+1}_0(\mathbf{u}) ]_i^{ \varpi }, u_{k+1} ) }
      ,
      X_t^{ \#_{I^\varpi_{i+1}}, [ \theta^{k+1}_0( \mathbf{u} ) ]_{i+1}^{ \varpi } }
      ,
      X_t^{ \#_{I^\varpi_{i+2}}, [ \theta^{k+1}_0( \mathbf{u} ) ]_{i+2}^{ \varpi } }
      ,
      \dots
      ,
      X_t^{ \#_{I^\varpi_{\#_\varpi}}, [ \theta^{k+1}_0( \mathbf{u} ) ]_{ \#_\varpi }^{ \varpi } } 
    \big) 
    \Big]  
\\
&\!\!\!\!
    +
      \sum_{
        \varpi \in \Pi_k
      }
    \bigg[
    \int^1_0
    \big[
    G_l^{ ( \#_\varpi + 1 ) }\big(
      X_t^{ 0, u_0 } 
      + 
      \rho [X^{0,u_0+u_{k+1}}_t-X^{0,u_0}_t]
    \big)
    -
    G_l^{ ( \#_\varpi + 1 ) }\big(
      X_t^{ 0,u_0 } 
    \big)
    \big]
    \big(
      X_t^{ \#_{I^\varpi_1}, [ \theta^{k+1}_1( \mathbf{u} ) ]_1^\varpi }
      ,
\\&\!\!\!\!
      X_t^{ \#_{I^\varpi_2}, [ \theta^{k+1}_1( \mathbf{u} ) ]_2^\varpi }
      ,
      \dots
      ,
      X_t^{ \#_{I^\varpi_{\#_\varpi}}, [ \theta^{k+1}_1( \mathbf{u} ) ]_{ \#_\varpi }^{ \varpi } }
      ,
      X^{0,u_0+u_{k+1}}_t-X^{0,u_0}_t
    \big)
    \, \diffns\rho
\\&\!\!\!\!+
    G_l^{ ( \#_\varpi + 1 ) }(
      X_t^{ 0,u_0 } 
    )\big(
    X_t^{ \#_{I^\varpi_1}, [ \theta^{k+1}_0( \mathbf{u} ) ]_1^\varpi }
    ,
    X_t^{ \#_{I^\varpi_2}, [ \theta^{k+1}_0( \mathbf{u} ) ]_2^\varpi }
    ,
    \ldots
    ,
    X_t^{ \#_{I^\varpi_{\#_\varpi}}, [ \theta^{k+1}_0( \mathbf{u} ) ]_{\#_\varpi}^\varpi }
    ,
\\&\!\!\!\!
    X^{0,u_0+u_{k+1}}_t-X^{0,u_0}_t
    -
    X_t^{ 1,( u_0, u_{k+1} ) }
    \big)
\\&\!\!\!\!+
    G_l^{ ( \#_\varpi + 1 ) }(
      X_t^{ 0,u_0 } 
    )\big(
    X_t^{ \#_{I^\varpi_1}, [ \theta^{k+1}_1( \mathbf{u} ) ]_1^\varpi }
    ,
    X_t^{ \#_{I^\varpi_2}, [ \theta^{k+1}_1( \mathbf{u} ) ]_2^\varpi }
    ,
    \ldots
    ,
    X_t^{ \#_{I^\varpi_{\#_\varpi}}, [ \theta^{k+1}_1( \mathbf{u} ) ]_{\#_\varpi}^\varpi }
    ,
    X^{0,u_0+u_{k+1}}_t-X^{0,u_0}_t
    \big)
\\&\!\!\!\!-
    G_l^{ ( \#_\varpi + 1 ) }(
      X_t^{ 0,u_0 } 
    )\big(
    X_t^{ \#_{I^\varpi_1}, [ \theta^{k+1}_0( \mathbf{u} ) ]_1^\varpi }
    ,
    X_t^{ \#_{I^\varpi_2}, [ \theta^{k+1}_0( \mathbf{u} ) ]_2^\varpi }
    ,
    \ldots
    ,
    X_t^{ \#_{I^\varpi_{\#_\varpi}}, [ \theta^{k+1}_0( \mathbf{u} ) ]_{\#_\varpi}^\varpi }
    ,
    X^{0,u_0+u_{k+1}}_t-X^{0,u_0}_t
    \big)
    \bigg]
    .
\end{split}
\end{equation}
This assures that for all 
$ l\in\{0,1\} $, 
$
  \mathbf{u} 
  = ( u_0, u_1, \dots, u_{k+1} ) 
  \in H^{k+2}
$,
$ t \in [0,T] $
it holds that 
\begin{equation}
\label{eq:explain.start}
\begin{split}
&\!\!
  \mathbf{\bar{G}}^{\mathbf{u}}_{ k+1,l }( t, X^{k,\theta^{k+1}_1(\mathbf{u})}_t - X^{k,\theta^{k+1}_0(\mathbf{u})}_t )
  -
  \mathbf{G}^{\mathbf{u}}_{ k+1,l }( t, X^{k,\theta^{k+1}_1(\mathbf{u})}_t - X^{k,\theta^{k+1}_0(\mathbf{u})}_t )
\\&\!\!=
  \sum_{
    \varpi \in \Pi_k^{ * }
  }
  \sum^{\#_\varpi}_{ i=1 }
    \Bigg[
    G^{ ( \#_\varpi ) }_l( X_t^{ 0,u_0 } )\big(
      X_t^{ \#_{I^\varpi_1}, [ \theta^{k+1}_0(\mathbf{u}) ]_1^{ \varpi } }
      ,
      X_t^{ \#_{I^\varpi_2}, [ \theta^{k+1}_0(\mathbf{u}) ]_2^{ \varpi } }
      ,
      \dots
      ,
      X_t^{ \#_{I^\varpi_{i-1}}, [ \theta^{k+1}_0(\mathbf{u}) ]_{i-1}^{ \varpi } }
      ,
\\&\!\!
      X_t^{ \#_{I^\varpi_i}, [\theta^{k+1}_1(\mathbf{u})]^\varpi_i }
      -
      X_t^{ \#_{I^\varpi_i}, [ \theta^{k+1}_0(\mathbf{u}) ]_i^{ \varpi } }
      -
      X_t^{ \#_{I^\varpi_i}+1, ( [ \theta^{k+1}_0(\mathbf{u}) ]_i^{ \varpi }, u_{k+1} ) }
      ,
      X_t^{ \#_{I^\varpi_{i+1}}, [ \theta^{k+1}_1( \mathbf{u} ) ]_{i+1}^\varpi }
      ,
\\&\!\!
      X_t^{ \#_{I^\varpi_{i+2}}, [ \theta^{k+1}_1( \mathbf{u} ) ]_{i+2}^\varpi }
      ,
      \dots
      ,
      X_t^{ \#_{I^\varpi_{\#_\varpi}}, [ \theta^{k+1}_1( \mathbf{u} ) ]_{ \#_\varpi }^\varpi }
    \big)
\\&\!\!+
    \sum^{\#_\varpi}_{j=i+1}
    G^{ ( \#_\varpi ) }_l( X_t^{ 0,u_0 } )
    \big(
      X_t^{ \#_{I^\varpi_1}, [ \theta^{k+1}_0(\mathbf{u}) ]_1^{ \varpi } }
      ,
      X_t^{ \#_{I^\varpi_2}, [ \theta^{k+1}_0(\mathbf{u}) ]_2^{ \varpi } }
      ,
      \dots
      ,
      X_t^{ \#_{I^\varpi_{i-1}}, [ \theta^{k+1}_0(\mathbf{u}) ]_{i-1}^{ \varpi } }
      ,
\\&\!\!
      X_t^{ \#_{I^\varpi_i}+1, ( [ \theta^{k+1}_0(\mathbf{u}) ]_i^{ \varpi }, u_{k+1} ) }
      ,
      X_t^{ \#_{I^\varpi_{i+1}}, [ \theta^{k+1}_0( \mathbf{u} ) ]_{i+1}^{ \varpi } }
      ,
      X_t^{ \#_{I^\varpi_{i+2}}, [ \theta^{k+1}_0( \mathbf{u} ) ]_{i+2}^{ \varpi } }
      ,
      \dots
      ,
      X_t^{ \#_{I^\varpi_{j-1}}, [ \theta^{k+1}_0( \mathbf{u} ) ]_{ j-1 }^{ \varpi } }
      ,
\\&\!\!
      X_t^{ \#_{I^\varpi_j}, [ \theta^{k+1}_1( \mathbf{u} ) ]_j^{ \varpi } }
      -
      X_t^{ \#_{I^\varpi_j}, [ \theta^{k+1}_0( \mathbf{u} ) ]_j^{ \varpi } }
      ,
      X_t^{ \#_{I^\varpi_{j+1}}, [ \theta^{k+1}_1( \mathbf{u} ) ]_{ j+1 }^{ \varpi } }
      ,
      X_t^{ \#_{I^\varpi_{j+2}}, [ \theta^{k+1}_1( \mathbf{u} ) ]_{ j+2 }^{ \varpi } }
      ,
      \dots 
      ,
      X_t^{ \#_{I^\varpi_{\#_\varpi}}, [ \theta^{k+1}_1( \mathbf{u} ) ]_{ \#_\varpi }^{ \varpi } } 
    \big) 
    \Bigg]  
\\
&\!\!
    +
      \sum_{
        \varpi \in \Pi_k
      }
    \Bigg[
    \int^1_0
    \big[
    G_l^{ ( \#_\varpi + 1 ) }\big(
      X_t^{ 0, u_0 } 
      + 
      \rho [X^{0,u_0+u_{k+1}}_t-X^{0,u_0}_t]
    \big)
    -
    G_l^{ ( \#_\varpi + 1 ) }\big(
      X_t^{ 0,u_0 } 
    \big)
    \big]
    \big(
      X_t^{ \#_{I^\varpi_1}, [ \theta^{k+1}_1( \mathbf{u} ) ]_1^\varpi }
      ,
\\&\!\!
      X_t^{ \#_{I^\varpi_2}, [ \theta^{k+1}_1( \mathbf{u} ) ]_2^\varpi }
      ,
      \dots
      ,
      X_t^{ \#_{I^\varpi_{\#_\varpi}}, [ \theta^{k+1}_1( \mathbf{u} ) ]_{ \#_\varpi }^{ \varpi } }
      ,
      X^{0,u_0+u_{k+1}}_t-X^{0,u_0}_t
    \big)
    \, \diffns\rho
\\&\!\!+
    G_l^{ ( \#_\varpi + 1 ) }(
      X_t^{ 0,u_0 } 
    )\big(
    X_t^{ \#_{I^\varpi_1}, [ \theta^{k+1}_0( \mathbf{u} ) ]_1^\varpi }
    ,
    X_t^{ \#_{I^\varpi_2}, [ \theta^{k+1}_0( \mathbf{u} ) ]_2^\varpi }
    ,
    \ldots
    ,
    X_t^{ \#_{I^\varpi_{\#_\varpi}}, [ \theta^{k+1}_0( \mathbf{u} ) ]_{\#_\varpi}^\varpi }
    ,
\\&\!\!
    X^{0,u_0+u_{k+1}}_t-X^{0,u_0}_t
    -
    X_t^{ 1,( u_0, u_{k+1} ) }
    \big)
\\&\!\!+
    \sum^{\#_\varpi}_{i=1}
    G_l^{ ( \#_\varpi + 1 ) }(
      X_t^{ 0,u_0 } 
    )\big(
    X_t^{ \#_{I^\varpi_1}, [ \theta^{k+1}_0( \mathbf{u} ) ]_1^\varpi }
    ,
    X_t^{ \#_{I^\varpi_2}, [ \theta^{k+1}_0( \mathbf{u} ) ]_2^\varpi }
    ,
    \ldots
    ,
    X_t^{ \#_{I^\varpi_{i-1}}, [ \theta^{k+1}_0( \mathbf{u} ) ]_{ i-1 }^{ \varpi } }
    ,
\\&\!\!
    X_t^{ \#_{I^\varpi_i}, [ \theta^{k+1}_1( \mathbf{u} ) ]_i^\varpi }
    -
    X_t^{ \#_{I^\varpi_i}, [ \theta^{k+1}_0( \mathbf{u} ) ]_i^\varpi }
    ,
    X_t^{ \#_{I^\varpi_{i+1}}, [ \theta^{k+1}_1( \mathbf{u} ) ]_{ i+1 }^{ \varpi } }
    ,
    X_t^{ \#_{I^\varpi_{i+2}}, [ \theta^{k+1}_1( \mathbf{u} ) ]_{ i+2 }^{ \varpi } }
    ,
    \ldots
    ,
    X_t^{ \#_{I^\varpi_{\#_\varpi}}, [ \theta^{k+1}_1( \mathbf{u} ) ]_{\#_\varpi}^\varpi }
    ,
\\&\!\!
    X^{0,u_0+u_{k+1}}_t-X^{0,u_0}_t
    \big)
    \Bigg]
    .
\end{split}
\end{equation}
Furthermore, H\"{o}lder's inequality shows that for all 
$ l\in\{0,1\} $,
$ p\in[2,\infty) $, 
$ \varpi \in \Pi^*_k $, 
$ j \in \{1,2,\ldots,\#_\varpi\} $, 
$ m \in \{j+1,j+2,\ldots,\#_\varpi\} $, 
$
  \mathbf{u} 
  = ( u_0, u_1, \dots, u_{k+1} ) 
  \in \times^{k+1}_{i=0} H^{[i]}
$,
$ t \in (0,T] $
it holds that 
\begin{equation}
\begin{split}
&
  \frac{1}{
    \prod^{k+1}_{ i=1 }
    \|u_i\|_H
  } \,
  \big\|
    G^{ ( \#_\varpi ) }_l( X_t^{ 0,u_0 } )\big(
      X_t^{ \#_{I^\varpi_1}, [ \theta^{k+1}_0(\mathbf{u}) ]_1^{ \varpi } }
      ,
      X_t^{ \#_{I^\varpi_2}, [ \theta^{k+1}_0(\mathbf{u}) ]_2^{ \varpi } }
      ,
      \dots
      ,
      X_t^{ \#_{I^\varpi_{j-1}}, [ \theta^{k+1}_0(\mathbf{u}) ]_{j-1}^{ \varpi } }
      ,
\\&\quad
      X_t^{ \#_{I^\varpi_j}, [\theta^{k+1}_1(\mathbf{u})]^\varpi_j }
      -
      X_t^{ \#_{I^\varpi_j}, [ \theta^{k+1}_0(\mathbf{u}) ]_j^{ \varpi } }
      -
      X_t^{ \#_{I^\varpi_j}+1, ( [ \theta^{k+1}_0(\mathbf{u}) ]_j^{ \varpi }, u_{k+1} ) }
      ,
      X_t^{ \#_{I^\varpi_{j+1}}, [ \theta^{k+1}_1( \mathbf{u} ) ]_{j+1}^\varpi }
      ,
\\&\quad
      X_t^{ \#_{I^\varpi_{j+2}}, [ \theta^{k+1}_1( \mathbf{u} ) ]_{j+2}^\varpi }
      ,
      \dots
      ,
      X_t^{ \#_{I^\varpi_{\#_\varpi}}, [ \theta^{k+1}_1( \mathbf{u} ) ]_{ \#_\varpi }^\varpi }
    \big)
  \big\|_{ \lpn{p}{\P}{ V_{l,0} } }
\\&\leq
  |G_l|_{ \Cb{\#_\varpi}(H, V_{l,0}) } \,
  \Bigg[
    \prod^{j-1}_{ i=1 }
    \frac{
      \|X^{\#_{I^\varpi_i}, [\theta^{k+1}_0(\mathbf{u})]^\varpi_i}_t\|_{ \lpn{p \#_\varpi}{\P}{H} }
    }{
      \prod^{ \#_{ I^\varpi_i } }_{ q=1 }
      \| u_{ I^\varpi_{ i,q } } \|_H
    }
  \Bigg]
  \Bigg[
    \prod^{ \#_\varpi }_{ i=j+1 }
    \frac{
      \|X^{\#_{I^\varpi_i}, [\theta^{k+1}_1(\mathbf{u})]^\varpi_i}_t\|_{ \lpn{p \#_\varpi}{\P}{H} }
    }{
      \prod^{ \#_{ I^\varpi_i } }_{ q=1 }
      \| u_{ I^\varpi_{ i,q } } \|_H
    }
  \Bigg]
\\&\quad\cdot
    \frac{
      \|
      X_t^{ \#_{I^\varpi_j}, [\theta^{k+1}_1(\mathbf{u})]^\varpi_j }
      -
      X_t^{ \#_{I^\varpi_j}, [ \theta^{k+1}_0(\mathbf{u}) ]_j^{ \varpi } }
      -
      X_t^{ \#_{I^\varpi_j}+1, ( [ \theta^{k+1}_0(\mathbf{u}) ]_j^{ \varpi }, u_{k+1} ) }
      \|_{ \lpn{p \#_\varpi}{\P}{H} }
    }{
      \|u_{k+1}\|_H
      \prod^{ \#_{ I^\varpi_j } }_{ q=1 }
      \| u_{ I^\varpi_{ j,q } } \|_H
    }
\\&\leq
    |G_l|_{ \Cb{\#_\varpi}( H,V_{l,0} ) } \,
      L^{\mathbf{0}_k}_{ \varpi\setminus\{I^\varpi_j\}, p\,\#_\varpi }
      \,
     \tilde{d}_{
       \#_{I^\varpi_j} + 1,
       p\,\#_\varpi
     }( u_0, u_{k+1} ) \,
      {\smallprod_{ I \in \varpi \setminus \{I^\varpi_j\} }}
      t^{ -\iota^{\mathbf{0}_k}_I }
\\&\leq
    |T \vee 1|^{ \lfloor k/2 \rfloor \min\{1-\alpha,\nicefrac{1}{2}-\beta\} } \,
    |G_l|_{ \Cb{\#_\varpi}( H,V_{l,0} ) } \,
      L^{\mathbf{0}_k}_{ \varpi\setminus\{I^\varpi_j\}, p\,\#_\varpi }
      \,
     \tilde{d}_{
       \#_{I^\varpi_j} + 1,
       p\,\#_\varpi
     }( u_0, u_{k+1} )
\end{split}
\end{equation}
and 
\begin{equation}
\begin{split}
&
  \frac{1}{
    \prod^{k+1}_{ i=1 }
    \|u_i\|_H
  } \,
  \big\|
    G^{ ( \#_\varpi ) }_l( X_t^{ 0, u_0 } )
    \big(
      X_t^{ \#_{I^\varpi_1}, [ \theta^{k+1}_0(\mathbf{u}) ]_1^{ \varpi } }
      ,
      X_t^{ \#_{I^\varpi_2}, [ \theta^{k+1}_0(\mathbf{u}) ]_2^{ \varpi } }
      ,
      \dots
      ,
      X_t^{ \#_{I^\varpi_{j-1}}, [ \theta^{k+1}_0(\mathbf{u}) ]_{j-1}^{ \varpi } }
      ,
\\&\quad
      X_t^{ \#_{I^\varpi_j}+1, ( [ \theta^{k+1}_0(\mathbf{u}) ]_j^{ \varpi }, u_{k+1} ) }
      ,
      X_t^{ \#_{I^\varpi_{j+1}}, [ \theta^{k+1}_0( \mathbf{u} ) ]_{j+1}^{ \varpi } }
      ,
      X_t^{ \#_{I^\varpi_{j+2}}, [ \theta^{k+1}_0( \mathbf{u} ) ]_{j+2}^{ \varpi } }
      ,
      \dots
      ,
      X_t^{ \#_{I^\varpi_{m-1}}, [ \theta^{k+1}_0( \mathbf{u} ) ]_{ m-1 }^{ \varpi } }
      ,
\\&\quad
      X_t^{ \#_{I^\varpi_m}, [ \theta^{k+1}_1( \mathbf{u} ) ]_m^{ \varpi } }
      -
      X_t^{ \#_{I^\varpi_m}, [ \theta^{k+1}_0( \mathbf{u} ) ]_m^{ \varpi } }
      ,
      X_t^{ \#_{I^\varpi_{m+1}}, [ \theta^{k+1}_1( \mathbf{u} ) ]_{ m+1 }^{ \varpi } }
      ,
      X_t^{ \#_{I^\varpi_{m+2}}, [ \theta^{k+1}_1( \mathbf{u} ) ]_{ m+2 }^{ \varpi } },
      \dots 
      ,
\\&\quad
      X_t^{ \#_{I^\varpi_{\#_\varpi}}, [ \theta^{k+1}_1( \mathbf{u} ) ]_{ \#_\varpi }^{ \varpi } } 
    \big) 
  \big\|_{ \lpn{p}{\P}{ V_{l,0} } }
\\&\leq
  |G_l|_{ \Cb{\#_\varpi}(H, V_{l,0}) } \,
  \Bigg[
    \prod_{ i \in \{1,2,\ldots,m-1\} \setminus \{j\} }
    \frac{
      \|X^{\#_{I^\varpi_i}, [\theta^{k+1}_0(\mathbf{u})]^\varpi_i}_t\|_{ \lpn{p \#_\varpi}{\P}{H} }
    }{
      \prod^{ \#_{ I^\varpi_i } }_{ q=1 }
      \| u_{ I^\varpi_{ i,q } } \|_H
    }
  \Bigg]
\\&\quad\cdot
  \Bigg[
    \prod^{\#_\varpi}_{ i=m+1 }
    \frac{
      \|X^{\#_{I^\varpi_i}, [\theta^{k+1}_1(\mathbf{u})]^\varpi_i}_t\|_{ \lpn{p \#_\varpi}{\P}{H} }
    }{
      \prod^{ \#_{ I^\varpi_i } }_{ q=1 }
      \| u_{ I^\varpi_{ i,q } } \|_H
    }
  \Bigg]
  \Bigg[
    \frac{
      \|X^{\#_{I^\varpi_j}+1, ([\theta^{k+1}_0(\mathbf{u})]^\varpi_j,u_{k+1})}_t\|_{ \lpn{p \#_\varpi}{\P}{H} }
    }{
      \|u_{k+1}\|_H
      \prod^{ \#_{ I^\varpi_j } }_{ q=1 }
      \| u_{ I^\varpi_{ j,q } } \|_H
    }
  \Bigg]
\\&\quad\cdot
    \frac{
      \|      X_t^{ \#_{I^\varpi_m}, [ \theta^{k+1}_1( \mathbf{u} ) ]_m^{ \varpi } }
            -
            X_t^{ \#_{I^\varpi_m}, [ \theta^{k+1}_0( \mathbf{u} ) ]_m^{ \varpi } }\|_{ \lpn{p \#_\varpi}{\P}{H} }
    }{
      \prod^{ \#_{ I^\varpi_m } }_{ q=1 }
      \| u_{ I^\varpi_{ m,q } } \|_H
    }
\\&\leq
  |G_l|_{ \Cb{\#_\varpi}(H, V_{l,0}) } \,
      L^{\mathbf{0}_{k+1}}_{ \{ I^\varpi_j \cup \{k+1\} \}, p\,\#_\varpi } \,
      L^{\mathbf{0}_k}_{ \varpi\setminus\{ I^\varpi_j, \, I^\varpi_m \}, p\,\#_\varpi } \,
      d_{ \#_{I^\varpi_m}, p\,\#_\varpi }( u_0, u_0+u_{k+1} )
\\&\quad\cdot
  t^{ -\iota^{\mathbf{0}_{k+1}}_{ I^\varpi_j \cup \{k+1\} } }
  \smallprod_{ I \in \varpi \setminus\{ I^\varpi_j, \, I^\varpi_m \} }
  t^{ -\iota^{\mathbf{0}_k}_I }
\\&\leq
    |T \vee 1|^{ \lfloor k/2 \rfloor \min\{1-\alpha,\nicefrac{1}{2}-\beta\} } \,
  |G_l|_{ \Cb{\#_\varpi}(H, V_{l,0}) } \,
      L^{\mathbf{0}_{k+1}}_{ \{ I^\varpi_j \cup \{k+1\} \}, p\,\#_\varpi } \,
      L^{\mathbf{0}_k}_{ \varpi\setminus\{ I^\varpi_j, \, I^\varpi_m \}, p\,\#_\varpi }
\\&\quad\cdot
      d_{ \#_{I^\varpi_m}, p\,\#_\varpi }( u_0, u_0+u_{k+1} )
      .
\end{split}
\end{equation}
In addition, H\"{o}lder's inequality also shows that for all 
$ l\in\{0,1\} $,
$ p\in[2,\infty) $, 
$ \varpi \in \Pi_k $, 
$
  \mathbf{u} 
  = ( u_0, u_1, \dots, u_{k+1} ) 
  \in \times^{k+1}_{i=0} H^{[i]}
$,
$ t \in (0,T] $
it holds that 
\begin{equation}
\begin{split}
&
  \frac{1}{
    \prod^{k+1}_{ i=1 }
    \|u_i\|_H
  } \,
  \bigg\|
    \int^1_0
    \big[
    G_l^{ ( \#_\varpi + 1 ) }\big(
      X_t^{ 0, u_0 } 
      + 
      \rho [X^{0,u_0+u_{k+1}}_t-X^{0,u_0}_t]
    \big)
    -
    G_l^{ ( \#_\varpi + 1 ) }\big(
      X_t^{ 0,u_0 } 
    \big)
    \big]
    \big(
      X_t^{ \#_{I^\varpi_1}, [ \theta^{k+1}_1( \mathbf{u} ) ]_1^\varpi }
      ,
\\&\quad
      X_t^{ \#_{I^\varpi_2}, [ \theta^{k+1}_1( \mathbf{u} ) ]_2^\varpi }
      ,
      \dots
      ,
      X_t^{ \#_{I^\varpi_{\#_\varpi}}, [ \theta^{k+1}_1( \mathbf{u} ) ]_{ \#_\varpi }^{ \varpi } }
      ,
      X^{0,u_0+u_{k+1}}_t-X^{0,u_0}_t
    \big)
    \, \diffns\rho
  \bigg\|_{ \lpn{p}{\P}{ V_{l,0} } }
\\&\leq
  \int^1_0
  \big\|
   G_l^{ ( \#_\varpi + 1 ) }\big(
      X_t^{ 0,u_0 } 
      + 
      \rho [X^{0,u_0+u_{k+1}}_t-X^{0,u_0}_t]
    \big)  
    -
    G_l^{ ( \#_\varpi + 1 ) }\big(
      X_t^{ 0,u_0 } 
    \big)
  \big\|_{ \lpn{ p(\#_\varpi+2) }{\P}{ L^{ (\#_\varpi+1) }( H, V_{l,0} ) } }
  \,\diffns{\rho}
\\&\quad\cdot
  \Bigg[
    \prod^{\#_\varpi}_{ i=1 }
    \frac{
      \|X^{\#_{I^\varpi_i}, [\theta^{k+1}_1(\mathbf{u})]^\varpi_i}_t\|_{ \lpn{p (\#_\varpi+2)}{\P}{H} }
    }{
      \prod^{ \#_{ I^\varpi_i } }_{ q=1 }
      \| u_{ I^\varpi_{ i,q } } \|_H
    }
  \Bigg]
    \frac{
      \|X^{0,u_0+u_{k+1}}_t-X^{0,u_0}_t\|_{ \lpn{p (\#_\varpi+2)}{\P}{H} }
    }{
      \|u_{k+1}\|_H
    }
\\&\leq
  \int^1_0
  \big\|
   G_l^{ ( \#_\varpi + 1 ) }\big(
      X_t^{ 0,u_0 } 
      + 
      \rho [X^{0,u_0+u_{k+1}}_t-X^{0,u_0}_t]
    \big)  
    -
    G_l^{ ( \#_\varpi + 1 ) }\big(
      X_t^{ 0,u_0 } 
    \big)
  \big\|_{ \lpn{ p(\#_\varpi+2) }{\P}{ L^{ (\#_\varpi+1) }( H, V_{l,0} ) } }
  \,\diffns{\rho}
\\&\quad\cdot
  L^{\mathbf{0}_k}_{\varpi,p(\#_\varpi+2)} \,
  \tilde{L}_{ p(\#_\varpi+2) }
  \smallprod_{ I \in \varpi }
  t^{-\iota^{\mathbf{0}_k}_I}
\\&\leq
  \int^1_0
  \big\|
   G_l^{ ( \#_\varpi + 1 ) }\big(
      X_t^{ 0,u_0 } 
      + 
      \rho [X^{0,u_0+u_{k+1}}_t-X^{0,u_0}_t]
    \big)  
    -
    G_l^{ ( \#_\varpi + 1 ) }\big(
      X_t^{ 0,u_0 } 
    \big)
  \big\|_{ \lpn{ p(\#_\varpi+2) }{\P}{ L^{ (\#_\varpi+1) }( H, V_{l,0} ) } }
  \,\diffns{\rho}
\\&\quad\cdot
  | T \vee 1 |^{ \lfloor k/2 \rfloor \min\{1-\alpha,\nicefrac{1}{2}-\beta\} } \,
  L^{\mathbf{0}_k}_{\varpi,p(\#_\varpi+2)} \,
  \tilde{L}_{ p(\#_\varpi+2) }.
\end{split}
\end{equation}
Again H\"{o}lder's inequality assures that for all 
$ l\in\{0,1\} $,
$ p\in[2,\infty) $, 
$ \varpi \in \Pi_k $, 
$ j \in \{1,2,\ldots,\#_\varpi\} $, 
$
  \mathbf{u} 
  = ( u_0, u_1, \dots, u_{k+1} ) 
  \in \times^{k+1}_{i=0} H^{[i]}
$,
$ t \in (0,T] $ 
it holds that 
\begin{equation}
\begin{split}
&
  \frac{1}{
    \prod^{k+1}_{ i=1 }
    \|u_i\|_H
  } \,
  \big\|
    G_l^{ ( \#_\varpi + 1 ) }(
      X_t^{ 0, u_0 } 
    )\big(
    X_t^{ \#_{I^\varpi_1}, [ \theta^{k+1}_0( \mathbf{u} ) ]_1^\varpi }
    ,
    X_t^{ \#_{I^\varpi_2}, [ \theta^{k+1}_0( \mathbf{u} ) ]_2^\varpi }
    ,
    \ldots
    ,
\\&\quad
    X_t^{ \#_{I^\varpi_{\#_\varpi}}, [ \theta^{k+1}_0( \mathbf{u} ) ]_{\#_\varpi}^\varpi }
    ,
    X^{0,u_0+u_{k+1}}_t-X^{0,u_0}_t
    -
    X_t^{ 1,( u_0, u_{k+1} ) }
    \big)
  \big\|_{ \lpn{p}{\P}{ V_{l,0} } }
\\&\leq
  |G_l|_{ \Cb{\#_\varpi+1}( H, V_{l,0} ) } \,
  \Bigg[
    \prod^{\#_\varpi}_{ i=1 }
    \frac{
      \|X^{\#_{I^\varpi_i}, [\theta^{k+1}_0(\mathbf{u})]^\varpi_i}_t\|_{ \lpn{p (\#_\varpi+1)}{\P}{H} }
    }{
      \prod^{ \#_{ I^\varpi_i } }_{ q=1 }
      \| u_{ I^\varpi_{ i,q } } \|_H
    }
  \Bigg]
\\&\quad\cdot
    \frac{
      \|
    X^{0,u_0+u_{k+1}}_t-X^{0,u_0}_t
    -
    X_t^{ 1,( u_0, u_{k+1} ) }
      \|_{ \lpn{p (\#_\varpi+1)}{\P}{H} }
    }{
      \|u_{k+1}\|_H
    }
\\&\leq
  | T \vee 1 |^{ \lfloor k/2 \rfloor \min\{1-\alpha,\nicefrac{1}{2}-\beta\} } \,
  |G_l|_{ \Cb{ \#_\varpi+1 }( H, V_{l,0} ) } \,
    L^{\mathbf{0}_k}_{\varpi,p(\#_\varpi+1)}
    \,
     \tilde{d}_{ 1,p(\#_\varpi+1) }( u_0, u_{k+1} )
\end{split}
\end{equation}
and 
\begin{equation}
\label{eq:explain.end}
\begin{split}
&
  \frac{1}{
    \prod^{k+1}_{ i=1 }
    \|u_i\|_H
  } \,
  \big\|
    G_l^{ ( \#_\varpi + 1 ) }(
      X_t^{ 0,u_0 } 
    )\big(
    X_t^{ \#_{I^\varpi_1}, [ \theta^{k+1}_0( \mathbf{u} ) ]_1^\varpi }
    ,
    X_t^{ \#_{I^\varpi_2}, [ \theta^{k+1}_0( \mathbf{u} ) ]_2^\varpi }
    ,
    \ldots
    ,
    X_t^{ \#_{I^\varpi_{j-1}}, [ \theta^{k+1}_0( \mathbf{u} ) ]_{ j-1 }^{ \varpi } }
    ,
\\&\quad
    X_t^{ \#_{I^\varpi_j}, [ \theta^{k+1}_1( \mathbf{u} ) ]_j^\varpi }
    -
    X_t^{ \#_{I^\varpi_j}, [ \theta^{k+1}_0( \mathbf{u} ) ]_j^\varpi }
    ,
    X_t^{ \#_{I^\varpi_{j+1}}, [ \theta^{k+1}_1( \mathbf{u} ) ]_{ j+1 }^{ \varpi } }
    ,
    X_t^{ \#_{I^\varpi_{j+2}}, [ \theta^{k+1}_1( \mathbf{u} ) ]_{ j+2 }^{ \varpi } }
    ,
    \ldots
    ,
\\&\quad
    X_t^{ \#_{I^\varpi_{\#_\varpi}}, [ \theta^{k+1}_1( \mathbf{u} ) ]_{\#_\varpi}^\varpi }
    ,
    X^{0,u_0+u_{k+1}}_t-X^{0,u_0}_t
    \big)
  \big\|_{ \lpn{p}{\P}{ V_{l,0} } }
\\&\leq
  |G_l|_{ \Cb{\#_\varpi+1}( H, V_{l,0} ) } \,
  \Bigg[
    \prod^{j-1}_{ i=1 }
    \frac{
      \|X^{\#_{I^\varpi_i}, [\theta^{k+1}_0(\mathbf{u})]^\varpi_i}_t\|_{ \lpn{p (\#_\varpi+1)}{\P}{H} }
    }{
      \prod^{ \#_{ I^\varpi_i } }_{ q=1 }
      \| u_{ I^\varpi_{ i,q } } \|_H
    }
  \Bigg]
  \Bigg[
    \prod^{\#_\varpi}_{ i=j+1 }
    \frac{
      \|X^{\#_{I^\varpi_i}, [\theta^{k+1}_1(\mathbf{u})]^\varpi_i}_t\|_{ \lpn{p (\#_\varpi+1)}{\P}{H} }
    }{
      \prod^{ \#_{ I^\varpi_i } }_{ q=1 }
      \| u_{ I^\varpi_{ i,q } } \|_H
    }
  \Bigg]
\\&\quad\cdot
  \Bigg[
    \frac{
      \|    X_t^{ \#_{I^\varpi_j}, [ \theta^{k+1}_1( \mathbf{u} ) ]_j^\varpi }
          -
          X_t^{ \#_{I^\varpi_j}, [ \theta^{k+1}_0( \mathbf{u} ) ]_j^\varpi }\|_{ \lpn{p (\#_\varpi+1)}{\P}{H} }
    }{
      \prod^{ \#_{ I^\varpi_j } }_{ q=1 }
      \| u_{ I^\varpi_{ j,q } } \|_H
    }
  \Bigg]
    \frac{
      \|
    X^{0,u_0+u_{k+1}}_t-X^{0,u_0}_t
      \|_{ \lpn{p (\#_\varpi+1)}{\P}{H} }
    }{
      \|u_{k+1}\|_H
    }
\\&\leq
  | T \vee 1 |^{ \lfloor k/2 \rfloor \min\{1-\alpha,\nicefrac{1}{2}-\beta\} } \,
  |G_l|_{ \Cb{\#_\varpi+1}( H, V_{l,0} ) } \,
    L^{\mathbf{0}_k}_{\varpi\setminus\{I^\varpi_j\}, p(\#_\varpi+1)} \,
    \tilde{L}_{ p(\#_\varpi+1) }\,
\\&\quad\cdot
      d_{ \#_{I^\varpi_j}, p(\#_\varpi+1) }( u_0, u_0+u_{k+1} ).
\end{split}
\end{equation}
Combining~\eqref{eq:explain.start}--\eqref{eq:explain.end} yields that for all 
$ l\in\{0,1\} $, 
$ p\in[2,\infty) $, 
$
  \mathbf{u} 
  = ( u_0, u_1, \dots, u_{k+1} ) 
  \in \times^{k+1}_{i=0} H^{[i]}
$,
$ t \in (0,T] $
it holds that 
\begin{equation}
\begin{split}
&
  \frac{
    \|
    \mathbf{\bar{G}}^{\mathbf{u}}_{ k+1,l }( t, X^{k,\theta^{k+1}_1(\mathbf{u})}_t - X^{k,\theta^{k+1}_0(\mathbf{u})}_t )
    -
    \mathbf{G}^{\mathbf{u}}_{ k+1,l }( t, X^{k,\theta^{k+1}_1(\mathbf{u})}_t - X^{k,\theta^{k+1}_0(\mathbf{u})}_t )
    \|_{ \lpn{p}{\P}{V_{l,0}} }
  }{
    \prod^{k+1}_{i=1}
    \|u_i\|_H
  }
\leq
  | T \vee 1 |^k \,
\\&\cdot
    \Bigg(
    \sum_{ \varpi\in\Pi^*_k }
    |G_l|_{ \Cb{\#_\varpi}( H,V_{l,0} ) }
    \sum_{ I\in\varpi }
    \Bigg[
      L^{\mathbf{0}_k}_{ \varpi\setminus\{I\}, p\,\#_\varpi }
      \,
     \tilde{d}_{
       \#_I + 1,
       p\,\#_\varpi
     }( u_0, u_{k+1} )
\\&+
      L^{\mathbf{0}_{k+1}}_{ \{ I \cup \{k+1\} \}, p\,\#_\varpi }
      \sum_{\substack{ J\in\varpi\colon\min(J) > \min(I) }}
      L^{\mathbf{0}_k}_{ \varpi\setminus\{ I,J \}, p\,\#_\varpi } \,
      d_{ \#_J, p\,\#_\varpi }( u_0, u_0+u_{k+1} )
    \Bigg]
\\&+
  \sum_{\varpi\in\Pi_k}
  \Bigg[
  L^{\mathbf{0}_k}_{\varpi,p(\#_\varpi+2)} \,
  \tilde{L}_{ p(\#_\varpi+2) }
\\&\cdot
  \int^1_0
  \big\|
   G_l^{ ( \#_\varpi + 1 ) }\big(
      X_t^{ 0,u_0 } 
      + 
      \rho [X^{0,u_0+u_{k+1}}_t-X^{0,u_0}_t]
    \big)  
    -
    G_l^{ ( \#_\varpi + 1 ) }\big(
      X_t^{ 0,u_0 } 
    \big)
  \big\|_{ \lpn{ p(\#_\varpi+2) }{\P}{ L^{ (\#_\varpi+1) }( H, V_{l,0} ) } }
  \,\diffns{\rho}
\\&+
  |G_l|_{ \Cb{ \#_\varpi+1 }( H, V_{l,0} ) }
  \Bigg(
    L^{\mathbf{0}_k}_{\varpi,p(\#_\varpi+1)}
    \,
     \tilde{d}_{ 1,p(\#_\varpi+1) }( u_0, u_{k+1} )
\\&+
    \sum_{ I\in\varpi }
    L^{\mathbf{0}_k}_{\varpi\setminus\{I\}, p(\#_\varpi+1)} \,
    \tilde{L}_{ p(\#_\varpi+1) }\,
    \,
      d_{ \#_I, p(\#_\varpi+1) }( u_0, u_0+u_{k+1} )
  \Bigg)
  \Bigg]
  \Bigg)
  .
\end{split}
\end{equation}
This and Minkowski's inequality imply that for all 
$ l\in\{0,1\} $, 
$ p\in[2,\infty) $, 
$
  \mathbf{u} 
  = ( u_0, u_1, \dots, u_{k+1} ) 
  \in \times^{k+1}_{i=0} H^{[i]}
$
it holds that 
\begin{equation}
\begin{split}
&
  \left[
  \int^T_0
  \left(
  \frac{
    \|
    \mathbf{\bar{G}}^{\mathbf{u}}_{ k+1,l }( t, X^{k,\theta^{k+1}_1(\mathbf{u})}_t - X^{k,\theta^{k+1}_0(\mathbf{u})}_t )
    -
    \mathbf{G}^{\mathbf{u}}_{ k+1,l }( t, X^{k,\theta^{k+1}_1(\mathbf{u})}_t - X^{k,\theta^{k+1}_0(\mathbf{u})}_t )
    \|_{ \lpn{p}{\P}{V_{l,0}} }
  }{
    \prod^{k+1}_{i=1}
    \|u_i\|_H
  }
  \right)^{\!(l+1)}
  dt
  \right]^{\!\nicefrac{1}{(l+1)}}
\\&\leq
  | T \vee 1 |^k \,
  \Bigg[
  \int^T_0
    \Bigg(
    \sum_{ \varpi\in\Pi^*_k }
    |G_l|_{ \Cb{\#_\varpi}( H,V_{l,0} ) }
    \sum_{ I\in\varpi }
    \Bigg[
      L^{\mathbf{0}_k}_{ \varpi\setminus\{I\}, p\,\#_\varpi }
      \,
     \tilde{d}_{
       \#_I + 1,
       p\,\#_\varpi
     }( u_0, u_{k+1} )
\\&\quad+
      L^{\mathbf{0}_{k+1}}_{ \{ I \cup \{k+1\} \}, p\,\#_\varpi }
      \sum_{\substack{ J\in\varpi\colon\min(J) > \min(I) }}
      L^{\mathbf{0}_k}_{ \varpi\setminus\{ I,J \}, p\,\#_\varpi } \,
      d_{ \#_J, p\,\#_\varpi }( u_0, u_0+u_{k+1} )
    \Bigg]
\\&\quad+
  \sum_{\varpi\in\Pi_k}
  \Bigg[
  L^{\mathbf{0}_k}_{\varpi,p(\#_\varpi+2)} \,
  \tilde{L}_{ p(\#_\varpi+2) }
\\&\quad\cdot
  {\int^1_0}
  \big\|
   G_l^{ ( \#_\varpi + 1 ) }\big(
      X_t^{ 0,u_0 } 
      + 
      \rho [X^{0,u_0+u_{k+1}}_t-X^{0,u_0}_t]
    \big)  
    -
    G_l^{ ( \#_\varpi + 1 ) }\big(
      X_t^{ 0,u_0 } 
    \big)
  \big\|_{ \lpn{ p(\#_\varpi+2) }{\P}{ L^{ (\#_\varpi+1) }( H, V_{l,0} ) } }
  \,\diffns{\rho}
\\&\quad+
  |G_l|_{ \Cb{ \#_\varpi+1 }( H, V_{l,0} ) }
  \Bigg(
    L^{\mathbf{0}_k}_{\varpi,p(\#_\varpi+1)}
    \,
     \tilde{d}_{ 1,p(\#_\varpi+1) }( u_0, u_{k+1} )
\\&\quad+
    \sum_{ I\in\varpi }
    L^{\mathbf{0}_k}_{\varpi\setminus\{I\}, p(\#_\varpi+1)} \,
    \tilde{L}_{ p(\#_\varpi+1) }\,
    \,
      d_{ \#_I, p(\#_\varpi+1) }( u_0, u_0+u_{k+1} )
  \Bigg)
  \Bigg]
  \Bigg)^{(l+1)}
  \, dt
  \Bigg]^{\nicefrac{1}{(l+1)}}
\\&\leq
  | T \vee 1 |^k \,
  \Bigg\{
  \Bigg[
  \int^T_0
    \Bigg(
    \sum_{ \varpi\in\Pi^*_k }
    |G_l|_{ \Cb{\#_\varpi}( H,V_{l,0} ) }
    \sum_{ I\in\varpi }
    \Bigg[
      L^{\mathbf{0}_k}_{ \varpi\setminus\{I\}, p\,\#_\varpi }
      \,
     \tilde{d}_{
       \#_I + 1,
       p\,\#_\varpi
     }( u_0, u_{k+1} )
\\&\quad+
      L^{\mathbf{0}_{k+1}}_{ \{ I \cup \{k+1\} \}, p\,\#_\varpi }
      \sum_{\substack{ J\in\varpi\colon\min(J) > \min(I) }}
      L^{\mathbf{0}_k}_{ \varpi\setminus\{ I,J \}, p\,\#_\varpi } \,
      d_{ \#_J, p\,\#_\varpi }( u_0, u_0+u_{k+1} )
    \Bigg]\Bigg)^{(l+1)}
    \, dt
    \Bigg]^{\nicefrac{1}{(l+1)}}
\\&\quad+
  {\sum_{\varpi\in\Pi_k}}
  \Bigg\{
  \bigg[
  \int^T_0
  \bigg(
  L^{\mathbf{0}_k}_{\varpi,p(\#_\varpi+2)} \,
  \tilde{L}_{ p(\#_\varpi+2) }
  {\int^1_0}
  \big\|
   G_l^{ ( \#_\varpi + 1 ) }\big(
      X_t^{ 0,u_0 } 
      + 
      \rho [X^{0,u_0+u_{k+1}}_t-X^{0,u_0}_t]
    \big)  
\\&\quad
    -
    G_l^{ ( \#_\varpi + 1 ) }\big(
      X_t^{ 0,u_0 } 
    \big)
  \big\|_{ \lpn{ p(\#_\varpi+2) }{\P}{ L^{ (\#_\varpi+1) }( H, V_{l,0} ) } }
  \,\diffns{\rho}
  \bigg)^{(l+1)}
  \, dt
  \bigg]^{\nicefrac{1}{(l+1)}}
\\&\quad+
  \Bigg[
  \int^T_0
  \Bigg(
  |G_l|_{ \Cb{ \#_\varpi+1 }( H, V_{l,0} ) }
  \Bigg(
    L^{\mathbf{0}_k}_{\varpi,p(\#_\varpi+1)}
    \,
     \tilde{d}_{ 1,p(\#_\varpi+1) }( u_0, u_{k+1} )
\\&\quad+
    \sum_{ I\in\varpi }
    L^{\mathbf{0}_k}_{\varpi\setminus\{I\}, p(\#_\varpi+1)} \,
    \tilde{L}_{ p(\#_\varpi+1) }\,
    \,
      d_{ \#_I, p(\#_\varpi+1) }( u_0, u_0+u_{k+1} )
  \Bigg)
  \Bigg)^{(l+1)}
  \, dt
  \Bigg]^{\nicefrac{1}{(l+1)}}
  \Bigg\}
  \Bigg\}
  .
\end{split}
\end{equation}
Jensen's inequality hence shows that for all 
$ l\in\{0,1\} $, 
$ p\in[2,\infty) $, 
$
  \mathbf{u} 
  = ( u_0, u_1, \dots, u_{k+1} ) 
  \in \times^{k+1}_{i=0} H^{[i]}
$
it holds that 
\begin{equation}
\label{eq:kth.derivative.nonlinear.difference}
\begin{split}
&
  \left[
  \int^T_0
  \left(
  \frac{
    \|
    \mathbf{\bar{G}}^{\mathbf{u}}_{ k+1,l }( t, X^{k,\theta^{k+1}_1(\mathbf{u})}_t - X^{k,\theta^{k+1}_0(\mathbf{u})}_t )
    -
    \mathbf{G}^{\mathbf{u}}_{ k+1,l }( t, X^{k,\theta^{k+1}_1(\mathbf{u})}_t - X^{k,\theta^{k+1}_0(\mathbf{u})}_t )
    \|_{ \lpn{p}{\P}{V_{l,0}} }
  }{
    \prod^{k+1}_{i=1}
    \|u_i\|_H
  }
  \right)^{\!(l+1)}
  dt
  \right]^{\!\nicefrac{1}{(l+1)}}
\\&\leq
  | T \vee 1 |^k \,
  \Bigg\{
    \sum_{ \varpi\in\Pi^*_k }
    T^{\nicefrac{1}{(l+1)}} \,
    |G_l|_{ \Cb{\#_\varpi}( H,V_{l,0} ) }
    \sum_{ I\in\varpi }
    \Bigg[
      L^{\mathbf{0}_k}_{ \varpi\setminus\{I\}, p\,\#_\varpi }
      \,
     \tilde{d}_{
       \#_I + 1,
       p\,\#_\varpi
     }( u_0, u_{k+1} )
\\&\quad+
      L^{\mathbf{0}_{k+1}}_{ \{ I \cup \{k+1\} \}, p\,\#_\varpi }
      \sum_{\substack{ J\in\varpi\colon\min(J) > \min(I) }}
      L^{\mathbf{0}_k}_{ \varpi\setminus\{ I,J \}, p\,\#_\varpi } \,
      d_{ \#_J, p\,\#_\varpi }( u_0, u_0+u_{k+1} )
    \Bigg]
\\&\quad+
  {\sum_{\varpi\in\Pi_k}}
  \Bigg\{
  L^{\mathbf{0}_k}_{\varpi,p(\#_\varpi+2)} \,
  \tilde{L}_{ p(\#_\varpi+2) }
  \bigg[
  \int^T_0
  \bigg(
  {\int^1_0}
  \big\|
   G_l^{ ( \#_\varpi + 1 ) }\big(
      X_t^{ 0,u_0 } 
      + 
      \rho [X^{0,u_0+u_{k+1}}_t-X^{0,u_0}_t]
    \big)  
\\&\quad-
    G_l^{ ( \#_\varpi + 1 ) }\big(
      X_t^{ 0,u_0 } 
    \big)
  \big\|_{ \lpn{ p(\#_\varpi+2) }{\P}{ L^{ (\#_\varpi+1) }( H, V_{l,0} ) } }
  \,\diffns{\rho}
  \bigg)^{(l+1)}
  \, dt
  \bigg]^{\nicefrac{1}{(l+1)}}
\\&\quad+
  T^{\nicefrac{1}{(l+1)}} \,
  |G_l|_{ \Cb{ \#_\varpi+1 }( H, V_{l,0} ) }
  \Bigg(
    L^{\mathbf{0}_k}_{\varpi,p(\#_\varpi+1)}
    \,
     \tilde{d}_{ 1,p(\#_\varpi+1) }( u_0, u_{k+1} )
\\&\quad+
    \sum_{ I\in\varpi }
    L^{\mathbf{0}_k}_{\varpi\setminus\{I\}, p(\#_\varpi+1)} \,
    \tilde{L}_{ p(\#_\varpi+1) }\,
    \,
      d_{ \#_I, p(\#_\varpi+1) }( u_0, u_0+u_{k+1} )
  \Bigg)
  \Bigg\}
  \Bigg\}
\\&\leq
  | T \vee 1 |^k \,
  \Bigg\{
    \sum_{ \varpi\in\Pi^*_k }
    T^{\nicefrac{1}{(l+1)}} \,
    |G_l|_{ \Cb{\#_\varpi}( H,V_{l,0} ) }
    \sum_{ I\in\varpi }
    \Bigg[
      L^{\mathbf{0}_k}_{ \varpi\setminus\{I\}, p\,\#_\varpi }
      \,
     \tilde{d}_{
       \#_I + 1,
       p\,\#_\varpi
     }( u_0, u_{k+1} )
\\&\quad+
      L^{\mathbf{0}_{k+1}}_{ \{ I \cup \{k+1\} \}, p\,\#_\varpi }
      \sum_{\substack{ J\in\varpi\colon\min(J) > \min(I) }}
      L^{\mathbf{0}_k}_{ \varpi\setminus\{ I,J \}, p\,\#_\varpi } \,
      d_{ \#_J, p\,\#_\varpi }( u_0, u_0+u_{k+1} )
    \Bigg]
\\&\quad+
  {\sum_{\varpi\in\Pi_k}}
  \Bigg\{
  L^{\mathbf{0}_k}_{\varpi,p(\#_\varpi+2)} \,
  \tilde{L}_{ p(\#_\varpi+2) }
  \bigg[
  \int^T_0
  {\int^1_0}
  \big\|
   G_l^{ ( \#_\varpi + 1 ) }\big(
      X_t^{ 0,u_0 } 
      + 
      \rho [X^{0,u_0+u_{k+1}}_t-X^{0,u_0}_t]
    \big)  
\\&\quad-
    G_l^{ ( \#_\varpi + 1 ) }\big(
      X_t^{ 0,u_0 } 
    \big)
  \big\|^{(l+1)}_{ \lpn{ p(\#_\varpi+2) }{\P}{ L^{ (\#_\varpi+1) }( H, V_{l,0} ) } }
  \,\diffns{\rho}
  \, dt
  \bigg]^{\nicefrac{1}{(l+1)}}
\\&\quad+
  T^{\nicefrac{1}{(l+1)}} \,
  |G_l|_{ \Cb{ \#_\varpi+1 }( H, V_{l,0} ) }
  \Bigg(
    L^{\mathbf{0}_k}_{\varpi,p(\#_\varpi+1)}
    \,
     \tilde{d}_{ 1,p(\#_\varpi+1) }( u_0, u_{k+1} )
\\&\quad+
    \sum_{ I\in\varpi }
    L^{\mathbf{0}_k}_{\varpi\setminus\{I\}, p(\#_\varpi+1)} \,
    \tilde{L}_{ p(\#_\varpi+1) }\,
    \,
      d_{ \#_I, p(\#_\varpi+1) }( u_0, u_0+u_{k+1} )
  \Bigg)
  \Bigg\}
  \Bigg\}
  .
\end{split}
\end{equation}
Combining~\eqref{calc2} with~\eqref{eq:kth.derivative.nonlinear.difference}
ensures that for all 
$ p\in[2,\infty) $, 
$
  x \in H
$, 
$ u_{k+1} \in \nzspace{H} $
it holds that 
\begin{equation}
\label{eq:kth.derivative.convergence}
\begin{split}
&
  \sup_{
  \mathbf{u}=
  ( u_1, u_2, \dots, u_k ) 
  \in (\nzspace{H})^k
  }
  \sup_{t\in[0,T]}
  \frac{
    \big\|
      X^{k,(x+u_{k+1},\mathbf{u})}_t - X^{k,(x,\mathbf{u})}_t
      - X_t^{ k+1,( x, \mathbf{u}, u_{k+1} ) }
    \big\|_{
      \mathcal{L}^p( \P ; H )
    }
  }{
    \prod^{k+1}_{i=1}
    \|u_i\|_H
  }
\\&\leq
  |T\vee1|^k \,
  \chi^{ 0, T }_{ A, \eta }\,
  \Theta_{A,\eta,p,T}^{0,0,0}
  \big(
    |F|_{\Cb{1}(H,H)},
    |B|_{\Cb{1}(H,HS(U,H))}
  \big)
\\&\quad\cdot
  \Bigg\{
    \sum_{ \varpi \in \Pi^*_k }
      \Big[
        T \, |F|_{  \Cb{ \#_\varpi }( H, H )  }\,
+
        \sqrt{
          \tfrac{ p \, (p-1) }{2} \,
          T
        } \,
        |B|_{  \Cb{ \#_\varpi }( H, HS( U, H ) )  }\,       
      \Big]
\\&\quad\cdot
    \sum_{ I\in\varpi }
    \Bigg[
      L^{\mathbf{0}_k}_{ \varpi\setminus\{I\}, p\,\#_\varpi }
      \,
     \tilde{d}_{ \#_I + 1, p\,\#_\varpi }( x, u_{k+1} )
\\&\quad+
      L^{\mathbf{0}_{k+1}}_{ \{ I \cup \{k+1\} \}, p\,\#_\varpi }
      \sum_{\substack{ J\in\varpi\colon\min(J) > \min(I) }}
      L^{\mathbf{0}_k}_{ \varpi\setminus\{ I,J \}, p\,\#_\varpi } \,
      d_{\#_J,p\,\#_\varpi}( x, x+u_{k+1} )
    \Bigg]
\\&\quad+
  \sum_{\varpi\in\Pi_k}
  \Bigg[
  L^{\mathbf{0}_k}_{\varpi, p(\#_\varpi+2) }\,
  \tilde{L}_{ p(\#_\varpi+2) } \,
  \bigg(
  \int^T_0
  \int^1_0
    \|
      F^{ ( \#_\varpi + 1 ) }(
        X_s^{0,x} 
        + 
        \rho [ X^{0,x+u_{k+1}}_s - X^{0,x}_s ]
      )
\\&\quad-
      F^{ ( \#_\varpi + 1 ) }(
        X_s^{0,x} 
      )
    \|_{ \lpn{ p(\#_\varpi+2) }{\P}{ L^{ (\#_\varpi+1) }( H, H ) } }
  \,\diffns{\rho}
  \,\diffns{s}
\\&\quad+
  \bigg[
  \tfrac{ p\,(p-1) }{ 2 }
  \int^T_0
  \int^1_0
    \|
      B^{ ( \#_\varpi + 1 ) }(
        X_s^{0,x} 
        + 
        \rho [ X^{0,x+u_{k+1}}_s - X^{0,x}_s ]
      )
\\&\quad-
      B^{ ( \#_\varpi + 1 ) }(
        X_s^{0,x}
      )
    \|^2_{ \lpn{ p(\#_\varpi+2) }{\P}{ L^{ (\#_\varpi+1) }( H, HS(U,H) ) } }
  \,\diffns{\rho}
  \,\diffns{s}
  \bigg]^{\nicefrac{1}{2}}
  \bigg)
\\&\quad+
      \Big[
        T \, |F|_{  \Cb{ \#_\varpi+1 }( H, H )  }\,
+
       \sqrt{
          \tfrac{ p \, (p-1) }{2} \,
          T
        } \
        |B|_{  \Cb{ \#_\varpi+1 }( H, HS( U, H ) )  }\,
      \Big]
\\&\quad\cdot
  \Bigg(
    L^{\mathbf{0}_k}_{\varpi,p(\#_\varpi+1)}
    \,
     \tilde{d}_{ 1,p(\#_\varpi+1) }( x, u_{k+1} )
\\&\quad+
    \sum_{ I\in\varpi }
    L^{\mathbf{0}_k}_{\varpi\setminus\{I\},p(\#_\varpi+1)} \,
    \tilde{L}_{ p(\#_\varpi+1) } \,
      d_{\#_I,p(\#_\varpi+1)}( x, x+u_{k+1} )
  \Bigg)
  \Bigg]
  \Bigg\}
  .
\end{split}
\end{equation}
This, \eqref{eq:1st.derivative.integral.limit}, and~\eqref{eq:distance.convergence} 
establish item~\eqref{item:lem_derivative:frechet} in the case $k+1$. 
Induction thus completes the proof of item~\eqref{item:lem_derivative:frechet}. 

Combining item~\eqref{item:thm_derivative}, 
item~\eqref{item:thm_derivative.continuous}, and item~\eqref{item:lem_derivative:frechet}
with item~\eqref{item:lem_derivative:a_priori} establishes 
item~\eqref{item:lem_derivative:smoothness} 
and item~\eqref{item:lem_derivative:representation}. 
Next we note that~\eqref{eq:k.linear} and item~\eqref{item:lem_derivative:a_priori} ensure that for all 
$ k \in \{1,2,\ldots,n\} $, 
$ p \in [2,\infty) $, 
$ x \in H $, 
$ t \in [0,T] $ 
it holds that 
\begin{equation}
\label{eq:time.derivative.linearity}
  \big(
    H^k \ni \mathbf{u} \mapsto
    [X^{k,(x,\mathbf{u})}_t]_{\P,\mathcal{B}(H)}
    \in \lpnb{p}{\P}{H}
  \big) \in L^{(k)}(H,\lpnb{p}{\P}{H}).
\end{equation}
In addition, item~\eqref{item:thm_derivative.continuous} ensures that for all 
$ k \in \{1,2,\ldots,n\} $, 
$ p \in [2,\infty) $, 
$ t \in [0,T] $
it holds that 
\begin{multline}
\label{eq:time.derivative.continuity}
  \big(
  H \ni x \mapsto
  \big[
    H^k \ni \mathbf{u} \mapsto
    [X^{k,(x,\mathbf{u})}_t]_{\P,\mathcal{B}(H)}
    \in \lpnb{p}{\P}{H}
  \big] \in L^{(k)}(H,\lpnb{p}{\P}{H})
  \big) 
\\
  \in \mathcal{C}(H,L^{(k)}(H,\lpnb{p}{\P}{H})).
\end{multline}
Combining~\eqref{eq:time.derivative.linearity} and~\eqref{eq:time.derivative.continuity} with item~\eqref{item:lem_derivative:a_priori} and  item~\eqref{item:lem_derivative:frechet} proves item~\eqref{item:time.derivative.smoothness} and item~\eqref{item:time.derivative.representation}.
The proof of Theorem~\ref{lem:derivative_processes} is thus completed.

\end{proof}

\section*{Acknowledgements}

Stig Larsson and Christoph Schwab are gratefully acknowledged for a number of useful comments.
This project has been supported through the SNSF-Research project 200021\_156603 
``Numerical approximations of nonlinear stochastic ordinary and partial differential equations" 
and the ETH Research Grant ETH-47 15-2 
``Mild stochastic calculus and numerical approximations for nonlinear stochastic evolution equations with L\'{e}vy noise".

\bibliographystyle{acm}
\bibliography{Bib/bibfile}
\end{document}